\documentclass[12pt,a4paper]{report}

\usepackage[utf8]{inputenc}
\usepackage[T1]{fontenc}
\usepackage{lmodern}

\usepackage[english,main=spanish]{babel}

\usepackage{amsmath}
\usepackage{amsfonts}
\usepackage{amssymb}
\usepackage{amsthm}

\usepackage{hyperref}

\theoremstyle{plain}
	\newtheorem{theorem}{Theorem}[chapter]
	\newtheorem{teorema}{Teorema}[chapter]
	\newtheorem{lemma}[theorem]{Lemma}
	\newtheorem{lema}[teorema]{Lema}
	\newtheorem{proposition}[theorem]{Proposition}
	\newtheorem{proposicion}[teorema]{Proposición}
\theoremstyle{definition}
	\newtheorem{objective}[theorem]{Objective}
	\newtheorem{objetivo}[teorema]{Objetivo}
	\newtheorem{example}[theorem]{Example}
	\newtheorem{ejemplo}[teorema]{Ejemplo}
	\newtheorem{remark}[theorem]{Remark}
	\newtheorem{observacion}[teorema]{Observación}

\DeclareMathOperator{\rad}{rad}
\DeclareMathOperator{\Arf}{Arf}
\DeclareMathOperator{\supp}{supp}
\DeclareMathOperator{\Skew}{Skew}
\DeclareMathOperator{\Cl}{Cl}
\DeclareMathOperator{\sign}{sign}

\begin{document}

\selectlanguage{english}
\begin{large}
	\noindent
	Government of Aragon
	\hfill
	European Union
	
	\bigskip
	\begin{center}
		\textit{Building Europe from Aragon}
	\end{center}
	
	\vspace{\stretch{2}}
	\begin{center}
		\begin{huge}
			\phantomsection
			\textbf{Gradings~on simple~real~Lie~algebras}
			\label{title}
			
		\end{huge}
		
		\bigskip \bigskip \bigskip
		\begin{Large}
			Adri\'an Rodrigo-Escudero
			
		\end{Large}
		
		\bigskip \bigskip
		14 june 2018
	\end{center}
	
	\vspace{\stretch{3}}
	\noindent
	Doctoral thesis in mathematics at the University of Zaragoza
	\\
	Supervised by Alberto Elduque
	
\end{large}
\newpage

\selectlanguage{english}
\section*{Abstract in english}
\addcontentsline{toc}{section}{Abstract in english}

This work is a doctoral thesis in mathematics
by compendium of four articles.
Here we explain,
using a language as simple as possible,
the results achieved in those articles.
The general objective is
the classification of
gradings on simple real Lie algebras.
The text is written in both english and spanish.
This document can be downloaded,
with the identifier that appears
on the left margin of page \pageref{title},
from the following web address:
\url{https://arxiv.org}

\selectlanguage{spanish}
\section*{Resumen en español}
\addcontentsline{toc}{section}{Resumen en español}

Esta obra es una tesis doctoral en matemáticas
por compendio de cuatro artículos.
Aquí explicamos,
utilizando un lenguaje lo más sencillo posible,
los resultados alcanzados en esos artículos.
El objetivo general es
la clasificación de
las graduaciones en las álgebras de Lie reales simples.
El texto está escrito tanto en inglés como en español.
Este documento se puede descargar,
con el identificador que aparece
en el margen izquierdo de la página \pageref{title},
desde la siguiente dirección web:
\url{https://arxiv.org}

\vfill

\selectlanguage{english}
\section*{Information}
\addcontentsline{toc}{section}{Information}

2010 Mathematics Subject Classification:
primary 17B70, 16W50;
secondary 16W10, 16S35, 15A66.

Keywords:
graded algebra;
simple real Lie algebra;
division grading;
classification;
Clifford algebra.

\selectlanguage{spanish}

Adrián Rodrigo-Escudero;
Departamento de Matemáticas e
Instituto Universitario de Matemáticas y Aplicaciones;
Universidad de Zaragoza;
50009 Zaragoza;
Spain.

\texttt{adrian.rodrigo.escudero@gmail.com}

\url{https://arxiv.org/a/0000-0001-9408-6535}

\url{https://orcid.org/0000-0001-9408-6535}

\selectlanguage{english}
\chapter*{Articles}
\addcontentsline{toc}{chapter}{Articles}
\label{ch:articles}

\begin{enumerate}

\item[\cite{Rod2016}]
A. Rodrigo-Escudero,
Classification of division gradings on fi\-nite-di\-men\-sional simple real algebras,
Linear Algebra Appl. {\bf 493} (2016), 164--182.
\\
\url{https://arxiv.org/abs/1506.01552}
\\
\url{https://doi.org/10.1016/j.laa.2015.11.025}

\item[\cite{BKR2018a}]
Y. Bahturin, M. Kochetov\ and\ A. Rodrigo-Escudero,
Classification of involutions on graded-division simple real algebras,
Linear Algebra Appl. {\bf 546} (2018), 1--36.
\\
\url{https://arxiv.org/abs/1707.05526}
\\
\url{https://doi.org/10.1016/j.laa.2018.01.040}

\item[\cite{BKR2018b}]
Y. Bahturin, M. Kochetov\ and\ A. Rodrigo-Escudero,
Gradings on classical central simple real Lie algebras,
J. Algebra {\bf 506} (2018), 1--42.
\\
\url{https://arxiv.org/abs/1707.07909}
\\
\url{https://doi.org/10.1016/j.jalgebra.2018.02.036}

\item[\cite{ER2018}]
A. Elduque\ and\ A. Rodrigo-Escudero,
Clifford algebras as twisted gro\-up algebras and the Arf invariant,
Adv. Appl. Clifford Algebr. {\bf 28} (2018), no.~2, 28:41.
\\
\url{https://arxiv.org/abs/1801.07002}
\\
\url{https://doi.org/10.1007/s00006-018-0862-y}

\end{enumerate}

\selectlanguage{english}
\clearpage
\phantomsection
\addcontentsline{toc}{chapter}{Contents}
\tableofcontents

\selectlanguage{english}
\chapter*{Introduction}
\addcontentsline{toc}{chapter}{Introduction}

A lot has changed since my supervisor,
Alberto Elduque,
defended his thesis in 1984.
He did not publish any of his results
while he was a doctoral student,
since it was not clear whether
a thesis that had previously appeared in journals
was an original work or not.
So the first thing that he had to do
after the defense
was to adapt the chapters of his thesis to the article format.

Today the situation is totally different,
and publishing is fundamental in a doctorate.
The research that we have carried out
since I started the doctorate in september 2013
has already appeared in the four articles
listed on page \pageref{ch:articles}.
Now what is not clear to us is
what the text of a thesis should contain.
Naturally when we wrote the articles
we did it in the best way that we knew,
without omitting any detail and including
key points, motivation, historical references,
necessary background, et cetera.
We hope that the introduction that appears in each article serves as a guide
to know what its objectives are
and where to find the main theorems.
My intention is not to repeat in this thesis
what we have already exposed in other places,
and in order to do so my friend Eva gave me
an alternative idea of what I could write.

The objective of this text is to explain,
using a language as simple as possible,
the results achieved during these five years.
I think that at least chapter \ref{ch:art_1},
corresponding to the first article \cite{Rod2016},
can be understood even by someone
who has not studied a degree in mathematics.
This is something exceptional in the mathematical research,
in fact both the second article \cite{BKR2018a}
and specially the third \cite{BKR2018b}
are much more technical.
On the other hand chapter \ref{ch:art_4},
corresponding to the fourth article \cite{ER2018},
addresses a mathematical audience,
but not necessarily specialized in algebra.

The general purpose of my thesis is to classify
gradings on simple real Lie algebras.
A grading is simply
a decomposition of an algebra
compatible with its operations
of sum, product and multiplication by scalars.
For a precise statement of the results achieved,
we refer to either the objectives \ref{obj_en:art_1},
\ref{obj_en:art_2}, \ref{obj_en:art_3} and \ref{obj_en:art_4},
or,
for a more detailed summary,
to the sections that contain them
(sections \ref{sect:art_1_mot_obj}, \ref{sect:art_2_obj},
\ref{sect:art_3_obj} and \ref{sect:art_4_intr}).
Besides,
section \ref{sect:state_art} complements this summary.
In the first and second article
we classify gradings on associative algebras,
with the intention of applying these results
to Lie algebras in the third article,
as we explain in section \ref{sect:alg_inv}.
In any case remark that each of these classifications
is of independent interest.
Finally in the fourth article we change to a more relaxed topic,
since we use results obtained in the previous articles
to give alternative proofs to already known theorems.

\selectlanguage{english}
\chapter{Background}

\section{Algebras and groups}

An algebra is a set endowed with
a structure of vector space and a bilinear product.
For example,
the set $M_2(\mathbb{R})$ of
the square matrices of size $ 2 \times 2 $
whose entries are real numbers,
with the usual operations of sum and product of matrices
and of multiplication of a scalar by a matrix,
is an algebra,
which also satisfies the following properties.
It is associative,
since $ (XY)Z = X(YZ) $ for all matrices
$X,Y,Z$ in $M_2(\mathbb{R})$;
it is real,
because in this case the scalar field
is the real numbers $\mathbb{R}$;
it has dimension $4$;
and it is unital,
since there exists a neutral element for the product,
the identity matrix $I$.

Another example is the algebra of quaternions $\mathbb{H}$
\cite{Ham1844},
let us define it step by step.
First we fix a real vector space $\mathbb{H}$ of dimension $4$,
and we call $ \{ 1,i,j,k \} $ the elements of one of its basis.
This can be expressed by the following formula:
\begin{equation}\label{eq:quaternions}
\mathbb{H} = \mathbb{R} 1 \oplus \mathbb{R} i
\oplus \mathbb{R} j \oplus \mathbb{R} k
\end{equation}
$\mathbb{R}1$, $\mathbb{R}i$, $\mathbb{R}j$ and $\mathbb{R}k$
are vector subspaces of dimension $1$,
and with equation \eqref{eq:quaternions} we indicate that
their sum is direct and that it is the whole $\mathbb{H}$.
The product is defined in the elements of the basis
by means of the multiplication table of figure \ref{fig_en:mult_table_H},
and is extended to all quaternions by bilinearity.
For instance,
$ (1+2i)(3i+j) = 3i+j+6i^2+2ij = -6+3i+j+2k $.

\begin{figure}
\begin{equation*}
\begin{array}{c|cccc}
	\cdot & 1 & i & j & k \\
	\hline
	1 & 1 & i & j & k \\
	i & i & -1 & k & -j \\
	j & j & -k & -1 & i \\
	k & k & j & -i & -1
\end{array}
\end{equation*}
\caption{Multiplication table of quaternions $\mathbb{H}$.}
\label{fig_en:mult_table_H}
\end{figure}

We could check,
by analyzing $ 4 \cdot 4 \cdot 4 = 64 $ cases,
that the product of quaternions is associative,
but fortunately there is a shorter way to prove it.
We consider the following matrices in $M_2(\mathbb{C})$:
\begin{equation}\label{eq:quatern_in_M2C}
I = \begin{pmatrix} 1 & 0 \\ 0 & 1 \end{pmatrix}
\quad
A_i = \begin{pmatrix} i & 0 \\ 0 & -i \end{pmatrix}
\quad
A_j = \begin{pmatrix} 0 & 1 \\ -1 & 0 \end{pmatrix}
\quad
A_k = \begin{pmatrix} 0 & i \\ i & 0 \end{pmatrix}
\end{equation}
We observe that the multiplication table of these four matrices
is the same as that of figure \ref{fig_en:mult_table_H},
for instance $ A_i A_j = A_k $ and $ A_i^2 = -I $.
Since the product of matrices is associative,
so it has to be that of quaternions.

Remark that the algebra of quaternions $\mathbb{H}$
is not only unital,
but it is also a division algebra,
because given a nonzero quaternion $ X = a+bi+cj+dk $,
there exists a quaternion $X^{-1}$ such that $ XX^{-1} = 1 = X^{-1}X $.
Specifically,
$X^{-1}$ is determined by the following formula:
\begin{equation}
(a+bi+cj+dk)^{-1} = \frac{1}{a^2+b^2+c^2+d^2} \, (a-bi-cj-dk)
\end{equation}

An algebra is said to be simple if
it does not contain any proper two-sided ideal and its product is nontrivial.
Later on we will give characterizations of simple algebras
that will be more manageable for us.
Just keep in mind that these algebras are very important,
because in some sense they are the indivisible atoms
from which a great amount of algebras are built.

Finally let us recall that
a group is a set endowed with a binary operation such that:
it is associative,
there exists a neutral element $e$,
and every element has an inverse.
If besides the operation is commutative,
the group is said to be abelian.
For example,
figure \ref{fig_en:mult_table_Z4_Z22} shows
the multiplication tables of two abelian groups of $4$ elements.
The first one is
$ \mathbb{Z}_4 = \langle a \mid a^4 = e \rangle $,
and the second one is
$ \mathbb{Z}_2^2 = \langle a,b \mid a^2=e=b^2 , \allowbreak \, ab=ba \rangle $.

\begin{figure}
\begin{equation*}
\begin{array}{c|cccc}
	\cdot & e & a & a^2 & a^3 \\
	\hline
	e & e & a & a^2 & a^3 \\
	a & a & a^2 & a^3 & e \\
	a^2 & a^2 & a^3 & e & a \\
	a^3 & a^3 & e & a & a^2
\end{array}
\qquad \qquad
\begin{array}{c|cccc}
	\cdot & e & a & b & ab \\
	\hline
	e & e & a & b & ab \\
	a & a & e & ab & b \\
	b & b & ab & e & a \\
	ab & ab & b & a & e
\end{array}
\end{equation*}
\caption{On the left the multiplication table of the group $\mathbb{Z}_4$,
on the right that of the group $\mathbb{Z}_2^2$.}
\label{fig_en:mult_table_Z4_Z22}
\end{figure}

\section{Classifications}

As we have mentioned in the introduction,
the objective of my thesis is to classify gradings,
but in order to make a classification
the first thing that has to be clear is
when we understand that two objects are equal.
This question is not at all trivial.

For instance,
let us consider the complex algebra $\mathbb{H}_{\mathbb{C}}$,
obtained by extending scalars
from the field $\mathbb{R}$ to the field $\mathbb{C}$
in the real algebra of quaternions $\mathbb{H}$.
That is,
the elements $X=a+bi+cj+dk$ of $\mathbb{H}_{\mathbb{C}}$
continue to satisfy the multiplication table
of figure \ref{fig_en:mult_table_H},
but in this case the coefficients $a$, $b$, $c$, $d$
are complex numbers instead of real numbers.
Now we compare the algebras $\mathbb{H}_{\mathbb{C}}$ and $M_2(\mathbb{C})$.
Although both of them are complex associative algebras of dimension $4$,
in principle they seem quite different.
However let us define a map
$ f : \mathbb{H}_{\mathbb{C}} \to M_2(\mathbb{C}) $
as $ f(a+bi+cj+dk) = a I + b A_i + c A_j + d A_k $,
where $I$, $A_i$, $A_j$ and $A_k$ are the complex matrices
of equation \eqref{eq:quatern_in_M2C}.
We observe that this map is bijective
and commutes with the operations of both complex algebras:
$ f(X+Y) = f(X)+f(Y) $,
$ f( \lambda X ) = \lambda f(X) $,
and $ f(XY) = f(X)f(Y) $
for all elements $X,Y$ in $\mathbb{H}_{\mathbb{C}}$
and all scalar $\lambda$ in $\mathbb{C}$.
Therefore we can state that
$\mathbb{H}_{\mathbb{C}}$ and $M_2(\mathbb{C})$
are the same algebra,
but written with different alphabets,
and that the map $f$ is a dictionary
that allows us to translate from one to the other.
Formally it is said that the complex algebras
$\mathbb{H}_{\mathbb{C}}$ and $M_2(\mathbb{C})$ are isomorphic,
written $ \mathbb{H}_{\mathbb{C}} \cong M_2(\mathbb{C}) $,
and that $f$ is an isomorphism.

On the other hand,
the real algebras $\mathbb{H}$ and $M_2(\mathbb{R})$ are not isomorphic,
since the first one is a division algebra but the second one is not.
Indeed,
the matrices whose determinant is $0$ are not invertible.

Let us come back to the groups $\mathbb{Z}_4$ and $\mathbb{Z}_2^2$
of figure \ref{fig_en:mult_table_Z4_Z22}.
Although both of them are abelian groups of $4$ elements,
they are not isomorphic,
because in $\mathbb{Z}_2^2$
the square of any element is the neutral element $e$,
whereas in $\mathbb{Z}_4$
there are elements of order $4$ ($a$ and $a^3$).
With some patience it can be checked that
any group of $4$ elements is isomorphic to one of these two groups.

A very illustrative example of classification is that of the
simple finite-dimensional real associative algebras.
It consists of two parts.
On the one hand Frobenius classifies in \cite{Fro1878} the
division finite-dimensional real associative algebras.
It turns out that, up to isomorphism, there are only three:
the field of real numbers $\mathbb{R}$,
the field of complex numbers $\mathbb{C}$,
and the algebra of quaternions $\mathbb{H}$.
The easiest part of the classification is to check that indeed
these three algebras satisfy the required properties of
associativity, divisibility, et cetera.
The true problem is to prove that
there are no more algebras with these characteristics,
hence this result is known as Frobenius theorem.
Finally it has to be checked that there are no redundancies;
in this case we can say that the three algebras are not isomorphic
since they have different dimension:
$1$, $2$ and $4$ respectively.

On the other hand,
the theorem of Wedderburn \cite{Wed1908} and Artin \cite{Art1927}
says that any
simple finite-dimensional associative algebra
can be written as an algebra of square matrices
with entries in a division algebra.
Besides this result is really a classification,
since it also states
when two of these matrix algebras are isomorphic.
This happens if and only if the size of the matrices coincide
and the corresponding division algebras are isomorphic.
Putting together the two theorems we can give the list,
exhaustive and without repetitions,
of the simple finite-dimensional real associative algebras:
\begin{equation}\label{eq:list_simple_alg}
\begin{array}{cccc}
	\mathbb{R} , & M_2(\mathbb{R}) , & M_3(\mathbb{R}) , & \dots \\
	\mathbb{C} , & M_2(\mathbb{C}) , & M_3(\mathbb{C}) , & \dots \\
	\mathbb{H} , & M_2(\mathbb{H}) , & M_3(\mathbb{H}) , & \dots
\end{array}
\end{equation}

\section{Gradings}

Algebras by themselves have been studied thoroughly for a long time.
The purpose of gradings is to go one step further
and investigate in which ways algebras can be decomposed.
For a decomposition to be valid,
we are going to require it to be compatible with the operations of the algebra.

The precise definition of grading
on an algebra $\mathcal{D}$ by a group $G$
consists of two parts.
First,
the grading as such is simply a decomposition
of the underlying vector space of $\mathcal{D}$
into a direct sum of vector subspaces,
so that every vector subspace
is indexed by an element of the group $G$.
This is expressed with mathematical notation in the following way:
\begin{equation}\label{eq:grad_cond_sum}
\mathcal{D} = \bigoplus_{ g \in G } \mathcal{D}_g
\end{equation}
The vector subspace $\mathcal{D}_g$ is called
the homogeneous component of degree $g$,
and its vectors $ X \in \mathcal{D}_g $ are said to be
homogeneous elements of degree $g$,
written $ \deg X = g $.
Equation \eqref{eq:grad_cond_sum} means that the decomposition
is compatible with the sum and with the multiplication by scalars;
but if $\mathcal{D}$ is a graded algebra,
the product also has to be respected.
Thus,
as a second condition,
the product
of a homogeneous element of degree $g$
by a homogeneous element of degree $h$
is required to be a homogeneous element of degree $gh$.
If we write this with symbols,
we would say that for all $g,h$ in $G$
the following formula has to be satisfied:
\begin{equation}\label{eq:grad_cond_prod}
\mathcal{D}_g \mathcal{D}_h \subseteq \mathcal{D}_{gh}
\end{equation}

\begin{example}\label{exam:grad_M2R_dim1}
We can decompose the aforementioned algebra $M_2(\mathbb{R})$
into a direct sum of four vector subspaces of dimension $1$
by means of the following equation:
\begin{equation}\label{eq:grad_M2R_dim1}
M_2(\mathbb{R}) =
	\mathbb{R} \begin{pmatrix} 1 & 0 \\ 0 & 1 \end{pmatrix}
	\oplus
	\mathbb{R} \begin{pmatrix} -1 & 0 \\ 0 & 1 \end{pmatrix}
	\oplus
	\mathbb{R} \begin{pmatrix} 0 & 1 \\ 1 & 0 \end{pmatrix}
	\oplus
	\mathbb{R} \begin{pmatrix} 0 & -1 \\ 1 & 0 \end{pmatrix}
\end{equation}
If we assign to these four subspaces
the degrees $e$, $a$, $b$, $ab$ in the group
$ \mathbb{Z}_2^2 = \langle a,b \mid a^2=e=b^2 , \allowbreak \, ab=ba \rangle $,
we obtain a grading.
Indeed,
analyzing the $ 4 \cdot 4 = 16 $ possible cases,
we see that equation \eqref{eq:grad_cond_prod} is satisfied.
\end{example}

\begin{example}\label{exam:grad_H_dim1}
The own definition of the quaternions $\mathbb{H}$
suggests a grading.
Again we assign degrees in the group $\mathbb{Z}_2^2$,
this time to the four subspaces of equation \eqref{eq:quaternions},
$ \mathbb{H} = \mathbb{R} 1 \oplus \mathbb{R} i
\oplus \mathbb{R} j \oplus \mathbb{R} k $.
If we compare the multiplication table of $\mathbb{H}$
(figure \ref{fig_en:mult_table_H})
with that of $\mathbb{Z}_2^2$
(figure \ref{fig_en:mult_table_Z4_Z22}),
we check that the product is respected.
\end{example}

\begin{example}\label{exam:grad_C_dim1}
Another natural grading is the one obtained on the complex numbers
by separating into real part and imaginary part:
\begin{equation}\label{eq:grad_C_dim1}
\mathbb{C} = \mathbb{R} 1 \oplus \mathbb{R} i
\end{equation}
In this case the grading group is
$ \mathbb{Z}_2 = \langle a \mid a^2=e \rangle $.
\end{example}

Let us point out a mathematical nuance.
Formally,
the definition of grading allows
some of the homogeneous components to have dimension $0$.
For this reason we define the support of a grading,
which is the set of elements of the grading group
whose homogeneous components are nonzero.

Given two gradings on the same algebra,
it may occur that the first one is a refinement of the second one,
or in other words,
that the second one is a coarsening of the first one.
This happens when every homogeneous component of the first grading
is contained in a homogeneous component of the second grading.

\begin{example}\label{exam:grad_M2R_dim2}
A coarsening of the grading
on the algebra $M_2(\mathbb{R})$
of example \ref{exam:grad_M2R_dim1}
is the grading by the group $\mathbb{Z}_2$
defined by the following equation:
\begin{equation}\label{eq:grad_M2R_dim2}
M_2(\mathbb{R}) =
\left[
	\mathbb{R} \begin{pmatrix} 1 & 0 \\ 0 & 1 \end{pmatrix}
	\oplus
	\mathbb{R} \begin{pmatrix} 0 & -1 \\ 1 & 0 \end{pmatrix}
\right]
	\oplus
\left[
	\mathbb{R} \begin{pmatrix} 0 & 1 \\ 1 & 0 \end{pmatrix}
	\oplus
	\mathbb{R} \begin{pmatrix} -1 & 0 \\ 0 & 1 \end{pmatrix}
\right]
\end{equation}
We observe that in this case there are
two homogeneous components of dimension $2$ each one.
\end{example}

The gradings of examples
\ref{exam:grad_M2R_dim1}, \ref{exam:grad_H_dim1} and \ref{exam:grad_C_dim1}
have all their homogeneous components of dimension $1$,
hence they cannot be refined any more.
Because of this last reason they are said to be fine.
Reciprocally the trivial grading,
which includes the whole algebra in a single homogeneous component,
is always the coarsest.

A graded unital associative algebra
is said to be a graded division algebra
if every nonzero homogeneous element has an inverse.
For instance,
the gradings of examples
\ref{exam:grad_M2R_dim1}, \ref{exam:grad_H_dim1},
\ref{exam:grad_C_dim1} and \ref{exam:grad_M2R_dim2}
are division gradings.

Recall that our objective is going to be to classify gradings.
Therefore,
given two graded algebras
$ \mathcal{D} = \bigoplus_{ g \in G } \mathcal{D}_g $ and
$ \mathcal{E} = \bigoplus_{ h \in H } \mathcal{E}_h $,
it has to be clear if we consider that they are equal or not.
There are two natural ways to define this equality,
depending on whether the grading group plays a secondary role
or it is a part of the definition,
so in order to prevent confusions
we will call one equivalence and the other isomorphism.
We say that
the graded algebras $\mathcal{D}$ and $\mathcal{E}$ are equivalent
if there exists an isomorphism of algebras
$ \psi : \mathcal{D} \to \mathcal{E} $
satisfying the following condition:
for all $g$ in the support of $\mathcal{D}$
there exists an $ h \in H $ such that
$ \psi(\mathcal{D}_g) = \mathcal{E}_h $.
If $G=H$,
we can strengthen the required condition to
$ \psi(\mathcal{D}_g) = \mathcal{E}_g $
for all $ g \in G $,
and then we say that
the graded algebras $\mathcal{D}$ and $\mathcal{E}$ are isomorphic.

\section{Tensor products and quadratic forms}

Finally let us review some more concepts that are going to appear.
We start with the graded tensor product.
We recall that,
if $ \{ X_1 , \allowbreak X_2 , \allowbreak \dots , \allowbreak X_r \} $
is a basis of a vector space $\mathcal{D}$
and $ \{ Y_1 , \allowbreak Y_2 , \allowbreak \dots , \allowbreak Y_s \} $
is a basis of another vector space $\mathcal{E}$,
then the tensor product $ \mathcal{D} \otimes \mathcal{E} $
is a vector space of which a basis is
$ \{ X_i \otimes Y_j \mid 1 \leq i \leq r , \allowbreak \, 1 \leq j \leq s \} $.
Note that the scalar field
plays an important role in this construction,
so we will indicate it with a subscript of the symbol $\otimes$
when it is different from the real numbers $\mathbb{R}$.
For instance,
the real vector space $ \mathbb{C} \otimes M_2(\mathbb{C}) $
has dimension $ 2 \cdot 8 = 16 $,
whereas the complex vector space
$ \mathbb{C} \otimes_{\mathbb{C}} M_2(\mathbb{C}) $
has dimension $ 1 \cdot 4 = 4 $.

If $\mathcal{D}$ and $\mathcal{E}$
are not only vector spaces
but algebras,
then $ \mathcal{D} \otimes \mathcal{E} $
also has an algebra structure,
with the componentwise product:
$ ( X \otimes Y ) ( X' \otimes Y' ) = (XX') \otimes (YY') $.
Therefore,
if the algebra $\mathcal{D}$ is graded by the group $G$
and the algebra $\mathcal{E}$ is graded by the group $H$,
it is natural that we define a grading
on the algebra $ \mathcal{D} \otimes \mathcal{E} $
by the group $ G \times H $
by saying that the homogeneous component of degree $(g,h)$
is simply $ \mathcal{D}_g \otimes \mathcal{E}_h $.

Let us now move to the topic of quadratic forms.
An ($\mathbb{R}$-valued) alternating bicharacter on a group $T$
is a map $ \beta : T \times T \to \mathbb{R} \setminus \{ 0 \} $
that satisfies that
$ \beta(uv,w) = \beta(u,w) \beta(v,w) $,
$ \beta(u,vw) = \beta(u,v) \beta(u,w) $,
and $ \beta(u,u) = 1 $
for all $ u,v,w \in T $.
If the group $T$ is finite,
then $\beta$ can only take the values $+1$ and $-1$.

The radical of $\beta$ is the set
$ \rad(\beta) =
\{ t \in T \mid \beta(u,t) = 1
\allowbreak \: \text{for} \allowbreak \: \text{all} \allowbreak \:
u \in T \} $.
We will say that $\beta$ is of type I
if the only element of its radical is the neutral element of the group,
$e$.
Analogously,
we will say that $\beta$ is of type II
if $\rad(\beta)$ has two elements;
in this case the element of the radical
that is not the neutral element
will be denoted $f_{\beta}$.

A quadratic form on $T$
is a map $ \mu : T \to \{ \pm 1 \} $
such that $\beta_{\mu}$ is an alternating bicharacter,
where $ \beta_{\mu} : T \times T \to \{ \pm 1 \} $ is the map,
called polarization of $\mu$,
defined by the following formula:
\begin{equation}\label{eq:polarization}
\beta_{\mu}(u,v) = \mu(uv) \mu(u)^{-1} \mu(v)^{-1}
\end{equation}

Let us remark that quadratic forms
are well-known objects in mathematics.
They are usually written with additive notation,
as in equation \eqref{eq:bil_form_quadr_form},
but in this text we have preferred to use multiplicative notation,
because it will be more convenient for our purposes.
Note also that the inverses that appear in
equation \eqref{eq:polarization}
have no effect,
but this is how this formula is usually written.

The Arf invariant of a map
$ \mu : T \to \{ \pm 1 \} $
defined on a finite set $T$
is the value which is assumed most often by the map.
That is,
if there are more elements in $T$ of value $+1$ than $-1$,
then the Arf invariant of $\mu$ is $+1$.
But if there are less,
then it is $-1$.
It can also occur that there is the same number,
and then we say that the Arf invariant is $0$.

\selectlanguage{english}
\chapter{First article}\label{ch:art_1}

\begin{enumerate}

\item[\cite{Rod2016}]
A. Rodrigo-Escudero,
Classification of division gradings on fi\-nite-di\-men\-sional simple real algebras,
Linear Algebra Appl. {\bf 493} (2016), 164--182.
\\
\url{https://arxiv.org/abs/1506.01552}
\\
\url{https://doi.org/10.1016/j.laa.2015.11.025}

\end{enumerate}

\section{State of the art}\label{sect:state_art}

Just before I started my thesis my supervisor,
Alberto Elduque,
and one of his main collaborators,
Mikhail Kochetov,
published the monograph \cite{EK2013}.
This has been very fortunate for me.
On the one hand because it accelerated considerably
my bibliographic search task,
which is the previous step to the actual research work.
In fact this book,
and the references that it contains,
are the recommendable starting point for anyone
who is going to begin working with gradings.
It summarizes and improves the works of many authors,
as Patera, Zassenhaus, Bahturin, Zaicev, Draper,
or Mart\'in Gonz\'alez
(\cite{PZ1989}, \cite{HPP1998},
\cite{BSZ2001}, \cite{BZ2002}, \cite{BSZ2005}, \cite{BZ2006}, \cite{BZ2007},
\cite{DM2006}, \cite{DM2009},
\cite{BK2010}, \cite{Eld2010},
et cetera).

And on the other hand because one of the themes of the monograph
is gradings on complex algebras,
so the classification problems that are tackled in
my first three articles \cite{Rod2016}, \cite{BKR2018a} and \cite{BKR2018b}
were already solved in \cite{EK2013}
for the case in which the scalar field
is the complex numbers instead of the real numbers.
The first approach to attack those classifications of gradings
has been to try to adapt the techniques from the complex case to the real case.

One of the results that was already proved
and on which my thesis relies
is the graded analogue of the Artin--Wedderburn theorem.
It appears in \cite[corollary 2.12]{EK2013}
in the most general version that is possible,
although it had already been studied in
\cite{BSZ2001}, \cite{BZ2002}, and even \cite{NO1982}.
Now we are going to state it in the particular case
of simple algebras of finite dimension.

Suppose that we have a graded division algebra $\mathcal{D}$
by an abelian group $G$,
and that $ g_1 , g_2 , \dots , g_k $ are $k$ elements
(not necessarily different) of $G$.
With these ingredients we can construct in a natural way
a $G$-graded algebra;
we take as underlying algebra the algebra $M_k(\mathcal{D})$
of $ k \times k $ matrices with entries in $\mathcal{D}$,
and we define the grading by saying that
the matrix whose entries are all zero,
except the entry $(i,j)$ which is $ d \in \mathcal{D}_t $,
has degree $ g_i t g_j^{-1} $.
The theorem tells us that any
$G$-graded simple finite-dimensional associative algebra
is isomorphic to one of the form $M_k(\mathcal{D})$.

Besides,
this result is really a classification,
because it also tells us
when two of these $G$-graded algebras are isomorphic.
In this text we are not going to state the isomorphism condition,
because it requires a more formal notation.
On the other hand,
the classification up to equivalence is a difficult problem,
which is known only in the case of fine gradings.

What matters to us at this moment is that,
in order to complete the classification up to isomorphism of the gradings
on the simple finite-dimensional associative algebras,
a graded analogue of the theorem of Frobenius has to be proved.
That is, the graded division algebras have to be classified.
This question is solved in \cite[theorem 2.15]{EK2013}
(see also \cite{BSZ2001} and \cite{BK2010})
in the case in which the scalar field is the complex numbers;
but the real case was an open problem at the beginning of my thesis.

\section{Motivation and objective}\label{sect:art_1_mot_obj}

In autumn of 2014 my supervisor poses to me
my first research problem.
The question is to find all the gradings
(by abelian groups)
on the real algebra $M_2(\mathbb{C})$.
We recall that a grading is simply
a decomposition of the algebra
compatible with its operations
of sum, product and multiplication by scalars.
It turns out that as a complex algebra of dimension $4$
the gradings of $M_2(\mathbb{C})$ were known,
but not as a real algebra of dimension $8$.
As we have already seen,
thanks to the graded analogue of the theorem of Artin--Wedderburn,
the question is reduced to searching those gradings that are division gradings.
The first discovery was the grading of the following example.

\begin{example}\label{exam:grad_M2C_dim1}
Figure \ref{fig_en:grad_M2C_dim1} shows us
a division grading on the real algebra $M_2(\mathbb{C})$,
in which the grading group is
$ \mathbb{Z}_2 \times \mathbb{Z}_4 =
\langle a,b \mid a^2=e=b^4 , \allowbreak \, ab=ba \rangle $
and the degrees are assigned in the following order:
$e$,~$a$; $b$,~$ab$; $b^2$,~$ab^2$; $b^3$,~$ab^3$.
\begin{figure}
\begin{align*}
M_2(\mathbb{C}) = {} &
	\mathbb{R} \begin{pmatrix} 1 & 0 \\ 0 & 1 \end{pmatrix}
	\oplus
	\mathbb{R} \begin{pmatrix} 0 & 1 \\ 1 & 0 \end{pmatrix}
	\oplus
\\ &
	\mathbb{R} \begin{pmatrix} 1+i & 0 \\ 0 & -1-i \end{pmatrix}
	\oplus
	\mathbb{R} \begin{pmatrix} 0 & -1-i \\ 1+i & 0 \end{pmatrix}
	\oplus
\\ &
	\mathbb{R} \begin{pmatrix} i & 0 \\ 0 & i \end{pmatrix}
	\oplus
	\mathbb{R} \begin{pmatrix} 0 & i \\ i & 0 \end{pmatrix}
	\oplus
\\ &
	\mathbb{R} \begin{pmatrix} -1+i & 0 \\ 0 & 1-i \end{pmatrix}
	\oplus
	\mathbb{R} \begin{pmatrix} 0 & 1-i \\ -1+i & 0 \end{pmatrix}
\end{align*}
\caption{Division grading of example \ref{exam:grad_M2C_dim1}.}
\label{fig_en:grad_M2C_dim1}
\end{figure}
\end{example}

Remark that the presence of a $\mathbb{Z}_4$ factor
in the grading group of example \ref{exam:grad_M2C_dim1}
is something exceptional.
Using similar arguments
to those that had already been employed to solve the complex case,
and which later on we will explain with some detail,
we found all the possible division gradings on $M_2(\mathbb{C})$
whose homogeneous components have dimension $1$.

The problem was the gradings
with components of dimension greater than $1$.
After some computations,
we obtained a list that we expected it contained all of them.
And indeed,
we proved that there are no more gradings,
by means of an argument that relies on the theorem of Cayley--Hamilton.
The bad news is that that proof
could not be generalized to the case of matrices
of size greater than $ 2 \times 2 $.
However,
knowing what happens in the cases of low dimension is always useful,
and in this occasion it allowed us to observe that
all these gradings were coarsenings of
those whose homogeneous components have dimension $1$.
This observation became a conjecture and,
with some ingenuity,
the following result.

\begin{theorem}\label{th:coars}
If a division grading by an abelian group
on a fi\-nite-di\-men\-sional real associative algebra is fine,
then either all the homogeneous components have dimension $1$,
or it is a grading as a complex algebra.
\end{theorem}

This theorem was fundamental,
because thanks to it
we thought that we could solve
not only the case of $ 2 \times 2 $ matrices,
but also that of matrices of any size.
It is at this moment that
we proposed as feasible the following objective,
which would be the purpose of
my first article \cite{Rod2016}.

\begin{objective}\label{obj_en:art_1}
We classify,
up to isomorphism and up to equivalence,
division gradings
(by abelian groups)
on simple finite-dimensional real associative algebras.
\end{objective}

As we have mentioned previously,
this is a graded analogue of the theorem of Frobenius,
and when we add it to the graded version of the theorem of Artin--Wedderburn,
we obtain the classification up to isomorphism
of (not necessarily division) gradings
on simple finite-dimensional real associative algebras,
$M_n(\mathbb{R})$, $M_n(\mathbb{C})$ and $M_n(\mathbb{H})$.
On the other hand,
another reason that we kept in mind
to classify gradings on associative algebras
was to use these results to do the same classification on Lie algebras.
But we will return to this question
when we speak about the third article \cite{BKR2018b}
in chapter \ref{ch:art_3}.

Our strategy in order to carry out the classification
will be to separate the problem into three cases,
according to the following remark,
which is a consequence of the theorem of Frobenius.

\begin{remark}\label{rem:dim_comp}
The homogeneous components
of a division grading
on a finite-dimensional real associative algebra
have all the same dimension,
which is equal to $1$, $2$ or $4$,
depending on whether the neutral component is isomorphic to
$\mathbb{R}$, $\mathbb{C}$ or $\mathbb{H}$,
respectively.
\end{remark}

\section{Dimension one}

Let us tackle the classification
of division gradings
on simple fi\-nite-di\-men\-sional real associative algebras.
First we consider the case in which
the homogeneous components have dimension $1$,
according to remark \ref{rem:dim_comp}.

We start by analyzing the graded division algebra $\mathcal{D}$
of example \ref{exam:grad_M2R_dim1}.
The underlying algebra is $M_2(\mathbb{R})$
and the support is $ T = \mathbb{Z}_2^2 = \{ e,a,b,ab \} $.
Call $X_a$ the second matrix that appears
in equation \eqref{eq:grad_M2R_dim1},
and $X_b$ the third one.
Thus,
that equation is written in the following way:
\begin{equation}\label{eq:exam_grad_class_dim1}
M_2(\mathbb{R}) = \mathbb{R} I \oplus \mathbb{R} X_a
\oplus \mathbb{R} X_b \oplus \mathbb{R} X_a X_b
\end{equation}
We also note that $X_a$ and $X_b$ satisfy the following relations:
\begin{equation}\label{eq:exam_rel_class_dim1}
X_a^2 = + I \qquad X_b^2 = + I \qquad X_a X_b = - X_b X_a
\end{equation}
The idea is that equations
\eqref{eq:exam_grad_class_dim1} and \eqref{eq:exam_rel_class_dim1}
completely determine the graded algebra $\mathcal{D}$.
For instance,
$ ( X_a + 3 X_a X_b ) ( X_a + 2 X_b ) =
X_a^2 + 2 X_a X_b + 3 X_a X_b X_a + 6 X_a X_b^2 =
I + 6 X_a - 3 X_b + 2 X_a X_b $.

The three relations
of equation \eqref{eq:exam_rel_class_dim1}
are enough,
but a problem arises.
Suppose that we had chosen the fourth matrix,
$X_{ab}$,
instead of the third one,
then the second relation would be different:
$ X_a^2 = + I $,
$ X_{ab}^2 = - I $,
$ X_a X_{ab} = - X_{ab} X_a $.
However it is clear that these new relations
define the same graded algebra.
The solution in order to avoid redundancies
will be to keep all relations.

We define the map
$ \beta : T \times T \to \mathbb{R} \setminus \{ 0 \} $
by means of the commutation relations of the homogeneous elements
($ X_e = I $):
\begin{equation}\label{eq:comm_rel}
X_u X_v = \beta(u,v) X_v X_u
\end{equation}
We also define the map $ \mu : T \to \{ \pm 1 \} $
by means of the signs of the squares of the homogeneous elements:
\begin{equation}\label{eq:sign_squar}
X_t^2 = \mu(t) I
\end{equation}
With a small computation we check that
$\beta$ is an alternating bicharacter of type I,
and that $\mu$ is a quadratic form
whose polarization $\beta_{\mu}$ is precisely $\beta$.
Besides,
these maps $\beta$ and $\mu$ are the invariants that characterize
the graded algebra $\mathcal{D}$ up to isomorphism.
This means that,
first,
$\beta$ and $\mu$ do not depend on the choice of the $X_t$,
and second,
if we repeat this construction
starting with another graded algebra $\mathcal{D}'$,
obtaining $\beta'$ and $\mu'$,
then $ \beta = \beta' $ and $ \mu = \mu' $ if and only if
$\mathcal{D}$ is isomorphic to $\mathcal{D}'$.
In fact,
since $ \beta_{\mu} = \beta $,
all the information is contained in the quadratic form $\mu$.

These arguments are still valid in the general case
in which the underlying algebra of $\mathcal{D}$
is $M_n(\mathbb{R})$ or $M_{n/2}(\mathbb{H})$.
Since $\beta$ is of type I,
the theory of alternating bicharacters tells us that
the group $T$ is isomorphic to $\mathbb{Z}_2^{2m}$.
Therefore there are $2^{2m}$ homogeneous components of dimension $1$,
and we already know that the dimension of $\mathcal{D}$ is $n^2$,
hence necessarily $ n = 2^m $.
That is,
these division gradings exist only if
the size of the matrices is a power of $2$.

Now let us consider the classification up to equivalence.
Two of these graded algebras are equivalent if and only if
their corresponding quadratic forms are equivalent,
that is,
if after renaming the elements of $T$,
both quadratic forms are equal.
The classification of quadratic forms with polarization of type I
is well known (the idea is to take a symplectic basis);
they are divided into two equivalence classes,
one corresponding to the quadratic forms whose Arf invariant is $+1$,
and another corresponding to those of Arf invariant $-1$.

On the other hand we have to prove,
in the general case $ n = 2^m $,
the existence of these division gradings.
We consider the following graded tensor product,
where the grading of each factor
is that of example \ref{exam:grad_M2R_dim1}:
\begin{equation}
M_n(\mathbb{R}) \cong M_2(\mathbb{R}) \otimes
\dots \otimes M_2(\mathbb{R})
\end{equation}
It is,
indeed,
a graded division algebra
with homogeneous components of dimension $1$,
and its corresponding quadratic form has Arf invariant $+1$.
Therefore all other quadratic forms
on $ (\mathbb{Z}_2^2)^m \cong \mathbb{Z}_2^{2m} $
with polarization of type I and Arf invariant $+1$
are also obtained from division gradings on $M_n(\mathbb{R})$,
because they are equivalent to the one that we have just constructed.
Analogously,
thanks to this other tensor product
whose factor $\mathbb{H}$ is graded
as in example \ref{exam:grad_H_dim1},
we see that those with Arf invariant $-1$
come from division gradings on $M_{n/2}(\mathbb{H})$:
\begin{equation}
M_{n/2}(\mathbb{H}) \cong M_2(\mathbb{R}) \otimes
\dots \otimes M_2(\mathbb{R}) \otimes \mathbb{H}
\end{equation}

We can already state the classification up to equivalence,
and the classification up to isomorphism,
for instance in the case of quaternion matrices.
If the underlying algebra is $M_n(\mathbb{C})$
the reasoning is similar,
but the alternating bicharacters that are obtained have type II;
this causes a second equivalence class to appear,
in which the arguments are a bit more complicated
because the grading group has elements of order $4$.

\begin{theorem}\label{th:art_one_dim1_equiv}
Any division grading,
by an abelian group,
on a simple finite-dimensional real associative algebra,
whose homogeneous components have dimension $1$,
is equivalent to one, and only one,
of the graded tensor products of the following list,
where the gradings of each factor are those of examples
\ref{exam:grad_M2R_dim1}, \ref{exam:grad_H_dim1},
\ref{exam:grad_C_dim1} and \ref{exam:grad_M2C_dim1}.
\begin{enumerate}
\item[(1-a)]
	$ M_n(\mathbb{R}) \cong M_2(\mathbb{R}) \otimes
	\dots \otimes M_2(\mathbb{R}) $,
	$ n = 2^m \geq 1 $.\\
	The grading group is
	$ (\mathbb{Z}_2^2)^m \cong \mathbb{Z}_2^{2m} $.
\item[(1-b)]
	$ M_{n/2}(\mathbb{H}) \cong M_2(\mathbb{R}) \otimes
	\dots \otimes M_2(\mathbb{R}) \otimes \mathbb{H} $,
	$ n = 2^m \geq 2 $.\\
	The grading group is
	$ (\mathbb{Z}_2^2)^m \cong \mathbb{Z}_2^{2m} $.
\item[(1-c)]
	$ M_n(\mathbb{C}) \cong M_2(\mathbb{R}) \otimes \dots
	\otimes M_2(\mathbb{R}) \otimes \mathbb{C} $,
	$ n = 2^m \geq 1 $.\\
	The grading group is
	$ (\mathbb{Z}_2^2)^m \times \mathbb{Z}_2 \cong \mathbb{Z}_2^{2m+1} $.
\item[(1-d)]
	$ M_n(\mathbb{C}) \cong M_2(\mathbb{R}) \otimes \dots
	\otimes M_2(\mathbb{R}) \otimes M_2(\mathbb{C}) $,
	$ n = 2^m \geq 2 $.\\
	The grading group is
	$ (\mathbb{Z}_2^2)^{m-1} \times ( \mathbb{Z}_2 \times \mathbb{Z}_4 )
	\cong \mathbb{Z}_2^{2m-1} \times \mathbb{Z}_4 $.
\end{enumerate}
\end{theorem}

\begin{theorem}\label{th:art_one_dim1_isom}
The division gradings,
by abelian groups,
on the algebra $M_{n/2}(\mathbb{H})$,
whose homogeneous components have dimension $1$,
exist if and only if $ n = 2^m \geq 2 $,
and are determined up to isomorphism
by the sign of the squares of the homogeneous elements.
In this way we obtain a bijective correspondence
between the isomorphism classes
and the quadratic forms on $\mathbb{Z}_2^{2m}$
with polarization of type I and Arf invariant $-1$.
All these gradings belong to the same equivalence class,
represented by (1-b) in the list of theorem \ref{th:art_one_dim1_equiv}.
\end{theorem}

\section{Dimension four}

The next step in the classification
is to consider the division gradings
whose homogeneous components have dimension $4$.
Curiously this case,
in which the neutral component is isomorphic to $\mathbb{H}$,
is easier than that of dimension $2$,
in which the neutral component is isomorphic to $\mathbb{C}$.
Also in theorem \ref{th:art_one_dim1_equiv},
the algebra of complex matrices
gave us one more equivalence class than the others.
This extra difficulty is due to the center.

Given an algebra $\mathcal{D}$,
its center $Z(\mathcal{D})$
is the set of elements of the algebra
that commute with all the others:
\begin{equation}
Z(\mathcal{D}) = \{ X \in \mathcal{D} \mid
XY = YX \text{ for all } Y \in \mathcal{D} \}
\end{equation}
In the real algebras $M_n(\mathbb{R})$ and $M_n(\mathbb{H})$,
the only matrices that commute with all the others
are the multiples of the identity matrix;
therefore $ Z ( M_n(\mathbb{R}) ) $ and $ Z ( M_n(\mathbb{H}) ) $
are real vector spaces of dimension $1$.
On the other hand,
if we consider $M_n(\mathbb{C})$ as an algebra
over the field of complex numbers,
then its center is also formed
by the multiples of the identity matrix
and it is a complex vector space of dimension $1$.
However,
as a real vector space,
$ Z ( M_n(\mathbb{C}) ) $ has dimension $2$.
It is said that the real simple algebras
$M_n(\mathbb{R})$ and $M_n(\mathbb{H})$
are central simple;
whereas the algebra $M_n(\mathbb{C})$
is central simple as a complex algebra,
but not as a real algebra.

Let us return to the classification problem.
We are treating the case in which
our graded division algebra $\mathcal{D}$
has a neutral component $\mathcal{D}_e$ isomorphic to $\mathbb{H}$.
We consider the centralizer
$ C_{\mathcal{D}} (\mathcal{D}_e) $
of $\mathcal{D}_e$ in $\mathcal{D}$,
which is defined,
in a similar way to the center,
as the set of the elements of $\mathcal{D}$
that commute with all the elements of $\mathcal{D}_e$:
\begin{equation}
C_{\mathcal{D}} (\mathcal{D}_e) = \{ X \in \mathcal{D} \mid
XY = YX \text{ for all } Y \in \mathcal{D}_e \}
\end{equation}
The centralizer $ C_{\mathcal{D}} (\mathcal{D}_e) $
is a subalgebra of $\mathcal{D}$,
because if we perform operations
with elements of $ C_{\mathcal{D}} (\mathcal{D}_e) $
the result is still an element of $ C_{\mathcal{D}} (\mathcal{D}_e) $;
moreover,
the subalgebra $ C_{\mathcal{D}} (\mathcal{D}_e) $ is graded,
because if we decompose
any element of $ C_{\mathcal{D}} (\mathcal{D}_e) $
as a sum of homogeneous elements,
these elements also belong to $ C_{\mathcal{D}} (\mathcal{D}_e) $.
It is clear that if two graded algebras are isomorphic or equivalent,
then the centralizers of its neutral components
are also isomorphic or equivalent.

Since $\mathcal{D}_e$ is a central simple algebra,
we can apply to it a deep result
of the theory of associative algebras,
the double centralizer theorem
(see for example \cite[theorem 4.7]{Jac1989}).
Adapting this theorem to the structure of grading,
we have that the graded algebra $\mathcal{D}$
is isomorphic in a natural way to the graded tensor product
of $\mathcal{D}_e$ and its centralizer $ C_{\mathcal{D}} (\mathcal{D}_e) $:
\begin{equation}\label{eq:th_double_centr}
\mathcal{D} \cong \mathcal{D}_e \otimes C_{\mathcal{D}} (\mathcal{D}_e)
\end{equation}
Therefore we obtain the converse:
if the centralizers of the neutral components
of two of these graded algebras
are isomorphic or equivalent,
then the algebras are also isomorphic or equivalent.
Besides,
equation \eqref{eq:th_double_centr} also indicates that
the homogeneous components of $ C_{\mathcal{D}} (\mathcal{D}_e) $
have dimension $1$.
Let us state these arguments in the form of a theorem,
with the minimum of hypothesis.

\begin{theorem}\label{th:art_one_dim4}
Let $\mathcal{D}$ be a unital real associative algebra
endowed with a division grading
whose homogeneous components have dimension $4$.
Then the centralizer of the neutral component
$ C_{\mathcal{D}} (\mathcal{D}_e) $
is also a graded division algebra,
but with homogeneous components of dimension $1$.
Besides,
the graded algebra $\mathcal{D}$
is isomorphic to the graded tensor product
$ \mathcal{D}_e \otimes C_{\mathcal{D}} (\mathcal{D}_e) $.
Therefore,
the isomorphism and equivalence classes of $\mathcal{D}$
are determined by those of
$ C_{\mathcal{D}} (\mathcal{D}_e) $.
\end{theorem}

For instance,
since $ M_n(\mathbb{R}) \cong \mathbb{H} \otimes M_{n/4}(\mathbb{H}) $,
putting together theorems
\ref{th:art_one_dim1_isom} and \ref{th:art_one_dim4}
we deduce that
there exist division gradings
by abelian groups on the algebra $M_n(\mathbb{R})$
with homogeneous components of dimension $4$
if and only if $ n = 2^m \geq 4 $,
and that the isomorphism classes are in bijective correspondence
with the quadratic forms on $\mathbb{Z}_2^{2m-2}$
with polarization of type I and Arf invariant $-1$.
Analogously,
each of the equivalence classes
of theorem \ref{th:art_one_dim1_equiv}
gives us exactly one equivalence class
for the case of homogeneous components of dimension $4$.

\section{Dimension two}

Finally let us analyze the case in which
the homogeneous components have dimension $2$.
We start by studying the division grading
on the real algebra $ \mathcal{D} = M_2(\mathbb{C}) $
given by the following equation:
\begin{equation}\label{eq:grad_M2C_Z22}
\begin{split}
M_2(\mathbb{C}) = {} &
\left[
	\mathbb{R} \begin{pmatrix} 1 & 0 \\ 0 & 1 \end{pmatrix}
	\oplus
	\mathbb{R} \begin{pmatrix} 0 & -1 \\ 1 & 0 \end{pmatrix}
\right]
\oplus
\left[
	\mathbb{R} \begin{pmatrix} 0 & 1 \\ 1 & 0 \end{pmatrix}
	\oplus
	\mathbb{R} \begin{pmatrix} -1 & 0 \\ 0 & 1 \end{pmatrix}
\right]
\oplus
\\ &
\left[
	\mathbb{R} \begin{pmatrix} i & 0 \\ 0 & i \end{pmatrix}
	\oplus
	\mathbb{R} \begin{pmatrix} 0 & -i \\ i & 0 \end{pmatrix}
\right]
\oplus
\left[
	\mathbb{R} \begin{pmatrix} 0 & i \\ i & 0 \end{pmatrix}
	\oplus
	\mathbb{R} \begin{pmatrix} -i & 0 \\ 0 & i \end{pmatrix}
\right]
\end{split}
\end{equation}
The grading group is
$ T = \mathbb{Z}_2^2 =
\langle g,f \mid g^2=e=f^2 , \allowbreak \, gf=fg \rangle $
and the degrees are assigned in the following order:
$e$,~$g$; $f$,~$gf$.
Unlike what happened
when the homogeneous components had dimension $1$,
now the signs of the squares of the homogeneous elements
are not always defined.
Indeed,
in the component of degree $e$ there are matrices whose square is positive,
but there are also others whose square is negative;
and the same occurs in the component of degree $f$.
Let us call $K$ the set of these degrees,
$ K = \{ e,f \} $.
We observe that this set is precisely
the support of the centralizer of the neutral component:
\begin{equation}
K = \supp ( C_{\mathcal{D}} (\mathcal{D}_e) )
\end{equation}

On the other hand,
given any nonzero homogeneous matrix of degree $g$,
its square is always a positive multiple of the identity matrix.
The same thing happens if instead of matrices of degree $g$
we take matrices of degree $gf$,
but now we obtain negative multiples of $I$.
Therefore we can define,
picking normalized matrices $X_t$,
the map $ \nu : T \setminus K \to \{ \pm 1 \} $
by means of the signs of the squares of the homogeneous elements:
\begin{equation}
X_t^2 = \nu(t) I
\end{equation}

The key point is that,
thanks to theorem \ref{th:coars},
the map $\nu$ is still well defined
in the general case in which $\mathcal{D}$ is any
simple finite-dimensional real associative algebra
(except for some complications when $T$ has elements of order $4$).
Besides,
applying the double centralizer theorem it is proved that $\nu$
determines the graded algebra $\mathcal{D}$ up to isomorphism.
Let us state this in the form of a theorem,
in the cases in which $\mathcal{D}$ is central simple,
that is,
$M_n(\mathbb{R})$ or $M_n(\mathbb{H})$,
because then the grading group has no elements of order $4$.

\begin{theorem}
The division gradings,
by abelian groups,
on the central simple finite-dimensional real associative algebras,
with homogeneous components of dimension $2$,
are determined up to isomorphism
by the signs of the squares of
the homogeneous elements that do not commute with the neutral component.
\end{theorem}

Not all the maps $ \nu : T \setminus K \to \{ \pm 1 \} $
come from division gradings,
but only those that satisfy certain properties,
which we characterize in the first article \cite{Rod2016}.

In the central simple cases
$ \mathcal{D} = M_n(\mathbb{R}) $ and $ \mathcal{D} = M_n(\mathbb{H}) $,
all these gradings belong to the same equivalence class,
which we call (2-a) and (2-b) respectively.
However,
if $ \mathcal{D} = M_n(\mathbb{C}) $ the situation is more complicated.
Let us see what is the list of the equivalence classes in this case.

\begin{example}\label{exam:grad_M2C_Z4}
Figure \ref{fig_en:grad_M2C_Z4}
defines a division grading
on the real algebra $M_2(\mathbb{C})$
by the group $\mathbb{Z}_4$.
Note that its homogeneous components have dimension $2$,
and that it is simply a coarsening of the grading
of example \ref{exam:grad_M2C_dim1}.
\begin{figure}
\begin{align*}
M_2(\mathbb{C}) = {} &
\left[
	\mathbb{R} \begin{pmatrix} 1 & 0 \\ 0 & 1 \end{pmatrix}
	\oplus
	\mathbb{R} \begin{pmatrix} 0 & i \\ i & 0 \end{pmatrix}
\right]
\oplus
\\ &
\left[
	\mathbb{R} \begin{pmatrix} 1+i & 0 \\ 0 & -1-i \end{pmatrix}
	\oplus
	\mathbb{R} \begin{pmatrix} 0 & 1-i \\ -1+i & 0 \end{pmatrix}
\right]
\oplus
\\ &
\left[
	\mathbb{R} \begin{pmatrix} i & 0 \\ 0 & i \end{pmatrix}
	\oplus
	\mathbb{R} \begin{pmatrix} 0 & 1 \\ 1 & 0 \end{pmatrix}
\right]
\oplus
\\ &
\left[
	\mathbb{R} \begin{pmatrix} -1+i & 0 \\ 0 & 1-i \end{pmatrix}
	\oplus
	\mathbb{R} \begin{pmatrix} 0 & -1-i \\ 1+i & 0 \end{pmatrix}
\right]
\end{align*}
\caption{Division grading of example \ref{exam:grad_M2C_Z4}.}
\label{fig_en:grad_M2C_Z4}
\end{figure}
\end{example}

\begin{theorem}\label{th:art_one_dim2_equiv}
Any division grading,
by an abelian group,
on the real algebra $M_n(\mathbb{C})$,
with homogeneous components of dimension $2$,
is equivalent to one, and only one,
of the graded algebras of the following list,
where the gradings of the first factor of each tensor product
are those of examples
\ref{exam:grad_M2R_dim2} and \ref{exam:grad_M2C_Z4},
whereas the rest of the factors are graded as in examples
\ref{exam:grad_M2R_dim1}, \ref{exam:grad_C_dim1} and \ref{exam:grad_M2C_dim1}.
\begin{enumerate}
\item[(2-c)]
	$ M_n(\mathbb{C}) \cong M_2(\mathbb{R}) \otimes M_2(\mathbb{R})
	\otimes \dots \otimes M_2(\mathbb{R}) \otimes \mathbb{C} $,
	$ n = 2^m \geq 2 $.\\
	The grading group is
	$ \mathbb{Z}_2 \times (\mathbb{Z}_2^2)^{m-1} \times \mathbb{Z}_2
	\cong \mathbb{Z}_2^{2m} $.
\item[(2-d)]
	$ M_n(\mathbb{C}) \cong M_2(\mathbb{R}) \otimes M_2(\mathbb{R})
	\otimes \dots \otimes M_2(\mathbb{R}) \otimes M_2(\mathbb{C}) $,
	$ n = 2^m \geq 4 $.\\
	The grading group is
	$ \mathbb{Z}_2 \times (\mathbb{Z}_2^2)^{m-2} \times
	( \mathbb{Z}_2 \times \mathbb{Z}_4 ) \cong
	\mathbb{Z}_2^{2m-2} \times \mathbb{Z}_4 $.
\item[(2-e)]
	$ M_n(\mathbb{C}) \cong M_2(\mathbb{C}) \otimes
	M_2(\mathbb{R}) \otimes \dots \otimes M_2(\mathbb{R}) $,
	$ n = 2^m \geq 2 $.\\
	The grading group is
	$ \mathbb{Z}_4 \times (\mathbb{Z}_2^2)^{m-1}
	\cong \mathbb{Z}_2^{2m-2} \times \mathbb{Z}_4 $.
\item[(2-f)]
	Gradings on $M_n(\mathbb{C})$ as an algebra over
	the scalar field of the complex numbers $\mathbb{C}$.
\end{enumerate}
\end{theorem}

If we compare the lists of theorems
\ref{th:art_one_dim1_equiv} and \ref{th:art_one_dim2_equiv},
we see that the equivalence class (1-c) gives rise to (2-c).
That is,
the situation is similar to what happened in the central simple cases.
However,
(1-d) gives rise to two different equivalence classes,
(2-d) and (2-e).
This last one occurs when $ T \setminus K $
is precisely the set of elements of order $4$ of $T$,
what makes the definition of the invariant $\nu$ more difficult.
In fact we warn anyone who reads the first article \cite{Rod2016}
that this case (2-e) is more complicated than the others.

\section{Publication}

In may 2015 we had written the article,
so we set it aside some time
in order to revise it afterwards one last time before submitting it to Arxiv.
Then we were informed that
Yuri Bahturin and Mikhail Zaicev
were working in exactly the same problem as us,
although in their case they were interested only
in the classification up to equivalence,
but not up to isomorphism.
I would lie if I said that I did not worry at that time.
However,
I think that this coincidence turned out to be very beneficial.
The fact that two groups are interested in the same problem
means that this question is important.
At the end,
we both published our articles
\cite{BZ2016} and \cite{Rod2016} simultaneously,
reaching the same results
(although in \cite{BZ2016} the equivalence class (2-d) was overlooked)
using different techniques.
Besides,
as a consequence of this coincidence,
I went to
Memorial University of Newfoundland (Canada)
for a three-month stay
in the spring of 2016,
and we continued the research in a collaborative way.

\selectlanguage{english}
\chapter{Second article}\label{ch:art_2}

\begin{enumerate}

\item[\cite{BKR2018a}]
Y. Bahturin, M. Kochetov\ and\ A. Rodrigo-Escudero,
Classification of involutions on graded-division simple real algebras,
Linear Algebra Appl. {\bf 546} (2018), 1--36.
\\
\url{https://arxiv.org/abs/1707.05526}
\\
\url{https://doi.org/10.1016/j.laa.2018.01.040}

\end{enumerate}

\section{Objective}\label{sect:art_2_obj}

Recall that the final purpose of my thesis is
to classify gradings on simple real Lie algebras.
In order to do so,
we adapt the arguments of the monograph \cite{EK2013}
from the case in which the scalar field is the complex numbers
to the case in which it is the real numbers.
In chapter \ref{ch:art_3} we will explain this process with more detail,
but the idea is that we can transfer
the classification of the gradings
on an associative algebra to a Lie algebra,
provided that the first one is endowed with an additional structure:
an involution.

Thus,
one of the steps is to study the involutions
that are compatible with the division gradings
that we have seen in chapter \ref{ch:art_1}.
The analogous of this question over the complex field
is solved in \cite[propositions 2.51 and 2.53]{EK2013}
(see also \cite{BZ2006});
the involutions are divided into two equivalence classes.
However
if the scalar field is $\mathbb{R}$
many more classes appear.
For this reason we considered that this classification
could be of independent interest,
and we proposed the objective of the second article \cite{BKR2018a}.

\begin{objective}\label{obj_en:art_2}
We classify,
up to isomorphism and up to equivalence,
involutions on
graded-division simple finite-dimensional real associative algebras,
when the grading group is abelian.
\end{objective}

\section{Involutions}

Throughout this chapter we fix
an abelian group $G$,
a simple finite-dimensional real associative algebra $\mathcal{D}$,
and a division $G$-grading $\Gamma$ on $\mathcal{D}$.
In other words,
$\mathcal{D}$ is one of the graded algebras
that we classified in chapter \ref{ch:art_1}.

For a moment let us forget about the grading.
An antiautomorphism of the algebra $\mathcal{D}$
is a map $ \varphi : \mathcal{D} \to \mathcal{D} $
that is an isomorphism of vector spaces
and reverses the order of the product.
That is,
$\varphi$ is bijective
and for all $ X,Y \in \mathcal{D} $
and all $ \lambda \in \mathbb{R} $
the following conditions are satisfied:
$ \varphi(X+Y) = \varphi(X) + \varphi(Y) $,
$ \varphi( \lambda X ) = \lambda \varphi(X) $,
and $ \varphi(XY) = \varphi(Y) \varphi(X) $.
If we say that $\varphi$ is
an antiautomorphism of the graded algebra $\mathcal{D}$,
it means that we are taking into account its grading,
and therefore one more condition has to be satisfied:
$ \varphi(X_g) \in \mathcal{D}_g $
for all $ g \in G $ and all $ X_g \in \mathcal{D}_g $.
Finally,
an involution $\varphi$ is an antiautomorphism
such that $ \varphi(\varphi(X)) = X $ for all $ X \in \mathcal{D} $.

For instance,
the map $ \varphi : M_2(\mathbb{C}) \to M_2(\mathbb{C}) $
given by transposition and conjugation,
$ \varphi(X) = \overline{X^T} $ for any matrix $ X \in M_2(\mathbb{C}) $,
is an involution on the real algebra $M_2(\mathbb{C})$.
We observe that this involution respects
the grading of equation \eqref{eq:grad_M2C_Z22}.
However,
it does not respect the grading of example \ref{exam:grad_M2C_dim1}.
In fact it can be proved that,
if the support $T$ is isomorphic to $ \mathbb{Z}_2^a \times \mathbb{Z}_4 $
(which implies that the underlying algebra $\mathcal{D}$
is isomorphic to $M_n(\mathbb{C})$),
then any antiautomorphism on the graded division algebra $\mathcal{D}$
acts as the identity on the center of $\mathcal{D}$.
This was one of the first results of the classification
that drew our attention.
Another example of involution,
this time on the graded algebra of example \ref{exam:grad_H_dim1},
is the map $ \varphi : \mathbb{H} \to \mathbb{H} $
given by $ \varphi(a+bi+cj+dk) = a+bi+cj-dk $.

Let $ \varphi : \mathcal{D} \to \mathcal{D} $
and $ \varphi' : \mathcal{D}' \to \mathcal{D}' $
be involutions on two graded algebras $\mathcal{D}$ and $\mathcal{D}'$,
whose gradings we denote by $\Gamma$ and $\Gamma'$.
We say that the pair $(\Gamma,\varphi)$ is isomorphic
(respectively equivalent)
to the pair $(\Gamma',\varphi')$
if there exists an isomorphism
(respectively equivalence)
of graded algebras
$ \psi : \mathcal{D} \to \mathcal{D}' $
that also commutes with the involutions $\varphi$ and $\varphi'$,
that is,
such that $ \varphi'(\psi(X)) = \psi(\varphi(X)) $
for all $ X \in \mathcal{D} $.

In the second article \cite{BKR2018a}
we classify the pairs $(\Gamma,\varphi)$,
up to isomorphism and up to equivalence,
for all the possible involutions $\varphi$
on the graded algebra $\mathcal{D}$.
Although a lot of cases appear,
in the article we made an effort to write explicitly all of them,
so that the classification serves as a reference.
In the following sections we are going to show
some of the most illustrative cases.

\section{Case (1-a)}

Let us suppose that the homogeneous components
of the grading $\Gamma$ have dimension $1$,
and that the underlying algebra $\mathcal{D}$
is isomorphic to $M_n(\mathbb{R})$.
That is,
the graded algebra $\mathcal{D}$ is equivalent to
the representative (1-a) of theorem \ref{th:art_one_dim1_equiv}.
We recall that,
in this situation,
the support $T$ of $\Gamma$ is a group isomorphic to $\mathbb{Z}_2^{2m}$
($ n = 2^m $),
the map $ \beta : T \times T \to \{ \pm 1 \} $
defined by equation \eqref{eq:comm_rel}
is an alternating bicharacter of type I,
and the map $ \mu : T \to \{ \pm 1 \} $
defined by equation \eqref{eq:sign_squar}
is a quadratic form such that
$ \beta_{\mu} = \beta $ and $ \Arf(\mu) = +1 $.
Besides,
the isomorphism class of $\Gamma$ is determined by $\mu$.

A straightforward computation tells us that
the involutions on the graded algebra $\mathcal{D}$
are in bijective correspondence with
the quadratic forms $ \eta : T \to \{ \pm 1 \} $
such that $ \beta_{\eta} = \beta $.
Namely,
the correspondence is given by the following equation,
where $ X_t \in \mathcal{D}_t $:
\begin{equation}\label{eq:eta}
\varphi(X_t) = \eta(t) X_t
\end{equation}
Note that there are $2^{2m}$ possible quadratic forms $\eta$
for the fixed alternating bicharacter $\beta$,
that is,
$2^{2m}$ involutions that are not isomorphic to each other.
In total,
the pair $(\Gamma,\varphi)$ is determined up to isomorphism by
two quadratic forms $\mu$ and $\eta$ defined on the same group,
the support $ T \cong \mathbb{Z}_2^{2m} $.

The classification up to equivalence is more complicated.
A reasonable conjecture is
to divide the $2^{2m}$ quadratic forms $\eta$ into three sets:
\begin{enumerate}
	\item[(1)] $ \eta = \mu $.
	\item[(2)] $ \Arf(\eta) = +1 $ but $ \eta \neq \mu $.
	\item[(3)] $ \Arf(\eta) = -1 $.
\end{enumerate}
That is,
on the one hand we have the distinguished case $ \eta = \mu $,
and on the other hand we separate
the quadratic forms that would correspond
to a division grading on $M_n(\mathbb{R})$
from those that would correspond
to a division grading on $M_{n/2}(\mathbb{H})$.

Indeed,
these are the three equivalence classes of involutions,
but the proof can be laborious,
because there are a lot of cases to analyze.
Fortunately
we found a way of dealing with all of them at the same time.
Here we are not going to repeat the arguments,
but at least let us indicate that the key idea
is to consider the two following lemmas.

\begin{lemma}
Let $\mu$ and $\eta$ be two different quadratic forms
on a finite abelian group $T$
such that $ \beta_{\mu} = \beta_{\eta} $.
Then $ \{ t \in T \mid \mu(t) = \eta(t) \} $
is a subgroup of $T$ of index $2$.
\end{lemma}

\begin{lemma}
Let $\beta$ be an alternating bicharacter of type I
on a finite abelian group $T$.
Then the following map is a bijection:
\begin{equation}
\begin{split}
T & \longrightarrow \{ S \mid
S \text{ is a subgroup of } T \text{ of index } 1 \text{ or } 2 \}
\\
u & \longmapsto u^{\perp} = \{ v \in T \mid \beta(u,v) = 1 \}
\end{split}
\end{equation}
\end{lemma}

Finally,
remark again that
the involution corresponding to $ \eta = \mu $ is distinguished.
If we take as $\mathcal{D}$
the tensor product (1-a) of theorem \ref{th:art_one_dim1_equiv},
and we identify it with $M_n(\mathbb{R})$ in the natural way,
then this distinguished involution $\varphi$
is the matrix transposition,
$ \varphi(X) = X^T $ for any matrix $ X \in M_n(\mathbb{R}) $.

\section{Case (1-c)}

The case (1-c),
in which the underlying algebra $\mathcal{D}$
is isomorphic to $M_n(\mathbb{C})$
and the support $T$ is a group isomorphic to $\mathbb{Z}_2^{2m+1}$
($ n = 2^m $),
is very similar.
In fact equations \eqref{eq:comm_rel},
\eqref{eq:sign_squar} and \eqref{eq:eta}
remain valid,
and the isomorphism class is determined by
the quadratic forms $\mu$ and $\eta$.
Let us point out the differences.
Now the alternating bicharacter $\beta$ is of type II,
and the conditions that $\mu$ has to satisfy are
$ \beta_{\mu} = \beta $ and $ \mu(f_{\beta}) = -1 $
(recall that $ \rad(\beta) = \{ e , f_{\beta} \} $).
The Arf invariant of $\mu$ is always $0$.

The condition $ \mu(f_{\beta}) = -1 $ is interesting.
It is due to the fact that the central element $iI$
is homogeneous of degree $f_{\beta}$
and $ (iI)^2 = -I $.
However
when we compute equation \eqref{eq:eta},
we realize that the only condition that $\eta$ has to satisfy
is $ \beta_{\eta} = \beta $;
so it is possible that $ \eta(f_{\beta}) = +1 $.
In fact,
the $2^{2m+1}$ quadratic forms $\eta$
are divided into four equivalence classes:
\begin{enumerate}
	\item[(1)] $ \eta = \mu $.
	\item[(2)] $ \eta(f_{\beta}) = -1 $ but $ \eta \neq \mu $.
	\item[(3)] $ \eta(f_{\beta}) = +1 $ and $ \Arf(\eta) = +1 $.
	\item[(4)] $ \eta(f_{\beta}) = +1 $ and $ \Arf(\eta) = -1 $.
\end{enumerate}

In order to arrive to this conclusion,
we computed the division gradings that would correspond to
the quadratic forms $\mu$ such that
$ \beta_{\mu} = \beta $ and $ \mu(f_{\beta}) = +1 $
(instead of $ \mu(f_{\beta}) = -1 $).
We obtained two equivalence classes,
one for the case $ \Arf(\mu) = +1 $,
with underlying algebra $ M_n(\mathbb{R}) \times M_n(\mathbb{R}) $,
and the other for the case $ \Arf(\mu) = -1 $,
with underlying algebra $ M_{n/2}(\mathbb{H}) \times M_{n/2}(\mathbb{H}) $.
We see that these two classes are reflected
in the cases (3) and (4) of the above list.
Therefore,
the classification of division gradings
on (not necessarily simple) algebras
whose center has dimension $1$ or $2$
helps us to understand that of
pairs $(\Gamma,\varphi)$ on simple algebras.
So we also tackled
part of this classification of division gradings
in the second article \cite{BKR2018a}.

Finally note that,
in the literature,
a quadratic form $\mu$ is said to be regular if
either its polarization $\beta$ is of type I,
or $\beta$ is of type II and also $ \mu(f_{\beta}) = -1 $.
This definition may seem a bit strange.
However,
now we see that the regular quadratic forms
are precisely those that correspond,
by means of equation \eqref{eq:sign_squar},
to division gradings on simple algebras.

\section{New division gradings}

Recall that the representatives of the equivalence classes
of theorem \ref{th:art_one_dim1_equiv}
are written as graded tensor products
of copies of the four graded division algebras of examples
\ref{exam:grad_M2R_dim1}, \ref{exam:grad_H_dim1},
\ref{exam:grad_C_dim1} and \ref{exam:grad_M2C_dim1}.
Now there appear more division gradings
on algebras that are not simple,
concretely on
$ M_n(\mathbb{R}) \times M_n(\mathbb{R}) $,
$ M_{n/2}(\mathbb{H}) \times M_{n/2}(\mathbb{H}) $ and
$ M_n(\mathbb{R}) \times M_{n/2}(\mathbb{H}) $.
In order to express them as graded tensor products,
we need three new types of factors,
the graded division algebras of example \ref{exam:grad_semisimple}.
We remark that it is curious that we found
the two division gradings of figure \ref{fig_en:grad_semisimple}
because we were looking for the involutions
on the graded algebra of figure \ref{fig_en:grad_M2C_dim1}.

\begin{example}\label{exam:grad_semisimple}
The following equation defines a division grading
on the algebra $ \mathbb{R} \times \mathbb{R} $
by the group $\mathbb{Z}_2$:
\begin{equation}
\mathbb{R} \times \mathbb{R} = \mathbb{R}(1,1) \oplus \mathbb{R}(1,-1)
\end{equation}
Besides,
figure \ref{fig_en:grad_semisimple} shows two division gradings
by the group $ \mathbb{Z}_2 \times \mathbb{Z}_4 =
\langle a,b \mid a^2=e=b^4 , \allowbreak \, ab=ba \rangle $
on the algebras $ M_2( \mathbb{R} \times \mathbb{R} ) $
and $ M_2(\mathbb{R}) \times \mathbb{H} $.
The degrees are assigned in the same order that we followed
in example \ref{exam:grad_M2C_dim1}:
$e$,~$a$; $b$,~$ab$; $b^2$,~$ab^2$; $b^3$,~$ab^3$.
\begin{figure}
\begin{align*}
M_2( \mathbb{R} \times \mathbb{R} ) = {} &
	\mathbb{R} \begin{pmatrix} (1,1) & (0,0) \\ (0,0) & (1,1) \end{pmatrix}
	\oplus
	\mathbb{R} \begin{pmatrix} (1,1) & (0,0) \\ (0,0) & (-1,-1) \end{pmatrix}
	\oplus
\\
&
	\mathbb{R} \begin{pmatrix} (0,0) & (1,-1) \\ (1,1) & (0,0) \end{pmatrix}
	\oplus
	\mathbb{R} \begin{pmatrix} (0,0) & (1,-1) \\ (-1,-1) & (0,0) \end{pmatrix}
	\oplus
\\
&
	\mathbb{R} \begin{pmatrix} (1,-1) & (0,0) \\ (0,0) & (1,-1) \end{pmatrix}
	\oplus
	\mathbb{R} \begin{pmatrix} (1,-1) & (0,0) \\ (0,0) & (-1,1) \end{pmatrix}
	\oplus
\\
&
	\mathbb{R} \begin{pmatrix} (0,0) & (1,1) \\ (1,-1) & (0,0) \end{pmatrix}
	\oplus
	\mathbb{R} \begin{pmatrix} (0,0) & (1,1) \\ (-1,1) & (0,0) \end{pmatrix}
\end{align*}
\begin{align*}
M_2(\mathbb{R}) \times \mathbb{H} = {} &
	\mathbb{R} \left( \begin{pmatrix} 1 & 0 \\ 0 & 1 \end{pmatrix} , 1 \right)
	\oplus
	\mathbb{R} \left( \begin{pmatrix} 0 & -1 \\ 1 & 0 \end{pmatrix} , i \right)
	\oplus
\\
&
	\mathbb{R} \left( \begin{pmatrix} 1 & 0 \\ 0 & -1 \end{pmatrix} , j \right)
	\oplus
	\mathbb{R} \left( \begin{pmatrix} 0 & 1 \\ 1 & 0 \end{pmatrix} , k \right)
	\oplus
\\
&
	\mathbb{R} \left( \begin{pmatrix} 1 & 0 \\ 0 & 1 \end{pmatrix} , -1 \right)
	\oplus
	\mathbb{R} \left( \begin{pmatrix} 0 & -1 \\ 1 & 0 \end{pmatrix} , -i \right)
	\oplus
\\
&
	\mathbb{R} \left( \begin{pmatrix} 1 & 0 \\ 0 & -1 \end{pmatrix} , -j \right)
	\oplus
	\mathbb{R} \left( \begin{pmatrix} 0 & 1 \\ 1 & 0 \end{pmatrix} , -k \right)
\end{align*}
\caption{Division gradings of example \ref{exam:grad_semisimple}.}
\label{fig_en:grad_semisimple}
\end{figure}
\end{example}

\section{Case (2-f)}

Finally let us consider the case (2-f),
in which the real algebra $\mathcal{D}$
is isomorphic to $M_n(\mathbb{C})$,
and the neutral component of the grading $\Gamma$
coincides with the center of $\mathcal{D}$.
This means that $\Gamma$ can be regarded
as a grading of the complex algebra $M_n(\mathbb{C})$.

In this situation,
the commutation relations of equation \eqref{eq:comm_rel}
are enough to determine the graded algebra $\mathcal{D}$.
The quadratic form $\mu$ of equation \eqref{eq:sign_squar}
is no longer defined.
Now $T$ is a group isomorphic to
$ \mathbb{Z}_{\ell_1}^2 \times \dots \times \mathbb{Z}_{\ell_r}^2 $,
where $ \ell_1 \cdots \ell_r = n $.
Whereas the alternating bicharacter $\beta$
can take values in $ \mathbb{C} \setminus \{ 0 \} $,
and the only condition that it has to satisfy is $ \rad(\beta) = \{ e \} $.

If $\varphi$ is an involution on the graded algebra $\mathcal{D}$,
then either $ \varphi(iI) = +iI $ or $ \varphi(iI) = -iI $.
The classification of the pairs $(\Gamma,\varphi)$
such that $ \varphi(iI) = +iI $
had already been studied in \cite[propositions 2.51 y 2.53]{EK2013},
because these involutions respect the complex structure of $\mathcal{D}$.
It is very similar to the classifications of the cases
that we have seen previously,
in which the homogeneous components have dimension $1$.

Let us focus on the involutions $\varphi$ such that $ \varphi(iI) = -iI $.
Their classification is totally different.
Moreover,
I think that section 8 of \cite{BKR2018a}
is one of the most interesting parts of the article;
by the way,
it can be read independently of the rest of the text.

We define $ T_{[2]} =  \{ t\in T \mid t^2 = e \} $.
The isomorphism classes of the pairs $(\Gamma,\varphi)$
are in bijective correspondence with
the subgroups $S$ of $T_{[2]}$ of index $1$ or $2$
via the following equation:
\begin{equation}
S = \{ t \in T_{[2]} \mid
\exists X \in \mathcal{D}_t
\text{ such that } X^2 = +I
\text{ and } \varphi(X) = X \}
\end{equation}

This is a good example of some of the difficulties
that appear in the mathematical research.
The proof is not so complicated,
the real challenge was to realize what we had to prove,
without even knowing if this was a solvable problem.
On the other hand,
we have to warn that the classification up to equivalence
is one of the most difficult parts of the article.

Note that there is a distinguished isomorphism class,
the one corresponding to $ S = T_{[2]} $.
The involutions that belong to this class
have signature $ \sqrt{ \vert T_{[2]} \vert } $,
whereas the rest of the involutions on the graded algebra $\mathcal{D}$
have signature $0$.

\selectlanguage{english}
\chapter{Third article}\label{ch:art_3}

\begin{enumerate}

\item[\cite{BKR2018b}]
Y. Bahturin, M. Kochetov\ and\ A. Rodrigo-Escudero,
Gradings on classical central simple real Lie algebras,
J. Algebra {\bf 506} (2018), 1--42.
\\
\url{https://arxiv.org/abs/1707.07909}
\\
\url{https://doi.org/10.1016/j.jalgebra.2018.02.036}

\end{enumerate}

\section{Objective}\label{sect:art_3_obj}

A Lie algebra $\mathcal{L}$ is
a (not necessarily associative) algebra
whose bilinear product
$ [ \cdot , \cdot ] : \mathcal{L} \times \mathcal{L} \to \mathcal{L} $
satisfies the following two properties:
\begin{enumerate}
\item Anticommutativity:
$ [x,x] = 0 $ for all $ x \in \mathcal{L} $.
\item Jacobi identity:
$ [[x,y],z] + [[y,z],x] + [[z,x],y] = 0 $
for all $ x,y,z \in \mathcal{L} $.
\end{enumerate}

Killing (1890) and Cartan \cite{Car1894} classified
the semisimple (fi\-nite-di\-men\-sional) complex Lie algebras
thanks to a grading:
the root space decomposition
relative to a Cartan subalgebra.
There appear four infinite families of simple Lie algebras,
$\{A_r\}_{r=1}^{\infty}$,
$\{B_r\}_{r=2}^{\infty}$,
$\{C_r\}_{r=3}^{\infty}$ and
$\{D_r\}_{r=4}^{\infty}$,
called classical,
and five exceptional simple Lie algebras,
$E_6$, $E_7$, $E_8$, $F_4$ and $G_2$.
On the other hand,
in \cite{Wed1908} Wedderburn solved
the analogous problem in the associative case.
Although the result is easier,
all the simple (finite-dimensional)complex associative algebras
belong to the same infinite family
$ \{ M_n(\mathbb{C}) \}_{n=1}^{\infty} $,
it was obtained 14 years after that of Cartan.
This is clearly a good example
of the importance that Lie algebras have
in mathematics and physics.

We are interested in
simple (finite-dimensional) real Lie algebras.
The centroid of one such algebra $\mathcal{L}$
is either isomorphic to $\mathbb{R}$ or to $\mathbb{C}$.
In the second case $\mathcal{L}$ is simply
an algebra over the field of complex numbers,
but considered as an algebra over the field of real numbers.
In the first case $\mathcal{L}$ is a real form of its complexification
$ \mathbb{C} \otimes_{\mathbb{R}} \mathcal{L} $,
and it is said that $\mathcal{L}$ is central simple.
The type of a simple real Lie algebra
is that of the complex algebra from which it comes.
Thus,
both the complex algebra $E_6$,
regarded as a real algebra,
and its real forms are algebras of type $E_6$.

Recall that our objective is not to classify Lie algebras,
because these have already been extensively studied,
but their possible gradings.
Concretely,
the result of the third article \cite{BKR2018b} is the following.

\begin{objective}\label{obj_en:art_3}
For any abelian group $G$,
we classify up to isomorphism all $G$-gradings
on the classical central simple Lie algebras,
except those of type $D_4$,
over the field of real numbers
(or any real closed field).
\end{objective}

Let us remark that it is a very technical article,
in which we need a lot of parameters in order to state the theorems.
The main difference with respect to
the analogous classification over the complex field
(see \cite{BK2010} and \cite{EK2013})
is the presence of a new parameter $\sigma$ of signatures.

\section{Associative algebras with involution}\label{sect:alg_inv}

Let us see how
the classification of gradings on Lie algebras
is reduced to
the classification of gradings on associative algebras with involution.

Let $\mathcal{R}$ be an associative algebra,
with product denoted by juxtaposition.
We consider the new product in $\mathcal{R}$
given by $ [x,y] = xy-yx $ for all $ x,y \in \mathcal{R} $.
Then $\mathcal{R}$,
endowed with this new product,
is a Lie algebra,
which we denote by $\mathcal{R}^{(-)}$.
If besides $\varphi$ is an involution on $\mathcal{R}$,
we can consider the space of skewsymmetric elements:
\begin{equation}
\Skew(\mathcal{R},\varphi) = \{ x \in \mathcal{R} \mid \varphi(x) = -x \}
\end{equation}
We have that
$\Skew(\mathcal{R},\varphi)^{(-)}$ is a subalgebra
of the Lie algebra $\mathcal{R}^{(-)}$.
It is well known that the
classical central simple real Lie algebras
can be expressed by means of this construction.
For instance,
those of type $C_r$ are given by taking
either $ \mathcal{R} = M_{2r}(\mathbb{R}) $
or $ \mathcal{R} = M_r(\mathbb{H}) $
and $\varphi$ symplectic.

Any grading on $(\mathcal{R},\varphi)$ induces
a grading on $ \mathcal{L} = \Skew(\mathcal{R},\varphi)^{(-)} $
by restriction.
The theory of affine group schemes allows to prove that
this correspondence is actually a bijection
between the gradings by abelian groups on $(\mathcal{R},\varphi)$
and the gradings on $\mathcal{L}$
(except if $\mathcal{L}$ is of type $A_1$ or $D_4$).
If we fix an abelian group $G$,
we obtain a bijection between the respective
isomorphism classes of $G$-gradings.

This is the reason why in the previous articles
we classified gradings on associative algebras.
Of course those classifications are of independent interest,
but our motivation was to use them in the study of Lie algebras.
Besides,
now we can explain why we have always assumed that
the grading groups are abelian:
the support of a grading on a simple Lie algebra
always generates an abelian subgroup of the grading group
(see for example \cite[lemma 2.1]{BZ2006} or \cite[proposition 1]{DM2006}).

\section{Summary}

In this section we are going to give a brief idea
of the ingredients that appear in the classification.

We want to characterize up to $G$-isomorphism
the graded algebra with involution $(\mathcal{R},\varphi)$.
Let us forget about $\varphi$ for the moment.
We write $\mathcal{R}$ as $M_k(\mathcal{D})$,
in the same way as in section \ref{sect:state_art}.
Recall that,
as a $G$-graded algebra,
$M_k(\mathcal{D})$ was defined by
the graded division algebra $\mathcal{D}$
and $k$ elements of the abelian group $G$.
In the first article \cite{Rod2016} we classified the possible $\mathcal{D}$.
Call $T$ the support of $\mathcal{D}$
and fix in each homogeneous component $\mathcal{D}_t$
a homogeneous element $X_t$.
Curiously,
in some cases the $X_e$ that we choose is not the unit
(this happens if $\mathcal{L}$ is
$\mathfrak{sp}_{2r}(\mathbb{R})$ or $ \mathfrak{u}^* (r) $,
and $ \mathcal{D}_e \cong \mathbb{C} $).
Actually,
the isomorphism class of $M_k(\mathcal{D})$
does not change if we permute the $k$ elements of $G$,
or if we multiply any of them by an element of $T$.
That is,
the information is contained in a ``multiplicity function''
$ \kappa : G/T \to \mathbb{Z}_{ \geq 0 } $.
If we enumerate the nonzero multiplicities,
$ \kappa(x_i) = k_i $ with $ 1 \leq i \leq s $,
then their total sum is $ k_1 + \dots + k_s = k $.

Let us now take into account the involution $\varphi$.
After some computations we see that we can write it as
$ \varphi(X) = \Phi^{-1} \varphi_0(X^T) \Phi $,
where $\Phi$ is a block diagonal matrix of $M_k(\mathcal{D})$
and $\varphi_0$ is an involution of $\mathcal{D}$,
which acts on $ X^T \in M_k(\mathcal{D}) $ elementwise.
Recall that we classified the possible $\varphi_0$
in the second article \cite{BKR2018a}.
In a lot of cases we can choose
$\varphi_0$ to be the distinguished involution.
The blocks of the matrix $\Phi$ are of the form $ X_{t_i} S_i $,
where the $S_i$ are matrices of size $ k_i \times k_i $.
As an example of the possible values that $S_i$ can take
we may mention the following two matrices;
in the first one $ k_i = p_i + q_i $,
and in the second one $k_i$ has to be even:
\begin{equation*}
\begin{pmatrix} I_{p_i} & 0 \\ 0 & - I_{q_i} \end{pmatrix}
\qquad \qquad
\begin{pmatrix} 0 & I_{ k_i / 2 } \\ - I_{ k_i / 2 } & 0 \end{pmatrix}
\end{equation*}
We define a ``signature function'' $ \sigma : G/T \to \mathbb{Z} $
as $ \sigma(x_i) = p_i - q_i $ for the values of $i$ with signature,
and $ \sigma(x) = 0 $ for all the other $ x \in G/T $.
Besides,
we need two other ingredients in order to describe $\varphi$,
a parameter $ \delta \in \{ \pm 1 \} $
given by $ \varphi_0(\Phi^T) = \delta \Phi $,
and an element $ g_0 \in G $ that in a certain sense collects the degree of $\Phi$.
In total,
the graded algebra with involution $(\mathcal{R},\varphi)$
is given by six parameters:
$\mathcal{D}$, $\varphi_0$, $g_0$, $\kappa$, $\sigma$, $\delta$.

The next step is to consider two of these graded algebras with involution,
$ (\mathcal{R},\varphi) = M( \mathcal{D} , \allowbreak
\varphi_0 , \allowbreak g_0 , \allowbreak
\kappa , \allowbreak \sigma , \allowbreak \delta ) $ and
$ (\mathcal{R}',\varphi') = M( \mathcal{D}' , \allowbreak
\varphi_0' , \allowbreak g_0' , \allowbreak
\kappa' , \allowbreak \sigma' , \allowbreak \delta' ) $,
and determine when they are isomorphic.
In order to do so we can suppose that $ \mathcal{D} = \mathcal{D}' $,
$ \varphi_0 = \varphi_0' $ and $ \delta = \delta' $,
but because of different reasons.
On the one hand,
if the graded division algebras $\mathcal{D}$ and $\mathcal{D}'$
are not isomorphic,
then $\mathcal{R}$ and $\mathcal{R}'$ are not isomorphic either.
Therefore the total classification will be the sum of the classifications
as $\mathcal{D}$ runs through all the possible isomorphism classes.
On the other hand,
we can change the involution $\varphi_0$ and choose the one that suits us best
(if it is possible we choose the distinguished involution of $\mathcal{D}$).
If we chose another one,
we would obtain the same classification,
but parametrized in a different way.
Finally we can also assume that $\delta$ is fixed;
curiously if $\mathcal{L}$ is of type $A_r$,
this is due to the second reason,
whereas if $\mathcal{L}$ is of type $B_r$, $C_r$ or $D_r$,
it is because of the first reason.

The issue is in the parameters $(g_0,\kappa,\sigma)$.
The signature function $\sigma$ is not the adequate invariant
to solve the isomorphism condition,
since the sign of $ p_i - q_i $
depends on the representative that we take
in the coset $ x_i \in G/T $.
In order to avoid this ambiguity,
we consider a new parameter,
which we call ``extended signature function''
$ \tilde{\sigma} : G \to \mathbb{Z} $,
and which is equivalent to $\sigma$.
The result is presented by means of
the action of a group that acts on the the parameters,
so that $(\mathcal{R},\varphi)$ and $(\mathcal{R}',\varphi')$
are isomorphic if and only if
$(g_0,\kappa,\sigma)$ and $(g_0',\kappa',\sigma')$
are in the same orbit.
For instance,
in the types $B_r$, $C_r$ and $D_r$
this action consists of the sum of
the natural action of $ g \in G $ on $\kappa$ and $\tilde{\sigma}$
(which also replaces $g_0$ with $ g^{-2} g_0 $),
and the action of changing the sign of all the signatures.

Finally we particularize the classification
to each of the types of
the classical central simple real Lie algebras.
In some cases we have to compute
the global signature of the involution $\varphi$
depending on the parameters,
what by the way gives us a beautiful formula.

\selectlanguage{english}
\chapter{Fourth article}\label{ch:art_4}

\begin{enumerate}

\item[\cite{ER2018}]
A. Elduque\ and\ A. Rodrigo-Escudero,
Clifford algebras as twisted gro\-up algebras and the Arf invariant,
Adv. Appl. Clifford Algebr. {\bf 28} (2018), no.~2, 28:41.
\\
\url{https://arxiv.org/abs/1801.07002}
\\
\url{https://doi.org/10.1007/s00006-018-0862-y}

\end{enumerate}

\section{Introduction}\label{sect:art_4_intr}

The purpose of the first three articles of my thesis
has been to study and classify gradings on real algebras.
Now we tackle a more practical topic;
the fourth article \cite{ER2018}
is an example of the applications of gradings
to solve problems in other areas of mathematics.

While we classified division gradings on associative algebras,
we realized that these are intimately related
with Clifford algebras \cite[remark 17]{Rod2016}.
So after achieving the initial goal of the doctorate,
we decided to return to analyzing in more detail this connection.
The result is an article that is more relaxed than the previous three,
since the theorems that appear were already known.
As we are going to see,
thanks to the gradings we can prove them
in an alternative and very easy way,
in which the Arf invariant plays a key role.

This article possesses
a very unusual peculiarity among mathematical articles:
it can be explained in a talk.
In fact this chapter is an adaptation
of a lecture that I gave in february 2018
at the seminar that we carry out among the doctoral students.

\begin{objective}\label{obj_en:art_4}
In this chapter we will show how the theory of gradings
allows to give alternative proofs to classical theorems.
Thus,
we will review the basic properties of Clifford algebras
and we will see that,
in the real case,
these are determined by the Arf invariant.
\end{objective}

\section{Division gradings on Clifford algebras}

We start by recalling the definition of Clifford algebra.
Let $V$ be a vector space of finite dimension $N$
over a field $\mathbb{F}$ of characteristic different from $2$.
Later on we will restrict ourselves to the case $ \mathbb{F} = \mathbb{R} $,
but for the moment the arguments are valid in general.
Let $ Q : V \to \mathbb{F} $ be a quadratic form on $V$.
This is equivalent to say that there exists a symmetric bilinear form
$ B : V \times V \to \mathbb{F} $ satisfying,
for all $ v \in V $,
the following equation:
\begin{equation}
2 Q(v) = B(v,v)
\end{equation}
We can recover $B$ from $Q$,
because for all $ u,v \in V $ we have the following formula:
\begin{equation}\label{eq:bil_form_quadr_form}
B(u,v) = Q(u+v) - Q(u) - Q(v)
\end{equation}
Note that equation \eqref{eq:polarization}
is nothing more than equation \eqref{eq:bil_form_quadr_form}
written with multiplicative notation.

Let $ T(V) = \mathbb{F} \oplus V \oplus ( V \otimes V)
\oplus ( V \otimes V \otimes V ) \oplus \dots $
be the tensor algebra of $V$,
and let $I(Q)$ be the two-sided ideal of $T(V)$
generated by all the elements of the form
$ v \otimes v - Q(v)1 $ with $ v \in V $.
The Clifford algebra of $(V,Q)$ is $ \Cl(V,Q) = T(V) / I(Q) $.
Thus,
$\Cl(V,Q)$ is an associative algebra over the field $\mathbb{F}$,
but at the moment we do not know what is its dimension.
In the next two paragraphs we are going to prove that
the algebra $\Cl(V,Q)$ is unital,
that is,
there exists an element that is different from zero and neutral for the product,
which we denote $ 1 \in \Cl(V,Q) $ as always.
This question is equivalent to prove that
$ 1 \in T(V) $ does not belong to $I(Q)$,
and it is a bit more complicated than it seems.
As a corollary we obtain that the dimension of the algebra $\Cl(V,Q)$
is strictly greater than $0$.

First we observe that the Clifford algebra
satisfies the following universal property:
For any unital associative $\mathbb{F}$-algebra $\mathcal{A}$
endowed with an $\mathbb{F}$-linear map $ f : V \to \mathcal{A} $
such that $ f(v)^2 = Q(v) 1 $ for all $ v \in V $,
there exists a unique homomorphism of unital $\mathbb{F}$-algebras
$ \varphi : \Cl(V,Q) \to \mathcal{A} $ such that
$ f(v) = \varphi(v+I(Q)) $ for all $ v \in V $.
The existence of such a unital algebra $\mathcal{A}$
will imply that $ 1 \in T(V) $ does not belong to $I(Q)$,
since $\varphi(1+I(Q))$ is the unit of $\mathcal{A}$,
which is different from the zero of $\mathcal{A}$.

Now we are going to see,
following \cite[page 38]{Che1954},
that we can take as $\mathcal{A}$
the algebra of endomorphisms of the exterior algebra $\Lambda(V)$.
Let $ B_0 : V \times V \to \mathbb{F} $
be the symmetric bilinear form given by $ 2 B_0 = B $
(hence $ Q(v) = B_0(v,v) $ for all $ v \in V $).
For all $ v \in V $,
we consider the endomorphisms $L_v$ and $\delta_v$ of $\Lambda(V)$,
where $L_v$ is the left multiplication by $v$,
and $\delta_v$ is the antiderivation corresponding
to the linear form $B_0(v,\cdot)$.
That is,
$\delta_v$ is defined by $ \delta_v(1) = 0 $ and the following equation:
\begin{equation*}
\delta_v ( u_1 \wedge \dots \wedge u_k )
= \sum_{i=1}^{k} (-1)^{i-1} B_0(v,u_i)
u_1 \wedge \dots \wedge u_{i-1}
\wedge u_{i+1} \wedge \dots \wedge u_k
\end{equation*}
For all $ u \in V $ and all $ x \in \Lambda(V) $ we have that
$ \delta_v ( u \wedge x ) = B_0(v,u) x - u \wedge \delta_v(x) $.
From this formula we deduce that $ \delta_v^2 = 0 $
and that $ L_u \delta_v + \delta_v L_u = B_0(v,u) 1 $.
Therefore $ ( L_v + \delta_v )^2 = Q(v)1 $,
and we have constructed our $\mathcal{A}$.
This concludes the proof that $\Cl(V,Q)$ is a unital algebra.

By the definition of Clifford algebra,
in $\Cl(V,Q)$ we have that,
for all $ v \in V $,
the following formula is satisfied:
\begin{equation}\label{eq:Cl_squar}
v^2 = Q(v)
\end{equation}
Therefore $ uv + vu = B(u,v) $ for all $ u,v \in V $.
Let $ v_1 , \dots , v_N $ be an orthogonal basis of $V$,
for all $ i \neq j $ we have the following equation:
\begin{equation}\label{eq:Cl_anticomm}
v_i v_j = - v_j v_i
\end{equation}
If besides we have that $ Q(v_1) = \dots = Q(v_N) = 1 $,
we denote the Clifford algebra $\Cl_N(\mathbb{F})$.
In the case $ \mathbb{F} = \mathbb{R} $,
if $ Q(v_1) = \dots = Q(v_p) = +1 $ and
$ Q(v_{p+1}) = \dots = Q(v_N) = -1 $,
we denote the Clifford algebra $\Cl_{p,q}(\mathbb{R})$,
where $ p+q = N $.

Call $ v_I = v_{i_1} \cdots v_{i_r} $,
where $ 1 \leq i_1 < \dots < i_r \leq N $
and $ I = \{ i_1 , \allowbreak \dots , \allowbreak i_r \} $,
and also denote $ v_{\emptyset} = 1 $.
Due to equations \eqref{eq:Cl_squar} and \eqref{eq:Cl_anticomm},
$\Cl(V,Q)$ is linearly spanned by
$ \{ v_I \}_{ I \subseteq \{ 1,\dots,N \} } $.
Hence the dimension of $\Cl(V,Q)$ is at most $2^N$.
In fact $ \dim \Cl(V,Q) = 2^N $,
but the proof,
in the same way as the proof that $\Cl(V,Q)$ is a unital algebra,
is not trivial,
see for instance \cite[II.1.2]{Che1954} or \cite[theorem V.1.8]{Lam2005}.
Thanks to the following grading we are going to be able to give
an alternative and easy proof.

\begin{example}\label{exam:Cl_grad}
We can endow the Clifford algebra $\Cl(V,Q)$
with a grading by the abelian group $\mathbb{Z}_2^N$,
as in \cite[proposition 2.2]{AM2002}.
First we define a $\mathbb{Z}_2^N$-grading
on the tensor algebra $T(V)$
by means of the following assignment:
\begin{equation}
\deg v_i = ( \bar{0},\dots,\bar{0}, \overbrace{\bar{1}}^i
,\bar{0},\dots,\bar{0} ) \in \mathbb{Z}_2^N
\end{equation}
The ideal $I(Q)$ is generated by the elements of the form
$ v_i \otimes v_j + v_j \otimes v_i - B(v_i,v_j) $,
which are homogeneous because our basis $ v_1 , \dots , v_N $ is orthogonal,
hence it is a graded ideal.
Therefore $\Cl(V,Q)$ inherits the $\mathbb{Z}_2^N$-grading of $T(V)$
by means of $ \Cl(V,Q)_g = T(V)_g + I(Q) $.
\end{example}

\begin{proposition}
The dimension of a Clifford algebra $\Cl(V,Q)$ is $2^N$,
where $N$ is the dimension of the vector space $V$:
\begin{equation}
\dim \Cl(V,Q) = 2^N
\end{equation}
\end{proposition}

\begin{proof}
Since the $v_I$ belong to different homogeneous components,
they are linearly independent as long as they are different from zero.
Suppose first that the quadratic form $Q$ is nondegenerate,
that is,
$ Q(v_i) \neq 0 $ for all $ 1 \leq i \leq N $.
Then $ v_I^2 = v_{i_1} \cdots v_{i_r} v_{i_1} \cdots v_{i_r}
= \pm Q(v_{i_1}) \cdots Q(v_{i_r}) $.
This last term is different from zero,
because we have already seen that $ 1 \neq 0 $ in $\Cl(V,Q)$,
therefore $ v_I \neq 0 $ for all
$ I \subseteq \{ 1,\dots,N \} $.

On the other hand,
let us allow now some of the $Q(v_i)$ to be zero.
We consider the $\mathbb{F}$-algebra $\mathcal{A}$:
\begin{equation*}
\mathcal{A} = \mathbb{F}[X_1] / (X_1^2-Q(v_1)) \otimes \dots \otimes
\mathbb{F}[X_N] / (X_N^2-Q(v_N)) \otimes \Cl_N(\mathbb{F})
\end{equation*}
Note that the bilinear form $B'$ of $\Cl_N(\mathbb{F})$
is different from the bilinear form $B$ of $\Cl(V,Q)$.
Let $ u_1 , \dots , u_N $ be an orthonormal basis of $(V,B')$;
we have already proved that
$ \{ u_I \}_{ I \subseteq \{ 1,\dots,N \} } $
is a linear basis of $\Cl_N(\mathbb{F})$.
Let $ f : V \to \mathcal{A} $ be the linear map
that sends $v_i$ to the following element $w_i$:
\begin{equation*}
w_i = 1 \otimes \dots \otimes 1 \otimes \overbrace{X_i}^i
\otimes 1 \otimes \dots \otimes 1 \otimes u_i \in \mathcal{A}
\end{equation*}
We have that $ w_i w_j + w_j w_i = B(v_i,v_j) 1 $,
hence we can apply the universal property to obtain
a homomorphism of unital $\mathbb{F}$-algebras
from $\Cl(V,Q)$ to $\mathcal{A}$ that sends $v_i$ to $w_i$.
We conclude that $ v_I = v_{i_1} \cdots v_{i_r} $ cannot be zero,
because its image $ w_{i_1} \cdots w_{i_r} $ is different from zero,
and therefore $ \dim \Cl(V,Q) = 2^N $.
\end{proof}

Note that if the quadratic form $Q$ is nondegenerate,
then the $\mathbb{Z}_2^N$-grading on $\Cl(V,Q)$
of example \ref{exam:Cl_grad}
is a division grading and its homogeneous components have dimension $1$.

\begin{proposition}
The dimension of the center of a Clifford algebra $\Cl(V,Q)$
is given by the following equation,
where $N$ is the dimension of the vector space $V$:
\begin{equation}
\dim Z( \Cl(V,Q) ) =
\begin{cases}
	1 & \text{if } N \text{ is even} \\
	2 & \text{if } N \text{ is odd}
\end{cases}
\end{equation}
\end{proposition}

\begin{proof}
The grading of example \ref{exam:Cl_grad} is useful again,
since it allows us to skip a step in the computation.
Indeed,
$Z(\Cl(V,Q))$ is a graded subspace because
the grading group $\mathbb{Z}_2^N$ is abelian
($ x = \sum_{g \in G} x_g \in Z(\mathcal{D}) $
if and only if $ x y_h = y_h x $
for all $ h \in G $ and all $ y_h \in \mathcal{D}_h $;
if and only if $ x_g y_h = y_h x_g $
for all $ g,h \in G $ and all $ y_h \in \mathcal{D}_h $;
if and only if $ x_g \in Z(\mathcal{D}) $
for all $ g \in G $),
therefore we only have to check which $v_I$ are in the center.
If $v_i$ is part of the product
$ v_I = v_{i_1} \cdots v_{i_r} $ but $v_j$ is not,
then one of them commutes with $v_I$ and the other anticommutes.
Hence the only candidates to belong to the center
are $1$ and $ v_1 \cdots v_N $.
This last element is central if and only if $N$ is odd.
\end{proof}

\section{Arf invariant}

From now on we will focus on the case $ \mathbb{F} = \mathbb{R} $.
Given a real number $x$ different from zero,
we denote by $\sign(x)$ its sign,
which can take the values $+1$ or $-1$.
Besides,
we define $ \sign(0) = 0 $.
Recall that
the Arf invariant of a map
$ \mu : T \to \{ \pm 1 \} $
defined on a finite set $T$
is the value which is assumed most often by the map:
$ \Arf(\mu) = \sign (
\vert \mu^{-1} (+1) \vert -
\vert \mu^{-1} (-1) \vert
) $.
One of the consequences of the classifications
of chapters \ref{ch:art_1} and \ref{ch:art_2}
is the following theorem.

\begin{theorem}\label{th:Cl_Arf_isom}
Let $\mathcal{D}$ be a finite-dimensional real associative algebra
whose center $Z(\mathcal{D})$ has dimension $1$ or $2$.
Suppose that we can endow $\mathcal{D}$
with a division grading
with homogeneous components of dimension $1$
and whose support $T$ is a group isomorphic to $\mathbb{Z}_2^N$.
Then equation \eqref{eq:sign_squar}
defines a quadratic form $ \mu : T \to \{ \pm 1 \} $
whose Arf invariant
determines the real algebra $\mathcal{D}$ up to isomorphism,
according to the following list.
\begin{itemize}
	\item If $ \dim Z (\mathcal{D}) = 1 $
	($ \Leftrightarrow N = 2m $),
	then:
	\begin{itemize}
		\item $ \Arf(\mu) = +1 $
		implies that
		$ \mathcal{D} \cong M_{2^m} (\mathbb{R}) $.
		\item $ \Arf(\mu) = -1 $
		implies that
		$ \mathcal{D} \cong M_{2^{m-1}} (\mathbb{H}) $.
	\end{itemize}
	\item If $ \dim Z (\mathcal{D}) = 2 $
	($ \Leftrightarrow N = 2m+1 $),
	then:
	\begin{itemize}
		\item $ \Arf(\mu) = +1 $
		implies that
		$ \mathcal{D} \cong M_{2^m} (\mathbb{R})
		\times M_{2^m} (\mathbb{R}) $.
		\item $ \Arf(\mu) = 0 $
		implies that
		$ \mathcal{D} \cong M_{2^m} (\mathbb{C}) $.
		\item $ \Arf(\mu) = -1 $
		implies that
		$ \mathcal{D} \cong M_{2^{m-1}} (\mathbb{H})
		\times M_{2^{m-1}} (\mathbb{H}) $.
	\end{itemize}
\end{itemize}
\end{theorem}

The grading of example \ref{exam:Cl_grad}
allows us to apply theorem \ref{th:Cl_Arf_isom}
to the Clifford algebra $\Cl_{p,q}(\mathbb{R})$.
Denote by $\mu_{p,q}$ the corresponding quadratic form.
The objective of this section is to compute
the Arf invariant of $\mu_{p,q}$
depending on the values of $p$ and $q$.

\begin{theorem}\label{th:Cl_period}
For all $ p,q \in \mathbb{N} \cup \{0\} $
we have the following isomorphisms of real algebras.
\begin{enumerate}
	\item $ \Cl_{p+1,q+1}(\mathbb{R})
	\cong \Cl_{p,q}(\mathbb{R})
	\otimes M_2(\mathbb{R}) $.
	\item $ \Cl_{p+2,q}(\mathbb{R})
	\cong \Cl_{q,p}(\mathbb{R})
	\otimes M_2(\mathbb{R}) $.
	\item $ \Cl_{p,q+2}(\mathbb{R})
	\cong \Cl_{q,p}(\mathbb{R})
	\otimes \mathbb{H} $.
\end{enumerate}
\end{theorem}

\begin{proof}
By theorem \ref{th:Cl_Arf_isom},
the first isomorphism is equivalent to
$ \Arf(\mu_{p+1,q+1}) = \Arf(\mu_{p,q}) $.
We identify
$ \Cl_{p,q}(\mathbb{R}) \subseteq \Cl_{p+1,q+1}(\mathbb{R}) $
so that
$ v_1^2 = \dots = v_p^2 = +1 $,
$ v_{p+1}^2 = \dots = v_{p+q}^2 = -1 $,
$ v_{N+1}^2 = +1 $ and $ v_{N+2}^2 = -1 $,
where $ N = p+q $.
In order to find the Arf invariant of $\mu_{p,q}$
we compute $v_I^2$ for all $ I \subseteq \{ 1,\dots,N \} $
and we count the difference between
the number of $+1$ and the number of $-1$.
Analogously for $\Arf(\mu_{p+1,q+1})$,
but now we have four types of terms
as $I$ runs through the subsets of $ \{ 1,\dots,N \} $:
\begin{equation*}
\text{(1) the } v_I
\text{; (2) the } v_I v_{N+1}
\text{; (3) the } v_I v_{N+2}
\text{; (4) the } v_I v_{N+1} v_{N+2} \text{.}
\end{equation*}
Since $ ( v_I v_{N+1} )^2 = -( v_I v_{N+2} )^2 $,
the terms of the second type cancel out with
the terms of the third type.
Besides,
$ v_{N+1} v_{N+2} $ commutes with $v_I$
and $ ( v_{N+1} v_{N+2} )^2 = +1 $,
hence $ ( v_I v_{N+1} v_{N+2} )^2 = v_I^2 $,
and the contribution of the terms of the fourth type
is the same as that of the terms of the first type,
which is given by $\Arf(\mu_{p,q})$.

For the second isomorphism,
let $ u_1 , \dots , u_N $ and
$ v_1 , \dots , v_N , v_{N+1} , v_{N+2} $
be generating systems of
$\Cl_{q,p}(\mathbb{R})$ and $\Cl_{p+2,q}(\mathbb{R})$
satisfying equation \eqref{eq:Cl_anticomm}
and such that $ v_{N+1}^2 = v_{N+2}^2 = +1 $
and $ u_i^2 = - v_i^2 $ for all $ 1 \leq i \leq N $.
We have the same four types of terms
in $\Cl_{p+2,q}(\mathbb{R})$ as before.
This time $ ( v_{N+1} v_{N+2} )^2 = -1 $,
so the terms of the fourth type
cancel out with the terms of the first type.
Besides,
$ ( v_I v_{N+1} )^2 = ( v_I v_{N+2} )^2 $,
hence the contribution of the terms of the third type
is the same as that of the terms of the second type.
If $ \vert I \vert $ is odd,
then $ ( v_I v_{N+1} )^2 = - v_I^2 v_{N+1}^2 = u_I^2 $;
whereas if $ \vert I \vert $ is even,
also $ ( v_I v_{N+1} )^2 = v_I^2 v_{N+1}^2 = u_I^2 $.
We conclude that $ \Arf(\mu_{p+2,q}) = \Arf(\mu_{q,p}) $,
which implies that $ \Cl_{p+2,q}(\mathbb{R})
\cong \Cl_{q,p}(\mathbb{R}) \otimes M_2(\mathbb{R}) $.

The proof of the third isomorphism is analogous to that of the second,
but in this case $ v_{N+1}^2 = v_{N+2}^2 = -1 $,
therefore we obtain
$ \Arf(\mu_{p,q+2}) = - \Arf(\mu_{q,p}) $.
\end{proof}

\begin{lemma}\label{lemm:Cl_Arf}
$ \Arf(\mu_{p,0}) =
\sign (
\cos ( p \pi / 4 ) + \allowbreak
\sin ( p \pi / 4 )
) $
for all $ p \in \mathbb{N} \cup \{0\} $,
and
$ \Arf(\mu_{0,q}) =
\sign (
\cos ( - q \pi / 4 ) + \allowbreak
\sin ( - q \pi / 4 )
) $
for all $ q \in \mathbb{N} \cup \{0\} $.
\end{lemma}

\begin{proof}
Let us start with the case in which $p=N$ and $q=0$.
Since $ v_1^2 = \dots = v_N^2 = +1 $,
the value of $v_I^2$ only depends on $ \vert I \vert = r $.
Specifically,
$ v_I^2 = v_{i_1} \cdots v_{i_r} v_{i_1} \cdots v_{i_r}
= (-1)^{\binom{r}{2}} $.
Thus,
we can count the difference between
the number of $+1$ and the number of $-1$:
\begin{equation*}
\vert \mu_{N,0}^{-1} (+1) \vert -
\vert \mu_{N,0}^{-1} (-1) \vert
= \sum_{r=0}^{N} \binom{N}{r} (-1)^{(r-1)r/2}
= S_0 + S_1 - S_2 - S_3
\end{equation*}
In the last step we have abbreviated
the writing of the binomial sums
with the notation of lemma \ref{lemm:bin_sum}
($S_0$, $S_1$, $S_2$ and $S_3$).
Precisely applying that lemma
we obtain the desired formula for $\Arf(\mu_{p,0})$.

The case in which $p=0$ and $q=N$ is analogous,
but now $ v_I^2 = (-1)^{\binom{r}{2}} (-1)^r $.
Hence we obtain the following equation:
\begin{equation*}
\vert \mu_{0,N}^{-1} (+1) \vert -
\vert \mu_{0,N}^{-1} (-1) \vert
= \sum_{r=0}^{N} \binom{N}{r} (-1)^{(r-1)r/2+r}
= S_0 - S_1 - S_2 + S_3
\end{equation*}
\end{proof}

Let us compute,
by means of an algebraic argument,
the binomial sums $S_0$, $S_1$, $S_2$ and $S_3$
that we have used in the proof of lemma \ref{lemm:Cl_Arf}.

\begin{lemma}\label{lemm:bin_sum}
For any integer number $N$ greater than or equal to $1$
the following binomial formulas are satisfied:
\begin{equation*}
S_0 := \binom{N}{0} + \binom{N}{4} + \dots
+ \binom{N}{ 4 \lfloor \frac{N-0}{4} \rfloor + 0 }
= \frac{1}{2} \left( 2^{N-1} + 2^{N/2}
\cos \frac{N\pi}{4} \right)
\end{equation*}
\begin{equation*}
S_1 := \binom{N}{1} + \binom{N}{5} + \dots
+ \binom{N}{ 4 \lfloor \frac{N-1}{4} \rfloor + 1 }
= \frac{1}{2} \left( 2^{N-1} + 2^{N/2}
\sin \frac{N\pi}{4} \right)
\end{equation*}
\begin{equation*}
S_2 := \binom{N}{2} + \binom{N}{6} + \dots
+ \binom{N}{ 4 \lfloor \frac{N-2}{4} \rfloor + 2 }
= \frac{1}{2} \left( 2^{N-1} - 2^{N/2}
\cos \frac{N\pi}{4} \right)
\end{equation*}
\begin{equation*}
S_3 := \binom{N}{3} + \binom{N}{7} + \dots
+ \binom{N}{ 4 \lfloor \frac{N-3}{4} \rfloor + 3 }
= \frac{1}{2} \left( 2^{N-1} - 2^{N/2}
\sin \frac{N\pi}{4} \right)
\end{equation*}
\end{lemma}

\begin{proof}
We consider the following isomorphism $\varphi$
of $\mathbb{R}$-algebras
given by the chinese remainder theorem:
\begin{align*}
\varphi : \mathbb{R}[T] / (T^4-1)
& \longrightarrow \mathbb{R}[X] / (X-1)
\times \mathbb{R}[Y] / (Y+1)
\times \mathbb{R}[Z] / (Z^2+1)
\\ f(T)+(T^4-1)
& \longmapsto ( f(X)+(X-1)
, f(Y)+(Y+1) , f(Z)+(Z^2+1) )
\end{align*}
If we write $ i = Z+(Z^2+1) $
and $ t = T+(T^4-1) $,
then $\varphi$ is defined by
$ \varphi(1) = (1,1,1) $,
$ \varphi(t) = (1,-1,i) $,
$ \varphi(t^2) = (1,1,-1) $ and
$ \varphi(t^3) = (1,-1,-i) $.
Reciprocally,
$\varphi^{-1}$ is determined by
$ \varphi^{-1}(1,0,0) = (1+t+t^2+t^3)/4 $,
$ \varphi^{-1}(0,1,0) = (1-t+t^2-t^3)/4 $,
$ \varphi^{-1}(0,0,1) = (1-t^2)/2 $ and
$ \varphi^{-1}(0,0,i) = (t-t^3)/2 $.

On the one hand,
we can compute $(1+t)^N$ in $\mathbb{R}[t]$:
\begin{equation*}
(1+t)^N = \sum_{r=0}^N \binom{N}{r} t^r
= S_0 1 + S_1 t + S_2 t^2 + S_3 t^3
\end{equation*}
On the other hand,
we can apply first $\varphi$ to $(1+t)$,
which gives $(2,0,1+i)$,
and raise the result to the $N$-th power:
\begin{equation*}
(2,0,1+i)^N = (2^N,0,(1+i)^N) = \left( 2^N,0,
2^{N/2} \cos \frac{N\pi}{4} +
2^{N/2} \sin \frac{N\pi}{4} i \right)
\end{equation*}
Applying $\varphi^{-1}$ to this last expression,
and comparing it with the other one that we have,
we obtain the desired formulas.
\end{proof}

Putting together theorem \ref{th:Cl_period} and lemma \ref{lemm:Cl_Arf}
we obtain the final theorem of this chapter.
Of course,
equation \eqref{eq:Cl_Arf_pq_mod} is well known,
although neither my supervisor nor I had seen it written before
in the form of equation \eqref{eq:Cl_Arf_pq_trigon}.

\begin{theorem}
For all $ p,q \in \mathbb{N} \cup \{0\} $
the Arf invariant of $\mu_{p,q}$
is given by the following formulas:
\begin{align}
\Arf(\mu_{p,q}) & {} =
\sign \left( \cos \frac{ (p-q) \pi }{4}
+ \sin \frac{ (p-q) \pi }{4} \right)
\label{eq:Cl_Arf_pq_trigon}
\\ & {} =
\begin{cases}
	1	& \text{if } p-q+1 \equiv 1,2,3	\pmod{8} \\
	0	& \text{if } p-q+1 \equiv 0,4	\pmod{8} \\
	-1	& \text{if } p-q+1 \equiv 5,6,7	\pmod{8}
\end{cases}
\label{eq:Cl_Arf_pq_mod}
\end{align}
\end{theorem}

\selectlanguage{english}
\chapter*{Conclusions}
\addcontentsline{toc}{chapter}{Conclusions}

These have been the main results of my thesis.
First,
the classification up to isomorphism and equivalence
of the division gradings
on the simple real associative algebras.
Recall that this result completed
the classification up to isomorphism
of the (not necessarily division) gradings
on such algebras.
Second,
the classification up to isomorphism and equivalence
of the involutions
on the graded division simple real associative algebras.
And third,
the classification up to isomorphism
of the gradings
on the classical central simple real Lie algebras,
except on those of type $D_4$.

I think that I have been lucky in the doctorate,
since these classifications have been approachable.
The number of cases that appear is big enough
so that the problems are interesting,
and small enough
so that they can be solved in a thesis.
Besides we have obtained some beautiful results,
as theorem \ref{th:coars},
the division $ \mathbb{Z}_2 \times \mathbb{Z}_4 $-grading
on the algebra $ M_2(\mathbb{R}) \times \mathbb{H} $
of example \ref{exam:grad_semisimple},
or equation \eqref{eq:Cl_Arf_pq_trigon}.
Also we have been able to apply this theory of gradings
to Clifford algebras.

Of course,
we have not solved all the questions that
the classification of gradings on simple real Lie algebras
poses.
Our model of reducing
the classification of gradings on Lie algebras to
the classification of gradings on associative algebras with involution
covers a lot of the cases when the Lie algebra is of type $D_4$,
but not all of them.
The rest have been recently analyzed in \cite{EK2018arx},
therefore the classification in the type $D_4$ is already complete.

The future work is divided naturally into three ways.
One,
classify fine gradings up to equivalence
both on simple real associative algebras
and on classical central simple real Lie algebras,
and compute their respective Weyl groups.
Two,
classify gradings on simple complex Lie algebras,
regarded as algebras over the field of real numbers.
And three,
classify gradings on exceptional simple real Lie algebras.
Note that some cases of this third point have already been studied in
\cite{CDM2010} and \cite{DG2016}.

\selectlanguage{spanish}
\chapter*{Introducción}
\addcontentsline{toc}{chapter}{Introducción}

Mucho ha cambiado desde que mi director,
Alberto Elduque,
defendiera su tesis en 1984.
Él no publicó ninguno de sus resultados
mientras era estudiante de doctorado,
ya que no estaba claro si
una tesis que había aparecido previamente en revistas
era un trabajo original o no.
Así que lo primero que tuvo que hacer
después de la defensa
fue adaptar los capítulos de su tesis al formato artículo.

A día de hoy la situación es totalmente diferente,
y publicar es fundamental en un doctorado.
La investigación que hemos llevado a cabo
desde que empecé el doctorado en septiembre de 2013
ya ha aparecido en los cuatro artículos
listados en la página \pageref{ch:articles}.
Ahora lo que no tenemos claro es
qué debe contener el texto de una tesis.
Naturalmente cuando escribimos los artículos
lo hicimos de la mejor manera que supimos,
sin omitir ningún detalle e incluyendo
puntos claves, motivación, referencias históricas,
preliminares necesarios, etcétera.
Esperamos que la introducción que aparece en cada artículo sirva de guía
para saber cuáles son sus objetivos
y dónde encontrar los teoremas principales.
Mi intención es no repetir en esta tesis
lo que ya hemos expuesto en otros sitios,
y para ello mi amiga Eva me dio
una idea alternativa de qué podía escribir.

El objetivo de este texto es explicar,
utilizando un lenguaje lo más sencillo posible,
los resultados alcanzados a lo largo de estos cinco años.
Creo que al menos el capítulo \ref{cap:art_1},
correspondiente al primer artículo \cite{Rod2016},
puede entenderse incluso por alguien
que no haya estudiado una carrera de matemáticas.
Esto es algo excepcional en la investigación matemática,
de hecho tanto el segundo artículo \cite{BKR2018a}
como especialmente el tercero \cite{BKR2018b}
son mucho más técnicos.
Por otro lado el capítulo \ref{cap:art_4},
correspondiente al cuarto artículo \cite{ER2018},
se dirige a un público matemático,
pero no necesariamente especializado en álgebra.

El propósito general de mi tesis es clasificar
graduaciones en álgebras de Lie reales simples.
Una graduación es simplemente
una descomposición de un álgebra
compatible con sus operaciones
de suma, producto y multiplicación por escalares.
Para una enunciado preciso de los resultados alcanzados,
nos remitimos o bien a los objetivos \ref{obj_es:art_1},
\ref{obj_es:art_2}, \ref{obj_es:art_3} y \ref{obj_es:art_4},
o bien,
para un resumen más detallado,
a las secciones que los contienen
(secciones \ref{secc:art_1_mot_obj}, \ref{secc:art_2_obj},
\ref{secc:art_3_obj} y \ref{secc:art_4_intr}).
Además,
la sección \ref{secc:est_art} complementa este resumen.
En el primer y segundo artículo
clasificamos graduaciones en álgebras asociativas,
con la intención de aplicar estos resultados
a las álgebras de Lie en el tercer artículo,
tal y como explicamos en la sección \ref{secc:alg_inv}.
En cualquier caso remarquemos que cada una de estas clasificaciones
tiene interés por sí misma.
Finalmente en el cuarto artículo cambiamos a un tema más relajado,
ya que usamos resultados obtenidos en los artículos anteriores
para dar demostraciones alternativas a teoremas ya conocidos.

\selectlanguage{spanish}
\chapter{Preliminares}

\section{Álgebras y grupos}

Un álgebra es un conjunto dotado de
una estructura de espacio vectorial y de un producto bilineal.
Por ejemplo,
el conjunto $M_2(\mathbb{R})$ de
las matrices cuadradas de tamaño $ 2 \times 2 $
cuyas entradas son números reales,
con las operaciones habituales de suma y producto de matrices
y de multiplicación de un escalar por una matriz,
es un álgebra,
que además cumple las siguientes propiedades.
Es asociativa,
ya que $ (XY)Z = X(YZ) $ para todas matrices
$X,Y,Z$ en $M_2(\mathbb{R})$;
es real,
porque en este caso el cuerpo de escalares
son los números reales $\mathbb{R}$;
tiene dimensión $4$;
y es unitaria,
ya que existe un elemento neutro para el producto,
la matriz identidad $I$.

Otro ejemplo es el álgebra de cuaternios $\mathbb{H}$
\cite{Ham1844},
definámosla paso a paso.
Primero fijamos un espacio vectorial real $\mathbb{H}$ de dimensión $4$,
y llamamos $ \{ 1,i,j,k \} $ a los elementos de una de sus bases.
Esto se puede expresar mediante la siguiente fórmula:
\begin{equation}\label{ec:cuaternios}
\mathbb{H} = \mathbb{R} 1 \oplus \mathbb{R} i
\oplus \mathbb{R} j \oplus \mathbb{R} k
\end{equation}
$\mathbb{R}1$, $\mathbb{R}i$, $\mathbb{R}j$ y $\mathbb{R}k$
son subespacios vectoriales de dimensión $1$,
y con la ecuación \eqref{ec:cuaternios} indicamos que
su suma es directa y que es todo $\mathbb{H}$.
El producto lo definimos en los elementos de la base
mediante la tabla de multiplicación de la figura \ref{fig_es:tabla_mult_H},
y lo extendemos a todos los cuaternios por bilinealidad.
Por ejemplo,
$ (1+2i)(3i+j) = 3i+j+6i^2+2ij = -6+3i+j+2k $.

\begin{figure}
\begin{equation*}
\begin{array}{c|cccc}
	\cdot & 1 & i & j & k \\
	\hline
	1 & 1 & i & j & k \\
	i & i & -1 & k & -j \\
	j & j & -k & -1 & i \\
	k & k & j & -i & -1
\end{array}
\end{equation*}
\caption{Tabla de multiplicación de los cuaternios $\mathbb{H}$.}
\label{fig_es:tabla_mult_H}
\end{figure}

Podríamos comprobar,
mediante el análisis de $ 4 \cdot 4 \cdot 4 = 64 $ casos,
que el producto de cuaternios es asociativo,
pero afortunadamente hay una forma más corta de demostrarlo.
Consideramos las siguientes matrices en $M_2(\mathbb{C})$:
\begin{equation}\label{ec:cuatern_en_M2C}
I = \begin{pmatrix} 1 & 0 \\ 0 & 1 \end{pmatrix}
\quad
A_i = \begin{pmatrix} i & 0 \\ 0 & -i \end{pmatrix}
\quad
A_j = \begin{pmatrix} 0 & 1 \\ -1 & 0 \end{pmatrix}
\quad
A_k = \begin{pmatrix} 0 & i \\ i & 0 \end{pmatrix}
\end{equation}
Observamos que la tabla de multiplicación de estas cuatro matrices
es la misma que la de la figura \ref{fig_es:tabla_mult_H},
por ejemplo $ A_i A_j = A_k $ y $ A_i^2 = -I $.
Como el producto de matrices es asociativo,
también debe serlo el de cuaternios.

Notemos que el álgebra de cuaternios $\mathbb{H}$
no solo es unitaria,
sino que además es un álgebra de división,
ya que dado un cuaternio no nulo $ X = a+bi+cj+dk $,
existe un cuaternio $X^{-1}$ tal que $ XX^{-1} = 1 = X^{-1}X $.
Específicamente,
$X^{-1}$ está determinado por la siguiente fórmula:
\begin{equation}
(a+bi+cj+dk)^{-1} = \frac{1}{a^2+b^2+c^2+d^2} \, (a-bi-cj-dk)
\end{equation}

Un álgebra se dice simple si
no contiene ideales biláteros propios y su producto no es trivial.
Más adelante daremos caracterizaciones de las álgebras simples
que nos resultarán más manejables.
Simplemente tengamos presente que dichas álgebras son muy importantes,
porque en cierto sentido son los átomos indivisibles
a partir de los cuales se construyen una gran cantidad de álgebras.

Finalmente recordemos que
un grupo es un conjunto dotado de una operación binaria tal que:
es asociativa,
existe un elemento neutro $e$,
y todo elemento tiene inverso.
Si además la operación es conmutativa,
se dice que el grupo es abeliano.
Por ejemplo,
la figura \ref{fig_es:tabla_mult_Z4_Z22} nos muestra
las tablas de multiplicación de dos grupos abelianos de $4$ elementos.
El primero es
$ \mathbb{Z}_4 = \langle a \mid a^4 = e \rangle $,
y el segundo es
$ \mathbb{Z}_2^2 = \langle a,b \mid a^2=e=b^2 , \allowbreak \, ab=ba \rangle $.

\begin{figure}
\begin{equation*}
\begin{array}{c|cccc}
	\cdot & e & a & a^2 & a^3 \\
	\hline
	e & e & a & a^2 & a^3 \\
	a & a & a^2 & a^3 & e \\
	a^2 & a^2 & a^3 & e & a \\
	a^3 & a^3 & e & a & a^2
\end{array}
\qquad \qquad
\begin{array}{c|cccc}
	\cdot & e & a & b & ab \\
	\hline
	e & e & a & b & ab \\
	a & a & e & ab & b \\
	b & b & ab & e & a \\
	ab & ab & b & a & e
\end{array}
\end{equation*}
\caption{A la izquierda la tabla de multiplicación del grupo $\mathbb{Z}_4$,
a la derecha la del grupo $\mathbb{Z}_2^2$.}
\label{fig_es:tabla_mult_Z4_Z22}
\end{figure}

\section{Clasificaciones}

Como hemos mencionado en la introducción,
el objetivo de mi tesis es clasificar graduaciones,
pero para hacer una clasificación
lo primero que debe quedar claro es
cuándo entendemos que dos objetos son iguales.
Esta pregunta no es en absoluto trivial.

Por ejemplo,
consideremos el álgebra compleja $\mathbb{H}_{\mathbb{C}}$,
obtenida al extender escalares
del cuerpo $\mathbb{R}$ al cuerpo $\mathbb{C}$
en el álgebra real de cuaternios $\mathbb{H}$.
Es decir,
los elementos $X=a+bi+cj+dk$ de $\mathbb{H}_{\mathbb{C}}$
siguen satisfaciendo la tabla de multiplicación
de la figura \ref{fig_es:tabla_mult_H},
pero en este caso los coeficientes $a$, $b$, $c$, $d$
son números complejos en lugar de reales.
Ahora comparamos las álgebras $\mathbb{H}_{\mathbb{C}}$ y $M_2(\mathbb{C})$.
Aunque las dos son álgebras asociativas complejas de dimensión $4$,
en principio parecen bastante distintas.
Sin embargo definamos una aplicación
$ f : \mathbb{H}_{\mathbb{C}} \to M_2(\mathbb{C}) $
como $ f(a+bi+cj+dk) = a I + b A_i + c A_j + d A_k $,
donde $I$, $A_i$, $A_j$ y $A_k$ son las matrices complejas
de la ecuación \eqref{ec:cuatern_en_M2C}.
Observamos que esta aplicación es biyectiva
y conmuta con las operaciones de ambas álgebras complejas:
$ f(X+Y) = f(X)+f(Y) $,
$ f( \lambda X ) = \lambda f(X) $,
y $ f(XY) = f(X)f(Y) $
para todos elementos $X,Y$ en $\mathbb{H}_{\mathbb{C}}$
y todo escalar $\lambda$ en $\mathbb{C}$.
Por lo tanto podemos afirmar que
$\mathbb{H}_{\mathbb{C}}$ y $M_2(\mathbb{C})$
son la misma álgebra,
pero escrita con diferentes alfabetos,
y que la aplicación $f$ es un diccionario
que nos permite traducir del uno al otro.
Formalmente se dice que las álgebras complejas
$\mathbb{H}_{\mathbb{C}}$ y $M_2(\mathbb{C})$ son isomorfas,
escrito $ \mathbb{H}_{\mathbb{C}} \cong M_2(\mathbb{C}) $,
y que $f$ es un isomorfismo.

Por otro lado,
las álgebras reales $\mathbb{H}$ y $M_2(\mathbb{R})$ no son isomorfas,
ya que la primera es un álgebra de división pero la segunda no.
En efecto,
las matrices cuyo determinante vale $0$ no son invertibles.

Volvamos ahora a los grupos $\mathbb{Z}_4$ y $\mathbb{Z}_2^2$
de la figura \ref{fig_es:tabla_mult_Z4_Z22}.
Aunque ambos son grupos abelianos de $4$ elementos,
no son isomorfos,
porque en $\mathbb{Z}_2^2$
el cuadrado de cualquier elemento es el elemento neutro $e$,
mientras que en $\mathbb{Z}_4$
hay elementos de orden $4$ ($a$ y $a^3$).
Con un poco de paciencia se puede comprobar que
todo grupo de $4$ elementos es isomorfo a uno de estos dos.

Un ejemplo muy ilustrativo de clasificación es la de las
álgebras asociativas reales de dimensión finita y simples.
Consta de dos partes.
Por un lado Frobenius clasifica en \cite{Fro1878} las
álgebras asociativas reales de dimensión finita y de división.
Resulta que, salvo isomorfismo, solo hay tres:
el cuerpo de los números reales $\mathbb{R}$,
el cuerpo de los números complejos $\mathbb{C}$,
y el álgebra de cuaternios $\mathbb{H}$.
La parte más fácil de la clasificación es comprobar que en efecto
estas tres álgebras cumplen las propiedades requeridas de
asociatividad, divisibilidad, etcétera.
El verdadero problema es demostrar que
no existen más álgebras con estas características,
por ello este resultado se conoce como teorema de Frobenius.
Finalmente hay que asegurarse de que no hay redundancias;
en este caso podemos decir que las tres álgebras no son isomorfas
ya que tienen diferente dimensión:
$1$, $2$ y $4$ respectivamente.

Por otro lado,
el teorema de Wedderburn \cite{Wed1908} y Artin \cite{Art1927}
dice que toda
álgebra asociativa de dimensión finita y simple
se puede escribir como un álgebra de matrices cuadradas
con entradas en un álgebra de división.
Además este resultado es realmente una clasificación,
ya que también establece
cuándo dos de estas álgebras de matrices son isomorfas.
Esto ocurre si y solo si el tamaño de las matrices coincide
y las correspondientes álgebras de división son isomorfas.
Juntando los dos teoremas podemos dar la lista,
exhaustiva y sin repeticiones,
de las álgebras asociativas reales de dimensión finita y simples:
\begin{equation}\label{ec:lista_alg_simples}
\begin{array}{cccc}
	\mathbb{R} , & M_2(\mathbb{R}) , & M_3(\mathbb{R}) , & \dots \\
	\mathbb{C} , & M_2(\mathbb{C}) , & M_3(\mathbb{C}) , & \dots \\
	\mathbb{H} , & M_2(\mathbb{H}) , & M_3(\mathbb{H}) , & \dots
\end{array}
\end{equation}

\section{Graduaciones}

Las álgebras por sí mismas han sido estudiadas a fondo desde hace tiempo.
El propósito de las graduaciones es ir un paso más lejos
e investigar de qué maneras se pueden descomponer las álgebras.
Para que una descomposición sea válida,
le vamos a exigir que sea compatible con las operaciones del álgebra.

La definición precisa de graduación
en un álgebra $\mathcal{D}$ por un grupo $G$
consta de dos partes.
Primero,
la graduación como tal es simplemente una descomposición
del espacio vectorial subyacente de $\mathcal{D}$
en suma directa de subespacios vectoriales,
de manera que cada subespacio vectorial
está indexado por un elemento del grupo $G$.
Esto lo expresamos con notación matemática de la siguiente manera:
\begin{equation}\label{ec:grad_cond_suma}
\mathcal{D} = \bigoplus_{ g \in G } \mathcal{D}_g
\end{equation}
El subespacio vectorial $\mathcal{D}_g$ se llama
componente homogénea de grado $g$,
y sus vectores $ X \in \mathcal{D}_g $ se dicen
elementos homogéneos de grado $g$,
escrito $ \deg X = g $.
La ecuación \eqref{ec:grad_cond_suma} quiere decir que la descomposición
es compatible con la suma y con la multiplicación por escalares;
pero si $\mathcal{D}$ es un álgebra graduada,
también se debe respetar el producto.
Así,
como segunda condición exigimos que el producto
de un elemento homogéneo de grado $g$
por un elemento homogéneo de grado $h$
sea un elemento homogéneo de grado $gh$.
Si esto lo escribimos con símbolos,
diríamos que para todos $g,h$ en $G$
se debe satisfacer la siguiente fórmula:
\begin{equation}\label{ec:grad_cond_prod}
\mathcal{D}_g \mathcal{D}_h \subseteq \mathcal{D}_{gh}
\end{equation}

\begin{ejemplo}\label{ejem:grad_M2R_dim1}
Podemos descomponer la ya mencionada álgebra $M_2(\mathbb{R})$
como suma directa de cuatro subespacios vectoriales de dimensión $1$
mediante la siguiente ecuación:
\begin{equation}\label{ec:grad_M2R_dim1}
M_2(\mathbb{R}) =
	\mathbb{R} \begin{pmatrix} 1 & 0 \\ 0 & 1 \end{pmatrix}
	\oplus
	\mathbb{R} \begin{pmatrix} -1 & 0 \\ 0 & 1 \end{pmatrix}
	\oplus
	\mathbb{R} \begin{pmatrix} 0 & 1 \\ 1 & 0 \end{pmatrix}
	\oplus
	\mathbb{R} \begin{pmatrix} 0 & -1 \\ 1 & 0 \end{pmatrix}
\end{equation}
Si asignamos a estos cuatro subespacios
los grados $e$, $a$, $b$, $ab$ en el grupo
$ \mathbb{Z}_2^2 = \langle a,b \mid a^2=e=b^2 , \allowbreak \, ab=ba \rangle $,
obtenemos una graduación.
En efecto,
analizando los $ 4 \cdot 4 = 16 $ casos posibles,
vemos que se cumple la ecuación \eqref{ec:grad_cond_prod}.
\end{ejemplo}

\begin{ejemplo}\label{ejem:grad_H_dim1}
La propia definición de los cuaternios $\mathbb{H}$
nos sugiere una graduación.
De nuevo asignamos grados en el grupo $\mathbb{Z}_2^2$,
esta vez a los cuatro subespacios de la ecuación \eqref{ec:cuaternios},
$ \mathbb{H} = \mathbb{R} 1 \oplus \mathbb{R} i
\oplus \mathbb{R} j \oplus \mathbb{R} k $.
Comparando la tabla de multiplicación de $\mathbb{H}$
(figura \ref{fig_es:tabla_mult_H})
con la de $\mathbb{Z}_2^2$
(figura \ref{fig_es:tabla_mult_Z4_Z22}),
comprobamos que se respeta el producto.
\end{ejemplo}

\begin{ejemplo}\label{ejem:grad_C_dim1}
Otra graduación natural es la que se obtiene en los números complejos
al separar entre parte real y parte imaginaria:
\begin{equation}\label{ec:grad_C_dim1}
\mathbb{C} = \mathbb{R} 1 \oplus \mathbb{R} i
\end{equation}
En este caso el grupo graduador es
$ \mathbb{Z}_2 = \langle a \mid a^2=e \rangle $.
\end{ejemplo}

Señalemos un matiz matemático.
Formalmente,
la definición de graduación permite que haya
componentes homogéneas de dimensión $0$.
Por este motivo se define el soporte de una graduación,
que es el conjunto de los elementos del grupo graduador
cuyas componentes homogéneas son no nulas.

Dadas dos graduaciones en una misma álgebra,
puede ocurrir que la primera sea un refinamiento de la segunda,
o dicho de otra forma,
que la segunda sea un engrosamiento de la primera.
Esto sucede cuando cada componente homogénea de la primera graduación
está contenida en una componente homogénea de la segunda.

\begin{ejemplo}\label{ejem:grad_M2R_dim2}
Un engrosamiento de la graduación
en el álgebra $M_2(\mathbb{R})$
del ejemplo \ref{ejem:grad_M2R_dim1}
es la graduación por el grupo $\mathbb{Z}_2$
definida por la siguiente ecuación:
\begin{equation}\label{ec:grad_M2R_dim2}
M_2(\mathbb{R}) =
\left[
	\mathbb{R} \begin{pmatrix} 1 & 0 \\ 0 & 1 \end{pmatrix}
	\oplus
	\mathbb{R} \begin{pmatrix} 0 & -1 \\ 1 & 0 \end{pmatrix}
\right]
	\oplus
\left[
	\mathbb{R} \begin{pmatrix} 0 & 1 \\ 1 & 0 \end{pmatrix}
	\oplus
	\mathbb{R} \begin{pmatrix} -1 & 0 \\ 0 & 1 \end{pmatrix}
\right]
\end{equation}
Observamos que en este caso hay
dos componentes homogéneas de dimensión $2$ cada una.
\end{ejemplo}

Las graduaciones de los ejemplos
\ref{ejem:grad_M2R_dim1}, \ref{ejem:grad_H_dim1} y \ref{ejem:grad_C_dim1}
tienen todas sus componentes homogéneas de dimensión $1$,
con lo que no se pueden refinar más.
Por este último motivo se dice que son finas.
Recíprocamente la graduación trivial,
que engloba toda el álgebra en una única componente homogénea,
es siempre la más gruesa.

Un álgebra asociativa unitaria y graduada
se dice que es graduada de división
si todo elemento homogéneo no nulo tiene inverso.
Las graduaciones de los ejemplos
\ref{ejem:grad_M2R_dim1}, \ref{ejem:grad_H_dim1},
\ref{ejem:grad_C_dim1} y \ref{ejem:grad_M2R_dim2}
son de división.

Recordemos que nuestro objetivo va a ser clasificar graduaciones.
Por tanto,
dadas dos álgebras graduadas
$ \mathcal{D} = \bigoplus_{ g \in G } \mathcal{D}_g $ y
$ \mathcal{E} = \bigoplus_{ h \in H } \mathcal{E}_h $,
debemos dejar claro si las consideramos iguales o no.
Hay dos formas naturales de definir esta igualdad,
dependiendo de si el grupo graduador juega un papel secundario
o si forma parte de la definición,
así que para evitar confusiones
llamaremos a una equivalencia y a la otra isomorfismo.
Decimos que
las álgebras graduadas $\mathcal{D}$ y $\mathcal{E}$ son equivalentes
si existe un isomorfismo de álgebras
$ \psi : \mathcal{D} \to \mathcal{E} $
satisfaciendo la siguiente condición:
para todo $g$ en el soporte de $\mathcal{D}$
existe $ h \in H $ tal que
$ \psi(\mathcal{D}_g) = \mathcal{E}_h $.
Si $G=H$,
podemos endurecer la condición requerida a que
$ \psi(\mathcal{D}_g) = \mathcal{E}_g $
para todo $ g \in G $,
en cuyo caso decimos que
las álgebras graduadas $\mathcal{D}$ y $\mathcal{E}$ son isomorfas.

\section{Productos tensoriales y formas cuadráticas}

Finalmente repasemos algunos conceptos más que nos van a aparecer.
Empezamos con el producto tensorial graduado.
Recordamos que,
si $ \{ X_1 , \allowbreak X_2 , \allowbreak \dots , \allowbreak X_r \} $
es una base de un espacio vectorial $\mathcal{D}$
y $ \{ Y_1 , \allowbreak Y_2 , \allowbreak \dots , \allowbreak Y_s \} $
es una base de otro espacio vectorial $\mathcal{E}$,
entonces el producto tensorial $ \mathcal{D} \otimes \mathcal{E} $
es un espacio vectorial del que una base es
$ \{ X_i \otimes Y_j \mid 1 \leq i \leq r , \allowbreak \, 1 \leq j \leq s \} $.
Notemos que el cuerpo de escalares
juega un papel importante en esta construcción,
por lo que lo indicaremos con un subíndice del símbolo $\otimes$
cuando sea distinto de los números reales $\mathbb{R}$.
Por ejemplo,
el espacio vectorial real $ \mathbb{C} \otimes M_2(\mathbb{C}) $
tiene dimensión $ 2 \cdot 8 = 16 $,
mientras que el espacio vectorial complejo
$ \mathbb{C} \otimes_{\mathbb{C}} M_2(\mathbb{C}) $
tiene dimensión $ 1 \cdot 4 = 4 $.

Si $\mathcal{D}$ y $\mathcal{E}$
son no solo espacios vectoriales
sino álgebras,
entonces $ \mathcal{D} \otimes \mathcal{E} $
también tiene estructura de álgebra,
con el producto componente a componente:
$ ( X \otimes Y ) ( X' \otimes Y' ) = (XX') \otimes (YY') $.
Por lo tanto,
si el álgebra $\mathcal{D}$ está graduada por el grupo $G$
y el álgebra $\mathcal{E}$ está graduada por el grupo $H$,
resulta natural que definamos una graduación
en el álgebra $ \mathcal{D} \otimes \mathcal{E} $
por el grupo $ G \times H $
diciendo que la componente homogénea de grado $(g,h)$
es simplemente $ \mathcal{D}_g \otimes \mathcal{E}_h $.

Cambiemos ahora al tema de las formas cuadráticas.
Un bicarácter alternado ($\mathbb{R}$-valuado) en un grupo $T$
es una aplicación $ \beta : T \times T \to \mathbb{R} \setminus \{ 0 \} $
que satisface que
$ \beta(uv,w) = \beta(u,w) \beta(v,w) $,
$ \beta(u,vw) = \beta(u,v) \beta(u,w) $,
y $ \beta(u,u) = 1 $
para todos $ u,v,w \in T $.
Si el grupo $T$ es finito,
entonces $\beta$ solo puede tomar los valores $+1$ y $-1$.

El radical de $\beta$ es el conjunto
$ \rad(\beta) =
\{ t \in T \mid \beta(u,t) = 1
\allowbreak \: \text{para} \allowbreak \: \text{todo} \allowbreak \:
u \in T \} $.
Diremos que $\beta$ es de tipo I
si el único elemento de su radical es el elemento neutro del grupo,
$e$.
Análogamente,
diremos que $\beta$ es de tipo II
si $\rad(\beta)$ tiene dos elementos;
en este caso el elemento del radical
que no es el elemento neutro
lo denotaremos $f_{\beta}$.

Una forma cuadrática en $T$
es una aplicación $ \mu : T \to \{ \pm 1 \} $
tal que $\beta_{\mu}$ es un bicarácter alternado,
donde $ \beta_{\mu} : T \times T \to \{ \pm 1 \} $ es la aplicación,
llamada polarización de $\mu$,
definida por la siguiente fórmula:
\begin{equation}\label{ec:polarizacion}
\beta_{\mu}(u,v) = \mu(uv) \mu(u)^{-1} \mu(v)^{-1}
\end{equation}

Subrayemos que las formas cuadráticas
son objetos bien conocidos dentro de las matemáticas.
Normalmente se escriben con notación aditiva,
como en la ecuación \eqref{ec:form_bil_form_cuadr},
pero en este texto hemos preferido usar notación multiplicativa,
ya que será más conveniente para nuestros propósitos.
Notemos también que los inversos que aparecen en
la ecuación \eqref{ec:polarizacion}
no tienen efecto,
pero así es como se escribe habitualmente esta fórmula.

El invariante de Arf de una aplicación
$ \mu : T \to \{ \pm 1 \} $
definida en un conjunto finito $T$
es el valor que toma más a menudo dicha aplicación.
Es decir,
si hay más elementos en $T$ de valor $+1$ que $-1$,
entonces el invariante de Arf de $\mu$ es $+1$.
Pero si hay menos,
entonces es $-1$.
También puede ocurrir que haya el mismo número,
y entonces decimos que el invariante de Arf es $0$.

\selectlanguage{spanish}
\chapter{Primer artículo}\label{cap:art_1}

\begin{enumerate}

\item[\cite{Rod2016}]
A. Rodrigo-Escudero,
Classification of division gradings on fi\-nite-di\-men\-sional simple real algebras,
Linear Algebra Appl. {\bf 493} (2016), 164--182.
\\
\url{https://arxiv.org/abs/1506.01552}
\\
\url{https://doi.org/10.1016/j.laa.2015.11.025}

\end{enumerate}

\section{Estado del arte}\label{secc:est_art}

Justo antes de que yo empezara mi tesis mi director,
Alberto Elduque,
y uno de sus principales colaboradores,
Mikhail Kochetov,
publicaron la monografía \cite{EK2013}.
Esto ha sido una gran suerte para mí.
Por un lado porque aceleró considerablemente
mi tarea de búsqueda bibliográfica,
que es el paso previo al trabajo de investigación propiamente dicho.
De hecho este libro,
y las referencias que contiene,
son el punto de partida recomendable para cualquiera
que vaya a empezar a trabajar con graduaciones.
En él se resumen y mejoran los trabajos de muchos autores,
como Patera, Zassenhaus, Bahturin, Zaicev, Draper,
o Martín González
(\cite{PZ1989}, \cite{HPP1998},
\cite{BSZ2001}, \cite{BZ2002}, \cite{BSZ2005}, \cite{BZ2006}, \cite{BZ2007},
\cite{DM2006}, \cite{DM2009},
\cite{BK2010}, \cite{Eld2010},
etcétera).

Y por otro lado porque uno de los temas de la monografía
son las graduaciones en álgebras complejas,
así que los problemas de clasificación que se abordan en
mis tres primeros artículos \cite{Rod2016}, \cite{BKR2018a} y \cite{BKR2018b}
ya estaban resueltos en \cite{EK2013}
para el caso en el que el cuerpo de escalares
son los números complejos en lugar de los números reales.
La primera aproximación para atacar esas clasificaciones de graduaciones
ha sido intentar adaptar las técnicas del caso complejo al caso real.

Uno de los resultados que ya estaba demostrado
y en el que se basa mi tesis
es el análogo graduado al teorema de Artin--Wedderburn.
Aparece en \cite[corolario 2.12]{EK2013}
en la versión más general posible,
aunque ya se había estudiado en
\cite{BSZ2001}, \cite{BZ2002}, e incluso \cite{NO1982}.
Ahora vamos a enunciarlo en el caso particular
de álgebras simples y de dimensión finita.

Supongamos que tenemos un álgebra graduada de división $\mathcal{D}$
por un grupo abeliano $G$,
y que $ g_1 , g_2 , \dots , g_k $ son $k$ elementos
(no necesariamente distintos) de $G$.
Con estos ingredientes podemos construir de manera natural
un álgebra $G$-graduada;
tomamos como álgebra subyacente el álgebra $M_k(\mathcal{D})$
de matrices $ k \times k $ con entradas en $\mathcal{D}$,
y la graduación la definimos diciendo que
la matriz cuyas entradas son todas nulas,
excepto la entrada $(i,j)$ que es $ d \in \mathcal{D}_t $,
tiene grado $ g_i t g_j^{-1} $.
El teorema nos dice que toda
álgebra asociativa de dimensión finita simple y $G$-graduada
es isomorfa a una de la forma $M_k(\mathcal{D})$.

Además,
este resultado es realmente una clasificación,
porque también nos dice
cuándo dos de estas álgebras $G$-graduadas son isomorfas.
En este texto no vamos a enunciar la condición de isomorfismo,
porque requiere de una notación más formal.
Por otro lado,
la clasificación salvo equivalencia es un problema difícil,
que solo se conoce para el caso de las graduaciones finas.

Lo que nos importa en este momento es que,
para completar la clasificación salvo isomorfismo de las graduaciones
en las álgebras asociativas de dimensión finita y simples,
hay que probar un análogo graduado del teorema de Frobenius.
Es decir, hay que clasificar las álgebras graduadas de división.
Esta cuestión se resuelve en \cite[teorema 2.15]{EK2013}
(ver también \cite{BSZ2001} y \cite{BK2010})
en el caso en el que el cuerpo de escalares son los números complejos;
pero el caso real era un problema abierto al comienzo de mi tesis.

\section{Motivación y objetivo}\label{secc:art_1_mot_obj}

En otoño de 2014 mi director me plantea
mi primer problema de investigación.
Se trata de hallar todas las graduaciones
(por grupos abelianos)
en el álgebra real $M_2(\mathbb{C})$.
Recordamos que una graduación es simplemente
una descomposición del álgebra
compatible con sus operaciones
de suma, producto y multiplicación por escalares.
Resulta que como álgebra compleja de dimensión $4$
sí que se conocían las graduaciones de $M_2(\mathbb{C})$,
pero no como álgebra real de dimensión $8$.
Como ya hemos visto,
gracias al análogo graduado del teorema de Artin--Wedderburn,
la cuestión se reduce a buscar aquellas graduaciones que son de división.
El primer descubrimiento fue la graduación del siguiente ejemplo.

\begin{ejemplo}\label{ejem:grad_M2C_dim1}
La figura \ref{fig_es:grad_M2C_dim1} nos muestra
una graduación de división en el álgebra real $M_2(\mathbb{C})$,
en la que el grupo graduador es
$ \mathbb{Z}_2 \times \mathbb{Z}_4 =
\langle a,b \mid a^2=e=b^4 , \allowbreak \, ab=ba \rangle $
y los grados se asignan en el siguiente orden:
$e$,~$a$; $b$,~$ab$; $b^2$,~$ab^2$; $b^3$,~$ab^3$.
\begin{figure}
\begin{align*}
M_2(\mathbb{C}) = {} &
	\mathbb{R} \begin{pmatrix} 1 & 0 \\ 0 & 1 \end{pmatrix}
	\oplus
	\mathbb{R} \begin{pmatrix} 0 & 1 \\ 1 & 0 \end{pmatrix}
	\oplus
\\ &
	\mathbb{R} \begin{pmatrix} 1+i & 0 \\ 0 & -1-i \end{pmatrix}
	\oplus
	\mathbb{R} \begin{pmatrix} 0 & -1-i \\ 1+i & 0 \end{pmatrix}
	\oplus
\\ &
	\mathbb{R} \begin{pmatrix} i & 0 \\ 0 & i \end{pmatrix}
	\oplus
	\mathbb{R} \begin{pmatrix} 0 & i \\ i & 0 \end{pmatrix}
	\oplus
\\ &
	\mathbb{R} \begin{pmatrix} -1+i & 0 \\ 0 & 1-i \end{pmatrix}
	\oplus
	\mathbb{R} \begin{pmatrix} 0 & 1-i \\ -1+i & 0 \end{pmatrix}
\end{align*}
\caption{Graduación de división del ejemplo \ref{ejem:grad_M2C_dim1}.}
\label{fig_es:grad_M2C_dim1}
\end{figure}
\end{ejemplo}

Subrayemos que la presencia de un factor $\mathbb{Z}_4$
en el grupo graduador del ejemplo \ref{ejem:grad_M2C_dim1}
es algo excepcional.
Utilizando argumentos similares
a los que ya se habían empleado para resolver el caso complejo,
y que más adelante explicaremos con algo de detalle,
encontramos todas las graduaciones de división posibles en $M_2(\mathbb{C})$
cuyas componentes homogéneas tienen dimensión $1$.

El problema estaba en las graduaciones
con componentes de dimensión mayor que $1$.
Después de algunos cálculos,
obtuvimos una lista que esperábamos que contuviera a todas.
Y en efecto,
probamos que no hay más,
mediante un argumento basado en el teorema de Cayley--Hamilton.
La mala noticia es que dicha demostración
no se podía generalizar al caso de matrices
de tamaño mayor que $ 2 \times 2 $.
Ahora bien,
conocer lo que ocurre en los casos de dimensión baja siempre es útil,
y en esta ocasión nos permitió observar que
todas estas graduaciones eran engrosamientos de
aquellas cuyas componentes homogéneas tienen dimensión $1$.
Dicha observación se convirtió en conjetura y,
con algo de ingenio,
en el siguiente resultado.

\begin{teorema}\label{teor:engros}
Si una graduación de división por un grupo abeliano
en un álgebra asociativa real de dimensión finita es fina,
entonces o bien todas las componentes homogéneas tienen dimensión $1$,
o bien se trata de una graduación como álgebra compleja.
\end{teorema}

Este teorema fue fundamental,
ya que gracias a él
pensamos que podríamos resolver
no solo el caso de matrices $ 2 \times 2 $,
sino el de matrices de cualquier tamaño.
Es en este momento que
nos planteamos como factible el siguiente objetivo,
que acabaría siendo el propósito de
mi primer artículo \cite{Rod2016}.

\begin{objetivo}\label{obj_es:art_1}
Clasificamos,
salvo isomorfismo y salvo equivalencia,
las graduaciones de división
(por grupos abelianos)
en las álgebras asociativas reales de dimensión finita y simples.
\end{objetivo}

Como ya hemos mencionado anteriormente,
esto es un análogo graduado del teorema de Frobenius,
y al sumarlo a la versión graduada del teorema de Artin--Wedderburn,
obtenemos la clasificación salvo isomorfismo
de las graduaciones (no necesariamente de división)
en las álgebras asociativas reales de dimensión finita y simples,
$M_n(\mathbb{R})$, $M_n(\mathbb{C})$ y $M_n(\mathbb{H})$.
Por otro lado,
otro motivo que teníamos en mente
para clasificar graduaciones en álgebras asociativas
era usar estos resultados para hacer lo mismo en álgebras de Lie.
Pero ya volveremos a esta cuestión
cuando hablemos del tercer artículo \cite{BKR2018b}
en el capítulo \ref{cap:art_3}.

Nuestra estrategia para llevar a cabo la clasificación
va a ser separar el problema en tres casos,
de acuerdo con la siguiente observación,
que es consecuencia del Teorema de Frobenius.

\begin{observacion}\label{obs:dim_comp}
Las componentes homogéneas
de una graduación de división
en un álgebra asociativa real de dimensión finita
tienen todas la misma dimensión,
que es igual a $1$, $2$ o $4$,
dependiendo de si la componente neutra es isomorfa a
$\mathbb{R}$, $\mathbb{C}$ o $\mathbb{H}$,
respectivamente.
\end{observacion}

\section{Dimensión uno}

Abordemos la clasificación
de las graduaciones de división
en las álgebras asociativas reales de dimensión finita y simples.
Primero consideramos el caso en el que
las componentes homogéneas tienen dimensión $1$,
de acuerdo con la observación \ref{obs:dim_comp}.

Empezamos analizando el álgebra graduada de división $\mathcal{D}$
del ejemplo \ref{ejem:grad_M2R_dim1}.
El álgebra subyacente es $M_2(\mathbb{R})$
y el soporte es $ T = \mathbb{Z}_2^2 = \{ e,a,b,ab \} $.
Llamemos $X_a$ a la segunda matriz que aparece
en la ecuación \eqref{ec:grad_M2R_dim1},
y $X_b$ a la tercera.
Así,
dicha ecuación se escribe de la siguiente forma:
\begin{equation}\label{ec:ejem_grad_clas_dim1}
M_2(\mathbb{R}) = \mathbb{R} I \oplus \mathbb{R} X_a
\oplus \mathbb{R} X_b \oplus \mathbb{R} X_a X_b
\end{equation}
Notamos además que $X_a$ y $X_b$ satisfacen las siguientes relaciones:
\begin{equation}\label{ec:ejem_rel_clas_dim1}
X_a^2 = + I \qquad X_b^2 = + I \qquad X_a X_b = - X_b X_a
\end{equation}
La idea está en que las ecuaciones
\eqref{ec:ejem_grad_clas_dim1} y \eqref{ec:ejem_rel_clas_dim1}
determinan completamente el álgebra graduada $\mathcal{D}$.
Por ejemplo,
$ ( X_a + 3 X_a X_b ) ( X_a + 2 X_b ) =
X_a^2 + 2 X_a X_b + 3 X_a X_b X_a + 6 X_a X_b^2 =
I + 6 X_a - 3 X_b + 2 X_a X_b $.

Las tres relaciones
de la ecuación \eqref{ec:ejem_rel_clas_dim1}
son suficientes,
pero surge un problema.
Supongamos que hubiéramos escogido la cuarta matriz,
$X_{ab}$,
en lugar de la tercera,
entonces la segunda relación sería distinta:
$ X_a^2 = + I $,
$ X_{ab}^2 = - I $,
$ X_a X_{ab} = - X_{ab} X_a $.
Sin embargo está claro que estas nuevas relaciones
definen la misma álgebra graduada.
La solución para evitar redundancias
va a ser quedarnos con todas las relaciones.

Definimos la aplicación
$ \beta : T \times T \to \mathbb{R} \setminus \{ 0 \} $
mediante las relaciones de conmutación de los elementos homogéneos
($ X_e = I $):
\begin{equation}\label{ec:rel_conm}
X_u X_v = \beta(u,v) X_v X_u
\end{equation}
También definimos la aplicación $ \mu : T \to \{ \pm 1 \} $
mediante los signos de los cuadrados de los elementos homogéneos:
\begin{equation}\label{ec:sign_cuadr}
X_t^2 = \mu(t) I
\end{equation}
Con un pequeño cálculo comprobamos que
$\beta$ es un bicarácter alternado de tipo I,
y que $\mu$ es una forma cuadrática
cuya polarización $\beta_{\mu}$ es precisamente $\beta$.
Además,
estas aplicaciones $\beta$ y $\mu$ son los invariantes que caracterizan
el álgebra graduada $\mathcal{D}$ salvo isomorfismo.
Esto quiere decir que,
primero,
$\beta$ y $\mu$ no dependen de la elección de los $X_t$,
y segundo,
que si repetimos esta construcción
partiendo de otra álgebra graduada $\mathcal{D}'$,
obteniendo $\beta'$ y $\mu'$,
entonces $ \beta = \beta' $ y $ \mu = \mu' $ si y solo si
$\mathcal{D}$ es isomorfa a $\mathcal{D}'$.
De hecho,
como $ \beta_{\mu} = \beta $,
toda la información está contenida en la forma cuadrática $\mu$.

Estos argumentos siguen siendo válidos en el caso general
en el que el álgebra subyacente de $\mathcal{D}$
es $M_n(\mathbb{R})$ o $M_{n/2}(\mathbb{H})$.
Como $\beta$ es de tipo I,
la teoría de bicaracteres alternados nos dice que
el grupo $T$ es isomorfo a $\mathbb{Z}_2^{2m}$.
Por lo tanto hay $2^{2m}$ componentes homogéneas de dimensión $1$,
y ya sabemos que la dimensión de $\mathcal{D}$ es $n^2$,
luego necesariamente $ n = 2^m $.
Es decir,
estas graduaciones de división existen solo si
el tamaño de las matrices es una potencia de $2$.

Consideremos ahora la clasificación salvo equivalencia.
Dos de estas álgebras graduadas son equivalentes si y solo si
sus correspondientes formas cuadráticas son equivalentes,
es decir,
si después de renombrar los elementos de $T$,
ambas formas cuadráticas son iguales.
La clasificación de las formas cuadráticas con polarización de tipo I
es bien conocida (la idea es coger una base simpléctica);
se dividen en dos clases de equivalencia,
una correspondiente a las formas cuadráticas cuyo invariante de Arf es $+1$,
y otra correspondiente a las de invariante de Arf $-1$.

Por otro lado tenemos que probar,
en el caso general $ n = 2^m $,
la existencia de estas graduaciones de división.
Consideramos el siguiente producto tensorial graduado,
donde la graduación de cada factor
es la del ejemplo \ref{ejem:grad_M2R_dim1}:
\begin{equation}
M_n(\mathbb{R}) \cong M_2(\mathbb{R}) \otimes
\dots \otimes M_2(\mathbb{R})
\end{equation}
Es,
en efecto,
un álgebra graduada de división
con componentes homogéneas de dimensión $1$,
y su correspondiente forma cuadrática tiene invariante de Arf $+1$.
Por lo tanto todas las demás formas cuadráticas
en $ (\mathbb{Z}_2^2)^m \cong \mathbb{Z}_2^{2m} $
con polarización de tipo I e invariante de Arf $+1$
también se obtienen de graduaciones de división en $M_n(\mathbb{R})$,
ya que son equivalentes a la que acabamos de construir.
Análogamente,
gracias a este otro producto tensorial
cuyo factor $\mathbb{H}$ está graduado
como en el ejemplo \ref{ejem:grad_H_dim1},
vemos que aquellas con invariante de Arf $-1$
provienen de graduaciones de división en $M_{n/2}(\mathbb{H})$:
\begin{equation}
M_{n/2}(\mathbb{H}) \cong M_2(\mathbb{R}) \otimes
\dots \otimes M_2(\mathbb{R}) \otimes \mathbb{H}
\end{equation}

Ya podemos enunciar la clasificación salvo equivalencia,
y la clasificación salvo isomorfismo,
por ejemplo en el caso de matrices de cuaternios.
Si el álgebra subyacente es $M_n(\mathbb{C})$
el razonamiento es similar,
pero se obtienen bicaracteres alternados de tipo II;
esto hace que aparezca una segunda clase de equivalencia,
en la que los argumentos se complican un poco
ya que el grupo graduador tiene elementos de orden $4$.

\begin{teorema}\label{teor:art_uno_dim1_equiv}
Toda graduación de división,
por un grupo abeliano,
en un álgebra asociativa real de dimensión finita y simple,
cuyas componentes homogéneas tienen dimensión $1$,
es equivalente a uno, y solo uno,
de los productos tensoriales graduados de la siguiente lista,
donde las graduaciones de cada factor son las de los ejemplos
\ref{ejem:grad_M2R_dim1}, \ref{ejem:grad_H_dim1},
\ref{ejem:grad_C_dim1} y \ref{ejem:grad_M2C_dim1}.
\begin{enumerate}
\item[(1-a)]
	$ M_n(\mathbb{R}) \cong M_2(\mathbb{R}) \otimes
	\dots \otimes M_2(\mathbb{R}) $,
	$ n = 2^m \geq 1 $.\\
	El grupo graduador es
	$ (\mathbb{Z}_2^2)^m \cong \mathbb{Z}_2^{2m} $.
\item[(1-b)]
	$ M_{n/2}(\mathbb{H}) \cong M_2(\mathbb{R}) \otimes
	\dots \otimes M_2(\mathbb{R}) \otimes \mathbb{H} $,
	$ n = 2^m \geq 2 $.\\
	El grupo graduador es
	$ (\mathbb{Z}_2^2)^m \cong \mathbb{Z}_2^{2m} $.
\item[(1-c)]
	$ M_n(\mathbb{C}) \cong M_2(\mathbb{R}) \otimes \dots
	\otimes M_2(\mathbb{R}) \otimes \mathbb{C} $,
	$ n = 2^m \geq 1 $.\\
	El grupo graduador es
	$ (\mathbb{Z}_2^2)^m \times \mathbb{Z}_2 \cong \mathbb{Z}_2^{2m+1} $.
\item[(1-d)]
	$ M_n(\mathbb{C}) \cong M_2(\mathbb{R}) \otimes \dots
	\otimes M_2(\mathbb{R}) \otimes M_2(\mathbb{C}) $,
	$ n = 2^m \geq 2 $.\\
	El grupo graduador es
	$ (\mathbb{Z}_2^2)^{m-1} \times ( \mathbb{Z}_2 \times \mathbb{Z}_4 )
	\cong \mathbb{Z}_2^{2m-1} \times \mathbb{Z}_4 $.
\end{enumerate}
\end{teorema}

\begin{teorema}\label{teor:art_uno_dim1_isom}
Las graduaciones de división,
por grupos abelianos,
en el álgebra $M_{n/2}(\mathbb{H})$,
cuyas componentes homogéneas tienen dimensión $1$,
existen si y solo si $ n = 2^m \geq 2 $,
y están determinadas salvo isomorfismo
por el signo de los cuadrados de los elementos homogéneos.
De esta manera obtenemos una correspondencia biyectiva
entre las clases de isomorfismo
y las formas cuadráticas en $\mathbb{Z}_2^{2m}$
con polarización de tipo I e invariante de Arf $-1$.
Todas estas graduaciones pertenecen a la misma clase de equivalencia,
representada por (1-b) en la lista del teorema \ref{teor:art_uno_dim1_equiv}.
\end{teorema}

\section{Dimensión cuatro}

El siguiente paso en la clasificación
es considerar las graduaciones de división
cuyas componentes homogéneas tienen dimensión $4$.
Curiosamente este caso,
en el que la componente neutra es isomorfa a $\mathbb{H}$,
es más sencillo que el de dimensión $2$,
en el que la componente neutra es isomorfa a $\mathbb{C}$.
También en el teorema \ref{teor:art_uno_dim1_equiv},
el álgebra de matrices complejas
nos daba una clase de equivalencia más que las otras.
Esta dificultad extra es debida al centro.

Dada un álgebra $\mathcal{D}$,
su centro $Z(\mathcal{D})$
es el conjunto de elementos de dicha álgebra
que conmutan con todos los demás:
\begin{equation}
Z(\mathcal{D}) = \{ X \in \mathcal{D} \mid
XY = YX \text{ para todo } Y \in \mathcal{D} \}
\end{equation}
En las álgebras reales $M_n(\mathbb{R})$ y $M_n(\mathbb{H})$,
las únicas matrices que conmutan con todas las demás
son los múltiplos de la matriz identidad;
por tanto $ Z ( M_n(\mathbb{R}) ) $ y $ Z ( M_n(\mathbb{H}) ) $
son espacios vectoriales reales de dimensión $1$.
Por otro lado,
si consideramos $M_n(\mathbb{C})$ como álgebra
sobre el cuerpo de los números complejos,
entonces su centro también está formado
por los múltiplos de la matriz identidad
y es un espacio vectorial complejo de dimensión $1$.
Sin embargo,
como espacio vectorial real,
$ Z ( M_n(\mathbb{C}) ) $ tiene dimensión $2$.
Se dice que las álgebras simples reales
$M_n(\mathbb{R})$ y $M_n(\mathbb{H})$
son simples centrales;
mientras que el álgebra $M_n(\mathbb{C})$
es simple central como álgebra compleja,
pero no como álgebra real.

Volvamos al problema de la clasificación.
Estamos tratando el caso en el que
nuestra álgebra graduada de división $\mathcal{D}$
tiene una componente neutra $\mathcal{D}_e$ isomorfa a $\mathbb{H}$.
Consideramos el centralizador
$ C_{\mathcal{D}} (\mathcal{D}_e) $
de $\mathcal{D}_e$ en $\mathcal{D}$,
que se define,
de manera similar al centro,
como el conjunto de los elementos de $\mathcal{D}$
que conmutan con todos los elementos de $\mathcal{D}_e$:
\begin{equation}
C_{\mathcal{D}} (\mathcal{D}_e) = \{ X \in \mathcal{D} \mid
XY = YX \text{ para todo } Y \in \mathcal{D}_e \}
\end{equation}
El centralizador $ C_{\mathcal{D}} (\mathcal{D}_e) $
es una subálgebra de $\mathcal{D}$,
ya que si realizamos operaciones
con elementos de $ C_{\mathcal{D}} (\mathcal{D}_e) $
el resultado sigue siendo un elemento de $ C_{\mathcal{D}} (\mathcal{D}_e) $;
más aún,
la subálgebra $ C_{\mathcal{D}} (\mathcal{D}_e) $ es graduada,
porque si descomponemos
cualquier elemento de $ C_{\mathcal{D}} (\mathcal{D}_e) $
como suma de elementos homogéneos,
dichos elementos también pertenecen a $ C_{\mathcal{D}} (\mathcal{D}_e) $.
Está claro que si dos álgebras graduadas son isomorfas o equivalentes,
entonces los centralizadores de sus componentes neutras
también son isomorfos o equivalentes.

Gracias a que $\mathcal{D}_e$ es un álgebra simple central,
podemos aplicarle un resultado profundo
de la teoría de álgebras asociativas,
el teorema del doble centralizador
(ver por ejemplo \cite[teorema 4.7]{Jac1989}).
Adaptando dicho teorema a la estructura de graduación,
tenemos que el álgebra graduada $\mathcal{D}$
es isomorfa de manera natural al producto tensorial graduado
de $\mathcal{D}_e$ y su centralizador $ C_{\mathcal{D}} (\mathcal{D}_e) $:
\begin{equation}\label{ec:teor_doble_centr}
\mathcal{D} \cong \mathcal{D}_e \otimes C_{\mathcal{D}} (\mathcal{D}_e)
\end{equation}
Por lo tanto obtenemos el recíproco:
si los centralizadores de las componentes neutras
de dos de estas álgebras graduadas
son isomorfos o equivalentes,
entonces las álgebras también son isomorfas o equivalentes.
Además,
la ecuación \eqref{ec:teor_doble_centr} también nos indica que
las componentes homogéneas de $ C_{\mathcal{D}} (\mathcal{D}_e) $
tienen dimensión $1$.
Enunciemos estos argumentos en forma de teorema,
con el mínimo de hipótesis.

\begin{teorema}\label{teor:art_uno_dim4}
Sea $\mathcal{D}$ un álgebra asociativa real y unitaria
dotada de una graduación de división
cuyas componentes homogéneas tienen dimensión $4$.
Entonces el centralizador de la componente neutra
$ C_{\mathcal{D}} (\mathcal{D}_e) $
es también un álgebra graduada de división,
pero con componentes homogéneas de dimensión $1$.
Además,
el álgebra graduada $\mathcal{D}$
es isomorfa al producto tensorial graduado
$ \mathcal{D}_e \otimes C_{\mathcal{D}} (\mathcal{D}_e) $.
Por lo tanto,
las clases de isomorfismo y equivalencia de $\mathcal{D}$
están determinadas por las de
$ C_{\mathcal{D}} (\mathcal{D}_e) $.
\end{teorema}

Por ejemplo,
como $ M_n(\mathbb{R}) \cong \mathbb{H} \otimes M_{n/4}(\mathbb{H}) $,
juntando los teoremas
\ref{teor:art_uno_dim1_isom} y \ref{teor:art_uno_dim4}
deducimos que
existen graduaciones de división
por grupos abelianos en el álgebra $M_n(\mathbb{R})$
con componentes homogéneas de dimensión $4$
si y solo si $ n = 2^m \geq 4 $,
y que las clases de isomorfismo están en correspondencia biyectiva
con las formas cuadráticas en $\mathbb{Z}_2^{2m-2}$
con polarización de tipo I e invariante de Arf $-1$.
Análogamente,
cada una de las clases de equivalencia
del teorema \ref{teor:art_uno_dim1_equiv}
nos da exactamente una clase de equivalencia
para el caso de componentes homogéneas de dimensión $4$.

\section{Dimensión dos}

Finalmente analicemos el caso en el que
las componentes homogéneas tienen dimensión $2$.
Empezamos estudiando la graduación de división
en el álgebra real $ \mathcal{D} = M_2(\mathbb{C}) $
dada por la siguiente ecuación:
\begin{equation}\label{ec:grad_M2C_Z22}
\begin{split}
M_2(\mathbb{C}) = {} &
\left[
	\mathbb{R} \begin{pmatrix} 1 & 0 \\ 0 & 1 \end{pmatrix}
	\oplus
	\mathbb{R} \begin{pmatrix} 0 & -1 \\ 1 & 0 \end{pmatrix}
\right]
\oplus
\left[
	\mathbb{R} \begin{pmatrix} 0 & 1 \\ 1 & 0 \end{pmatrix}
	\oplus
	\mathbb{R} \begin{pmatrix} -1 & 0 \\ 0 & 1 \end{pmatrix}
\right]
\oplus
\\ &
\left[
	\mathbb{R} \begin{pmatrix} i & 0 \\ 0 & i \end{pmatrix}
	\oplus
	\mathbb{R} \begin{pmatrix} 0 & -i \\ i & 0 \end{pmatrix}
\right]
\oplus
\left[
	\mathbb{R} \begin{pmatrix} 0 & i \\ i & 0 \end{pmatrix}
	\oplus
	\mathbb{R} \begin{pmatrix} -i & 0 \\ 0 & i \end{pmatrix}
\right]
\end{split}
\end{equation}
El grupo graduador es
$ T = \mathbb{Z}_2^2 =
\langle g,f \mid g^2=e=f^2 , \allowbreak \, gf=fg \rangle $
y los grados se asignan en el siguiente orden:
$e$,~$g$; $f$,~$gf$.
A diferencia de lo que pasaba
cuando las componentes homogéneas tenían dimensión $1$,
ahora los signos de los cuadrados de los elementos homogéneos
no siempre están definidos.
En efecto,
en la componente de grado $e$ hay matrices cuyo cuadrado es positivo,
pero también hay otras cuyo cuadrado es negativo;
y lo mismo sucede en la componente de grado $f$.
Llamemos $K$ al conjunto de estos grados,
$ K = \{ e,f \} $.
Observamos que este conjunto es precisamente
el soporte del centralizador de la componente neutra:
\begin{equation}
K = \supp ( C_{\mathcal{D}} (\mathcal{D}_e) )
\end{equation}

Por otro lado,
dada cualquier matriz homogénea no nula de grado $g$,
su cuadrado es siempre un múltiplo positivo de la matriz identidad.
Ocurre lo mismo si en lugar de matrices de grado $g$
cogemos matrices de grado $gf$,
pero ahora obtenemos múltiplos negativos de $I$.
Por lo tanto podemos definir,
tomando matrices $X_t$ normalizadas,
la aplicación $ \nu : T \setminus K \to \{ \pm 1 \} $
mediante los signos de los cuadrados de los elementos homogéneos:
\begin{equation}
X_t^2 = \nu(t) I
\end{equation}

La clave está en que,
gracias al teorema \ref{teor:engros},
la aplicación $\nu$ sigue estando bien definida
en el caso general en el que $\mathcal{D}$ es cualquier
álgebra asociativa real de dimensión finita y simple
(excepto por algunas complicaciones cuando $T$ tiene elementos de orden $4$).
Además,
aplicando el teorema del doble centralizador se demuestra que $\nu$
determina el álgebra graduada $\mathcal{D}$ salvo isomorfismo.
Enunciemos esto en forma de teorema,
en los casos en los que $\mathcal{D}$ es simple central,
es decir,
$M_n(\mathbb{R})$ o $M_n(\mathbb{H})$,
ya que entonces el grupo graduador no tiene elementos de orden $4$.

\begin{teorema}
Las graduaciones de división,
por grupos abelianos,
en las álgebras asociativas reales de dimensión finita y simples centrales,
con componentes homogéneas de dimensión $2$,
están determinadas salvo isomorfismo
por los signos de los cuadrados de
los elementos homogéneos que no conmutan con la componente neutra.
\end{teorema}

No todas las aplicaciones $ \nu : T \setminus K \to \{ \pm 1 \} $
provienen de graduaciones de división,
sino solo las que cumplen determinadas propiedades,
que caracterizamos en el primer artículo \cite{Rod2016}.

En los casos simples centrales
$ \mathcal{D} = M_n(\mathbb{R}) $ y $ \mathcal{D} = M_n(\mathbb{H}) $,
todas estas graduaciones pertenecen a la misma clase de equivalencia,
que llamamos (2-a) y (2-b) respectivamente.
Sin embargo,
si $ \mathcal{D} = M_n(\mathbb{C}) $ la situación se complica.
Veamos cuál es la lista de las clases de equivalencia en este caso.

\begin{ejemplo}\label{ejem:grad_M2C_Z4}
La figura \ref{fig_es:grad_M2C_Z4}
define una graduación de división
en el álgebra real $M_2(\mathbb{C})$
por el grupo $\mathbb{Z}_4$.
Notemos que sus componentes homogéneas tienen dimensión $2$,
y que simplemente es un engrosamiento de la graduación
del ejemplo \ref{ejem:grad_M2C_dim1}.
\begin{figure}
\begin{align*}
M_2(\mathbb{C}) = {} &
\left[
	\mathbb{R} \begin{pmatrix} 1 & 0 \\ 0 & 1 \end{pmatrix}
	\oplus
	\mathbb{R} \begin{pmatrix} 0 & i \\ i & 0 \end{pmatrix}
\right]
\oplus
\\ &
\left[
	\mathbb{R} \begin{pmatrix} 1+i & 0 \\ 0 & -1-i \end{pmatrix}
	\oplus
	\mathbb{R} \begin{pmatrix} 0 & 1-i \\ -1+i & 0 \end{pmatrix}
\right]
\oplus
\\ &
\left[
	\mathbb{R} \begin{pmatrix} i & 0 \\ 0 & i \end{pmatrix}
	\oplus
	\mathbb{R} \begin{pmatrix} 0 & 1 \\ 1 & 0 \end{pmatrix}
\right]
\oplus
\\ &
\left[
	\mathbb{R} \begin{pmatrix} -1+i & 0 \\ 0 & 1-i \end{pmatrix}
	\oplus
	\mathbb{R} \begin{pmatrix} 0 & -1-i \\ 1+i & 0 \end{pmatrix}
\right]
\end{align*}
\caption{Graduación de división del ejemplo \ref{ejem:grad_M2C_Z4}.}
\label{fig_es:grad_M2C_Z4}
\end{figure}
\end{ejemplo}

\begin{teorema}\label{teor:art_uno_dim2_equiv}
Toda graduación de división,
por un grupo abeliano,
en el álgebra real $M_n(\mathbb{C})$,
con componentes homogéneas de dimensión $2$,
es equivalente a una, y solo una,
de las álgebras graduadas de la siguiente lista,
donde las graduaciones del primer factor de cada producto tensorial
son las de los ejemplos
\ref{ejem:grad_M2R_dim2} y \ref{ejem:grad_M2C_Z4},
mientras que el resto de factores están graduados como en los ejemplos
\ref{ejem:grad_M2R_dim1}, \ref{ejem:grad_C_dim1} y \ref{ejem:grad_M2C_dim1}.
\begin{enumerate}
\item[(2-c)]
	$ M_n(\mathbb{C}) \cong M_2(\mathbb{R}) \otimes M_2(\mathbb{R})
	\otimes \dots \otimes M_2(\mathbb{R}) \otimes \mathbb{C} $,
	$ n = 2^m \geq 2 $.\\
	El grupo graduador es
	$ \mathbb{Z}_2 \times (\mathbb{Z}_2^2)^{m-1} \times \mathbb{Z}_2
	\cong \mathbb{Z}_2^{2m} $.
\item[(2-d)]
	$ M_n(\mathbb{C}) \cong M_2(\mathbb{R}) \otimes M_2(\mathbb{R})
	\otimes \dots \otimes M_2(\mathbb{R}) \otimes M_2(\mathbb{C}) $,
	$ n = 2^m \geq 4 $.\\
	El grupo graduador es
	$ \mathbb{Z}_2 \times (\mathbb{Z}_2^2)^{m-2} \times
	( \mathbb{Z}_2 \times \mathbb{Z}_4 ) \cong
	\mathbb{Z}_2^{2m-2} \times \mathbb{Z}_4 $.
\item[(2-e)]
	$ M_n(\mathbb{C}) \cong M_2(\mathbb{C}) \otimes
	M_2(\mathbb{R}) \otimes \dots \otimes M_2(\mathbb{R}) $,
	$ n = 2^m \geq 2 $.\\
	El grupo graduador es
	$ \mathbb{Z}_4 \times (\mathbb{Z}_2^2)^{m-1}
	\cong \mathbb{Z}_2^{2m-2} \times \mathbb{Z}_4 $.
\item[(2-f)]
	Graduaciones en $M_n(\mathbb{C})$ como álgebra sobre
	el cuerpo de escalares de los números complejos $\mathbb{C}$.
\end{enumerate}
\end{teorema}

Si comparamos las listas de los teoremas
\ref{teor:art_uno_dim1_equiv} y \ref{teor:art_uno_dim2_equiv},
vemos que la clase de equivalencia (1-c) da lugar a la (2-c).
Es decir,
la situación es similar a lo que pasaba en los casos simples centrales.
Sin embargo,
la (1-d) da lugar a dos clases de equivalencia distintas,
(2-d) y (2-e).
Esta última ocurre cuando $ T \setminus K $
es precisamente el conjunto de elementos de orden $4$ de $T$,
lo que dificulta la definición del invariante $\nu$.
De hecho advertimos a quien se lea el primer artículo \cite{Rod2016}
de que este caso (2-e) es más complicado que los demás.

\section{Publicación}

En mayo de 2015 habíamos escrito el artículo,
así que lo dejamos de lado un tiempo
para revisarlo después una última vez antes de subirlo al Arxiv.
Entonces nos informaron de que
Yuri Bahturin y Mikhail Zaicev
estaban trabajando exactamente en el mismo problema que nosotros,
aunque en su caso estaban interesados solo
en la clasificación salvo equivalencia,
pero no salvo isomorfismo.
Mentiría si dijera que no me preocupé en aquel momento.
Sin embargo,
creo que esta coincidencia resultó ser muy beneficiosa.
El hecho de que dos grupos se interesen en un mismo problema
significa que dicha cuestión es importante.
Al final,
ambos publicamos nuestros artículos
\cite{BZ2016} y \cite{Rod2016} simultáneamente,
llegando a los mismos resultados
(aunque en \cite{BZ2016} la clase de equivalencia (2-d) fue pasada por alto)
usando técnicas distintas.
Además,
como consecuencia de esta coincidencia,
me fui tres meses de estancia
a la Memorial University of Newfoundland (Canadá)
en la primavera de 2016,
y continuamos la investigación de manera conjunta.

\selectlanguage{spanish}
\chapter{Segundo artículo}\label{cap:art_2}

\begin{enumerate}

\item[\cite{BKR2018a}]
Y. Bahturin, M. Kochetov\ and\ A. Rodrigo-Escudero,
Classification of involutions on graded-division simple real algebras,
Linear Algebra Appl. {\bf 546} (2018), 1--36.
\\
\url{https://arxiv.org/abs/1707.05526}
\\
\url{https://doi.org/10.1016/j.laa.2018.01.040}

\end{enumerate}

\section{Objetivo}\label{secc:art_2_obj}

Recordemos que el propósito final de mi tesis es
clasificar graduaciones en álgebras de Lie reales simples.
Para ello adaptamos los argumentos de la monografía \cite{EK2013}
del caso en el que el cuerpo de escalares son los números complejos
al caso en el que son los números reales.
En el capítulo \ref{cap:art_3} explicaremos este proceso con más detalle,
pero la idea es que podemos transferir
la clasificación de las graduaciones
en un álgebra asociativa a un álgebra de Lie,
a condición de que la primera esté dotada de una estructura adicional:
una involución.

Así,
uno de los pasos es estudiar las involuciones
que son compatibles con las graduaciones de división
que hemos visto en el capítulo \ref{cap:art_1}.
El análogo de esta cuestión sobre el cuerpo de los complejos
se resuelve en \cite[proposiciones 2.51 y 2.53]{EK2013}
(ver también \cite{BZ2006});
las involuciones se dividen en dos clases de equivalencia.
Sin embargo
si el cuerpo de escalares es $\mathbb{R}$
aparecen muchas más clases.
Por ello consideramos que esta clasificación
podía ser interesante por sí misma,
y nos planteamos el objetivo del segundo artículo \cite{BKR2018a}.

\begin{objetivo}\label{obj_es:art_2}
Clasificamos,
salvo isomorfismo y salvo equivalencia,
las involuciones
en las álgebras asociativas reales de dimensión finita simples
y graduadas de división,
cuando el grupo graduador es abeliano.
\end{objetivo}

\section{Involuciones}

A lo largo de este capítulo fijamos
un grupo abeliano $G$,
un álgebra asociativa real de dimensión finita y simple $\mathcal{D}$,
y una $G$-graduación de división $\Gamma$ en $\mathcal{D}$.
En otras palabras,
$\mathcal{D}$ es una de las álgebras graduadas
que clasificamos en el capítulo \ref{cap:art_1}.

Por un momento olvidémonos de la graduación.
Un antiautomorfismo del álgebra $\mathcal{D}$
es una aplicación $ \varphi : \mathcal{D} \to \mathcal{D} $
que es isomorfismo de espacios vectoriales
y que invierte el orden del producto.
Es decir,
$\varphi$ es biyectiva
y para todos $ X,Y \in \mathcal{D} $
y todo $ \lambda \in \mathbb{R} $
se satisfacen las siguientes condiciones:
$ \varphi(X+Y) = \varphi(X) + \varphi(Y) $,
$ \varphi( \lambda X ) = \lambda \varphi(X) $,
y $ \varphi(XY) = \varphi(Y) \varphi(X) $.
Si decimos que $\varphi$ es
un antiautomorfismo del álgebra graduada $\mathcal{D}$,
significa que estamos teniendo en cuenta su graduación,
y que por tanto se debe satisfacer una condición más,
que $ \varphi(X_g) \in \mathcal{D}_g $
para todo $ g \in G $ y todo $ X_g \in \mathcal{D}_g $.
Finalmente,
una involución $\varphi$ es un antiautomorfismo
tal que $ \varphi(\varphi(X)) = X $ para todo $ X \in \mathcal{D} $.

Por ejemplo,
la aplicación $ \varphi : M_2(\mathbb{C}) \to M_2(\mathbb{C}) $
dada por transponer y conjugar,
$ \varphi(X) = \overline{X^T} $ para toda matriz $ X \in M_2(\mathbb{C}) $,
es una involución en el álgebra real $M_2(\mathbb{C})$.
Observamos que esta involución respeta
la graduación de la ecuación \eqref{ec:grad_M2C_Z22}.
Sin embargo,
no respeta la graduación del ejemplo \ref{ejem:grad_M2C_dim1}.
De hecho se puede demostrar que,
si el soporte $T$ es isomorfo a $ \mathbb{Z}_2^a \times \mathbb{Z}_4 $
(lo que implica que el álgebra subyacente $\mathcal{D}$
es isomorfa a $M_n(\mathbb{C})$),
entonces todo antiautomorfismo en el álgebra graduada de división $\mathcal{D}$
actúa como la identidad en el centro de $\mathcal{D}$.
Este fue uno de los primeros resultados de la clasificación
que nos llamó la atención.
Otro ejemplo de involución,
esta vez en el álgebra graduada del ejemplo \ref{ejem:grad_H_dim1},
es la aplicación $ \varphi : \mathbb{H} \to \mathbb{H} $
dada por $ \varphi(a+bi+cj+dk) = a+bi+cj-dk $.

Sean $ \varphi : \mathcal{D} \to \mathcal{D} $
y $ \varphi' : \mathcal{D}' \to \mathcal{D}' $
involuciones en dos álgebras graduadas $\mathcal{D}$ y $\mathcal{D}'$,
cuyas graduaciones denotamos $\Gamma$ y $\Gamma'$.
Decimos que el par $(\Gamma,\varphi)$ es isomorfo
(respectivamente equivalente)
al par $(\Gamma',\varphi')$
si existe un isomorfismo
(respectivamente equivalencia)
de álgebras graduadas
$ \psi : \mathcal{D} \to \mathcal{D}' $
que también conmuta con las involuciones $\varphi$ y $\varphi'$,
es decir,
tal que $ \varphi'(\psi(X)) = \psi(\varphi(X)) $
para todo $ X \in \mathcal{D} $.

En el segundo artículo \cite{BKR2018a}
clasificamos los pares $(\Gamma,\varphi)$,
salvo isomorfismo y salvo equivalencia,
para todas las posibles involuciones $\varphi$
en el álgebra graduada $\mathcal{D}$.
Aunque aparecen muchos casos,
en dicho artículo hicimos el esfuerzo de escribir explícitamente todos,
para que la clasificación sirva como referencia.
En las siguientes secciones vamos a mostrar
algunos de los casos más ilustrativos.

\section{Caso (1-a)}

Supongamos que las componentes homogéneas
de la graduación $\Gamma$ tienen dimensión $1$,
y que el álgebra subyacente $\mathcal{D}$
es isomorfa a $M_n(\mathbb{R})$.
Es decir,
el álgebra graduada $\mathcal{D}$ es equivalente
al representante (1-a) del teorema \ref{teor:art_uno_dim1_equiv}.
Recordamos que,
en esta situación,
el soporte $T$ de $\Gamma$ es un grupo isomorfo a $\mathbb{Z}_2^{2m}$
($ n = 2^m $),
la aplicación $ \beta : T \times T \to \{ \pm 1 \} $
definida por la ecuación \eqref{ec:rel_conm}
es un bicarácter alternado de tipo I,
y la aplicación $ \mu : T \to \{ \pm 1 \} $
definida por la ecuación \eqref{ec:sign_cuadr}
es una forma cuadrática tal que
$ \beta_{\mu} = \beta $ y $ \Arf(\mu) = +1 $.
Además,
la clase de isomorfismo de $\Gamma$ está determinada por $\mu$.

Un cálculo sencillo nos dice que
las involuciones en el álgebra graduada $\mathcal{D}$
están en correspondencia biyectiva con
las formas cuadráticas $ \eta : T \to \{ \pm 1 \} $
tales que $ \beta_{\eta} = \beta $.
A saber,
la correspondencia viene dada por la siguiente ecuación,
donde $ X_t \in \mathcal{D}_t $:
\begin{equation}\label{ec:eta}
\varphi(X_t) = \eta(t) X_t
\end{equation}
Notemos que hay $2^{2m}$ posibles formas cuadráticas $\eta$
para el bicarácter alternado $\beta$ fijado,
es decir,
$2^{2m}$ involuciones no isomorfas entre sí.
En total,
el par $(\Gamma,\varphi)$ está determinado salvo isomorfismo por
dos formas cuadráticas $\mu$ y $\eta$ definidas en el mismo grupo,
el soporte $ T \cong \mathbb{Z}_2^{2m} $.

Más complicada es la clasificación salvo equivalencia.
Una conjetura razonable es
dividir las $2^{2m}$ formas cuadráticas $\eta$ en tres conjuntos:
\begin{enumerate}
	\item[(1)] $ \eta = \mu $.
	\item[(2)] $ \Arf(\eta) = +1 $ pero $ \eta \neq \mu $.
	\item[(3)] $ \Arf(\eta) = -1 $.
\end{enumerate}
Es decir,
por un lado tenemos el caso distinguido $ \eta = \mu $,
y por otro lado separamos
las formas cuadráticas que corresponderían
con una graduación de división en $M_n(\mathbb{R})$
de aquellas que corresponderían
con una graduación de división en $M_{n/2}(\mathbb{H})$.

En efecto,
estas son las tres clases de equivalencia de involuciones,
pero la demostración puede ser laboriosa,
porque hay muchos casos que analizar.
Afortunadamente
encontramos una manera de tratar con todos ellos a la vez.
Aquí no vamos a repetir los argumentos,
pero al menos indiquemos que la idea clave
está en considerar los dos lemas siguientes.

\begin{lema}
Sean $\mu$ y $\eta$ dos formas cuadráticas diferentes
en un grupo abeliano finito $T$
tales que $ \beta_{\mu} = \beta_{\eta} $.
Entonces $ \{ t \in T \mid \mu(t) = \eta(t) \} $
es un subgrupo de $T$ de índice $2$.
\end{lema}

\begin{lema}
Sea $\beta$ un bicarácter alternado de tipo I
en un grupo abeliano finito $T$.
Entonces la siguiente aplicación es biyectiva:
\begin{equation}
\begin{split}
T & \longrightarrow \{ S \mid
S \text{ es un subgrupo de } T \text{ de índice } 1 \text{ o } 2 \}
\\
u & \longmapsto u^{\perp} = \{ v \in T \mid \beta(u,v) = 1 \}
\end{split}
\end{equation}
\end{lema}

Finalmente,
subrayemos de nuevo que
la involución correspondiente a $ \eta = \mu $ es distinguida.
Si tomamos como $\mathcal{D}$
el producto tensorial (1-a) del teorema \ref{teor:art_uno_dim1_equiv},
y lo identificamos con $M_n(\mathbb{R})$ de la manera natural,
entonces esta involución distinguida $\varphi$
es la transposición de matrices,
$ \varphi(X) = X^T $ para toda matriz $ X \in M_n(\mathbb{R}) $.

\section{Caso (1-c)}

El caso (1-c),
en el que el álgebra subyacente $\mathcal{D}$
es isomorfa a $M_n(\mathbb{C})$
y el soporte $T$ es un grupo isomorfo a $\mathbb{Z}_2^{2m+1}$
($ n = 2^m $),
es muy similar.
De hecho las ecuaciones \eqref{ec:rel_conm},
\eqref{ec:sign_cuadr} y \eqref{ec:eta}
siguen siendo válidas,
y la clase de isomorfismo está determinada por
las formas cuadráticas $\mu$ y $\eta$.
Señalemos las diferencias.
Ahora el bicarácter alternado $\beta$ es de tipo II,
y las condiciones que tiene que satisfacer $\mu$ son
$ \beta_{\mu} = \beta $ y $ \mu(f_{\beta}) = -1 $
(recordemos que $ \rad(\beta) = \{ e , f_{\beta} \} $).
El invariante de Arf de $\mu$ es siempre $0$.

La condición $ \mu(f_{\beta}) = -1 $ es interesante.
Es debida a que el elemento central $iI$
es homogéneo de grado $f_{\beta}$
y a que $ (iI)^2 = -I $.
Sin embargo
cuando calculamos la ecuación \eqref{ec:eta},
nos damos cuenta de que la única condición que tiene que satisfacer $\eta$
es $ \beta_{\eta} = \beta $;
así que es posible que $ \eta(f_{\beta}) = +1 $.
De hecho,
las $2^{2m+1}$ formas cuadráticas $\eta$
se dividen en cuatro clases de equivalencia:
\begin{enumerate}
	\item[(1)] $ \eta = \mu $.
	\item[(2)] $ \eta(f_{\beta}) = -1 $ pero $ \eta \neq \mu $.
	\item[(3)] $ \eta(f_{\beta}) = +1 $ y $ \Arf(\eta) = +1 $.
	\item[(4)] $ \eta(f_{\beta}) = +1 $ y $ \Arf(\eta) = -1 $.
\end{enumerate}

Para llegar a esta conclusión,
calculamos las graduaciones de división que corresponderían con
las formas cuadráticas $\mu$ tales que
$ \beta_{\mu} = \beta $ y $ \mu(f_{\beta}) = +1 $
(en lugar de $ \mu(f_{\beta}) = -1 $).
Obtuvimos dos clases de equivalencia,
una para el caso $ \Arf(\mu) = +1 $,
con álgebra subyacente $ M_n(\mathbb{R}) \times M_n(\mathbb{R}) $,
y la otra para el caso $ \Arf(\mu) = -1 $,
con álgebra subyacente $ M_{n/2}(\mathbb{H}) \times M_{n/2}(\mathbb{H}) $.
Vemos que estas dos clases están reflejadas
en los casos (3) y (4) de la lista anterior.
Por lo tanto,
la clasificación de las graduaciones de división
en las álgebras (no necesariamente simples)
cuyo centro tiene dimensión $1$ o $2$
nos ayuda a entender la de
los pares $(\Gamma,\varphi)$ en las álgebras simples.
Así que también abordamos
parte de esta clasificación de graduaciones de división
en el segundo artículo \cite{BKR2018a}.

Finalmente notemos que,
en la literatura,
una forma cuadrática $\mu$ se dice regular si
o bien su polarización $\beta$ es de tipo I,
o bien $\beta$ es de tipo II y además $ \mu(f_{\beta}) = -1 $.
Esta definición puede parecer un poco extraña.
Sin embargo,
ahora vemos que las formas cuadráticas regulares
son precisamente aquellas que corresponden,
mediante la ecuación \eqref{ec:sign_cuadr},
con graduaciones de división en álgebras simples.

\section{Nuevas graduaciones de división}

Recordemos que los representantes de las clases de equivalencia
del teorema \ref{teor:art_uno_dim1_equiv}
están escritos como productos tensoriales graduados
de copias de las cuatro álgebras graduadas de división de los ejemplos
\ref{ejem:grad_M2R_dim1}, \ref{ejem:grad_H_dim1},
\ref{ejem:grad_C_dim1} y \ref{ejem:grad_M2C_dim1}.
Ahora nos aparecen más graduaciones de división
en álgebras que no son simples,
concretamente en
$ M_n(\mathbb{R}) \times M_n(\mathbb{R}) $,
$ M_{n/2}(\mathbb{H}) \times M_{n/2}(\mathbb{H}) $ y
$ M_n(\mathbb{R}) \times M_{n/2}(\mathbb{H}) $.
Para expresarlas como productos tensoriales graduados,
necesitamos tres nuevos tipos de factores,
las álgebras graduadas de división del ejemplo \ref{ejem:grad_semisimple}.
Subrayamos que es curioso que halláramos
las dos graduaciones de división de la figura \ref{fig_es:grad_semisimple}
porque estuviéramos buscando las involuciones
en el álgebra graduada de la figura \ref{fig_es:grad_M2C_dim1}.

\begin{ejemplo}\label{ejem:grad_semisimple}
La siguiente ecuación define una graduación de división
en el álgebra $ \mathbb{R} \times \mathbb{R} $
por el grupo $\mathbb{Z}_2$:
\begin{equation}
\mathbb{R} \times \mathbb{R} = \mathbb{R}(1,1) \oplus \mathbb{R}(1,-1)
\end{equation}
Además,
la figura \ref{fig_es:grad_semisimple} muestra dos graduaciones de división
por el grupo $ \mathbb{Z}_2 \times \mathbb{Z}_4 =
\langle a,b \mid a^2=e=b^4 , \allowbreak \, ab=ba \rangle $
en las álgebras $ M_2( \mathbb{R} \times \mathbb{R} ) $
y $ M_2(\mathbb{R}) \times \mathbb{H} $.
Los grados se asignan en el mismo orden que seguimos
en el ejemplo \ref{ejem:grad_M2C_dim1}:
$e$,~$a$; $b$,~$ab$; $b^2$,~$ab^2$; $b^3$,~$ab^3$.
\begin{figure}
\begin{align*}
M_2( \mathbb{R} \times \mathbb{R} ) = {} &
	\mathbb{R} \begin{pmatrix} (1,1) & (0,0) \\ (0,0) & (1,1) \end{pmatrix}
	\oplus
	\mathbb{R} \begin{pmatrix} (1,1) & (0,0) \\ (0,0) & (-1,-1) \end{pmatrix}
	\oplus
\\
&
	\mathbb{R} \begin{pmatrix} (0,0) & (1,-1) \\ (1,1) & (0,0) \end{pmatrix}
	\oplus
	\mathbb{R} \begin{pmatrix} (0,0) & (1,-1) \\ (-1,-1) & (0,0) \end{pmatrix}
	\oplus
\\
&
	\mathbb{R} \begin{pmatrix} (1,-1) & (0,0) \\ (0,0) & (1,-1) \end{pmatrix}
	\oplus
	\mathbb{R} \begin{pmatrix} (1,-1) & (0,0) \\ (0,0) & (-1,1) \end{pmatrix}
	\oplus
\\
&
	\mathbb{R} \begin{pmatrix} (0,0) & (1,1) \\ (1,-1) & (0,0) \end{pmatrix}
	\oplus
	\mathbb{R} \begin{pmatrix} (0,0) & (1,1) \\ (-1,1) & (0,0) \end{pmatrix}
\end{align*}
\begin{align*}
M_2(\mathbb{R}) \times \mathbb{H} = {} &
	\mathbb{R} \left( \begin{pmatrix} 1 & 0 \\ 0 & 1 \end{pmatrix} , 1 \right)
	\oplus
	\mathbb{R} \left( \begin{pmatrix} 0 & -1 \\ 1 & 0 \end{pmatrix} , i \right)
	\oplus
\\
&
	\mathbb{R} \left( \begin{pmatrix} 1 & 0 \\ 0 & -1 \end{pmatrix} , j \right)
	\oplus
	\mathbb{R} \left( \begin{pmatrix} 0 & 1 \\ 1 & 0 \end{pmatrix} , k \right)
	\oplus
\\
&
	\mathbb{R} \left( \begin{pmatrix} 1 & 0 \\ 0 & 1 \end{pmatrix} , -1 \right)
	\oplus
	\mathbb{R} \left( \begin{pmatrix} 0 & -1 \\ 1 & 0 \end{pmatrix} , -i \right)
	\oplus
\\
&
	\mathbb{R} \left( \begin{pmatrix} 1 & 0 \\ 0 & -1 \end{pmatrix} , -j \right)
	\oplus
	\mathbb{R} \left( \begin{pmatrix} 0 & 1 \\ 1 & 0 \end{pmatrix} , -k \right)
\end{align*}
\caption{Graduaciones de división del ejemplo \ref{ejem:grad_semisimple}.}
\label{fig_es:grad_semisimple}
\end{figure}
\end{ejemplo}

\section{Caso (2-f)}

Finalmente consideremos el caso (2-f),
en el que el álgebra real $\mathcal{D}$
es isomorfa a $M_n(\mathbb{C})$,
y la componente neutra de la graduación $\Gamma$
coincide con el centro de $\mathcal{D}$.
Esto significa que $\Gamma$ puede ser interpretada
como una graduación del álgebra compleja $M_n(\mathbb{C})$.

En esta situación,
las relaciones de conmutación de la ecuación \eqref{ec:rel_conm}
son suficientes para determinar el álgebra graduada $\mathcal{D}$.
La forma cuadrática $\mu$ de la ecuación \eqref{ec:sign_cuadr}
ya no está definida.
Ahora $T$ es un grupo isomorfo a
$ \mathbb{Z}_{\ell_1}^2 \times \dots \times \mathbb{Z}_{\ell_r}^2 $,
donde $ \ell_1 \cdots \ell_r = n $.
Mientras que el bicarácter alternado $\beta$
puede tomar valores en $ \mathbb{C} \setminus \{ 0 \} $,
y la única condición que tiene que satisfacer es $ \rad(\beta) = \{ e \} $.

Si $\varphi$ es una involución en el álgebra graduada $\mathcal{D}$,
entonces o bien $ \varphi(iI) = +iI $ o bien $ \varphi(iI) = -iI $.
La clasificación de los pares $(\Gamma,\varphi)$
tales que $ \varphi(iI) = +iI $
ya había sido estudiada en \cite[proposiciones 2.51 y 2.53]{EK2013},
porque dichas involuciones respetan la estructura compleja de $\mathcal{D}$.
Es muy similar a las clasificaciones de los casos
que hemos visto anteriormente,
en los que las componentes homogéneas tienen dimensión $1$.

Centrémonos en las involuciones $\varphi$ tales que $ \varphi(iI) = -iI $.
Su clasificación es totalmente diferente.
Es más,
creo que la sección 8 de \cite{BKR2018a}
es una de las partes más interesantes del artículo;
por cierto,
puede ser leída independientemente del resto del texto.

Definimos $ T_{[2]} =  \{ t\in T \mid t^2 = e \} $.
Las clases de isomorfismo de los pares $(\Gamma,\varphi)$
están en correspondencia biyectiva con
los subgrupos $S$ de $T_{[2]}$ de índice $1$ o $2$
a través de la siguiente ecuación:
\begin{equation}
S = \{ t \in T_{[2]} \mid
\exists X \in \mathcal{D}_t
\text{ tal que } X^2 = +I
\text{ y } \varphi(X) = X \}
\end{equation}

Este es un buen ejemplo de algunas de las dificultades
que aparecen en la investigación matemática.
La demostración no es tan complicada,
el verdadero reto fue darnos cuenta de qué es lo que teníamos que demostrar,
sin siquiera saber si este era un problema resoluble.
Por otro lado,
tenemos que advertir de que la clasificación salvo equivalencia
es una de la partes más difíciles del artículo.

Notemos que hay una clase de isomorfismo distinguida,
la correspondiente a $ S = T_{[2]} $.
Las involuciones que pertenecen a esta clase
tienen signatura $ \sqrt{ \vert T_{[2]} \vert } $,
mientras que el resto de involuciones del álgebra graduada $\mathcal{D}$
tienen signatura $0$.

\selectlanguage{spanish}
\chapter{Tercer artículo}\label{cap:art_3}

\begin{enumerate}

\item[\cite{BKR2018b}]
Y. Bahturin, M. Kochetov\ and\ A. Rodrigo-Escudero,
Gradings on classical central simple real Lie algebras,
J. Algebra {\bf 506} (2018), 1--42.
\\
\url{https://arxiv.org/abs/1707.07909}
\\
\url{https://doi.org/10.1016/j.jalgebra.2018.02.036}

\end{enumerate}

\section{Objetivo}\label{secc:art_3_obj}

Un álgebra de Lie $\mathcal{L}$ es
un álgebra (no necesariamente asociativa)
cuyo producto bilineal
$ [ \cdot , \cdot ] : \mathcal{L} \times \mathcal{L} \to \mathcal{L} $
satisface las siguientes dos propiedades:
\begin{enumerate}
\item Anticonmutatividad:
$ [x,x] = 0 $ para todo $ x \in \mathcal{L} $.
\item Identidad de Jacobi:
$ [[x,y],z] + [[y,z],x] + [[z,x],y] = 0 $
para todos $ x,y,z \in \mathcal{L} $.
\end{enumerate}

Killing (1890) y Cartan \cite{Car1894} clasificaron
las álgebras de Lie complejas (de dimensión finita) semisimples
gracias a una graduación:
la descomposición en espacios de raíces
relativa a una subálgebra de Cartan.
Aparecen cuatro familias infinitas de álgebras de Lie simples,
$\{A_r\}_{r=1}^{\infty}$,
$\{B_r\}_{r=2}^{\infty}$,
$\{C_r\}_{r=3}^{\infty}$ y
$\{D_r\}_{r=4}^{\infty}$,
llamadas clásicas,
y cinco álgebras de Lie simples excepcionales,
$E_6$, $E_7$, $E_8$, $F_4$ y $G_2$.
Por otro lado,
en \cite{Wed1908} Wedderburn resolvió
el problema análogo en el caso asociativo.
Pese a que el resultado es más sencillo,
todas las álgebras asociativas complejas (de dimensión finita) simples
pertenecen a la misma familia infinita
$ \{ M_n(\mathbb{C}) \}_{n=1}^{\infty} $,
es 14 años posterior al de Cartan.
Este es claramente un buen ejemplo
de la importancia que tienen las álgebras de Lie
dentro de las matemáticas y la física.

Nosotros estamos interesados en
las álgebras de Lie reales (de dimensión finita) simples.
El centroide de una tal álgebra $\mathcal{L}$
es o bien isomorfo a $\mathbb{R}$ o bien a $\mathbb{C}$.
En el segundo caso $\mathcal{L}$ es simplemente
un álgebra sobre el cuerpo de los complejos,
pero considerada como un álgebra sobre el cuerpo de los reales.
En el primer caso $\mathcal{L}$ es una forma real de su complexificación
$ \mathbb{C} \otimes_{\mathbb{R}} \mathcal{L} $,
y se dice que $\mathcal{L}$ es simple central.
El tipo de un álgebra de Lie real simple
es el del álgebra compleja del que proviene.
Así,
tanto el álgebra compleja $E_6$,
vista como álgebra real,
como sus formas reales son álgebras de tipo $E_6$.

Recordemos que nuestro objetivo no es clasificar las álgebras de Lie,
porque estas ya han sido extensamente estudiadas,
sino sus posibles graduaciones.
En concreto,
el resultado del tercer artículo \cite{BKR2018b} es el siguiente.

\begin{objetivo}\label{obj_es:art_3}
Para todo grupo abeliano $G$,
clasificamos salvo isomorfismo todas las $G$-graduaciones
en las álgebras de Lie simples centrales clásicas,
excepto las de tipo $D_4$,
sobre el cuerpo de los números reales
(o cualquier cuerpo real cerrado).
\end{objetivo}

Remarquemos que se trata de un artículo muy técnico,
en el que necesitamos muchos parámetros para enunciar los teoremas.
La principal diferencia con respecto a
la clasificación análoga sobre el cuerpo de los complejos
(ver \cite{BK2010} y \cite{EK2013})
es la aparición de un parámetro $\sigma$ de signaturas.

\section{Álgebras asociativas con involución}\label{secc:alg_inv}

Veamos cómo
la clasificación de graduaciones en álgebras de Lie
se reduce a
la clasificación de graduaciones en álgebras asociativas con involución.

Sea $\mathcal{R}$ un álgebra asociativa,
con producto denotado mediante yuxtaposición.
Consideramos el nuevo producto en $\mathcal{R}$
dado por $ [x,y] = xy-yx $ para todos $ x,y \in \mathcal{R} $.
Entonces $\mathcal{R}$,
dotada de este nuevo producto,
es un álgebra de Lie,
que denotamos mediante $\mathcal{R}^{(-)}$.
Si además $\varphi$ es una involución en $\mathcal{R}$,
podemos considerar el espacio de elementos antisimétricos:
\begin{equation}
\Skew(\mathcal{R},\varphi) = \{ x \in \mathcal{R} \mid \varphi(x) = -x \}
\end{equation}
Tenemos que
$\Skew(\mathcal{R},\varphi)^{(-)}$ es una subálgebra
del álgebra de Lie $\mathcal{R}^{(-)}$.
Es bien conocido que las
álgebras de Lie reales simples centrales clásicas
se pueden expresar mediante esta construcción.
Por ejemplo,
las de tipo $C_r$ vienen dadas por tomar
o bien $ \mathcal{R} = M_{2r}(\mathbb{R}) $
o bien $ \mathcal{R} = M_r(\mathbb{H}) $
y $\varphi$ simpléctica.

Toda graduación en $(\mathcal{R},\varphi)$ induce
una graduación en $ \mathcal{L} = \Skew(\mathcal{R},\varphi)^{(-)} $
por restricción.
La teoría de esquemas afines en grupos permite demostrar que
esta correspondencia es en realidad una biyección
entre las graduaciones de grupos abelianos en $(\mathcal{R},\varphi)$
y las graduaciones en $\mathcal{L}$
(excepto si $\mathcal{L}$ es de tipo $A_1$ o $D_4$).
Si fijamos un grupo abeliano $G$,
obtenemos una biyección entre las respectivas
clases de isomorfismo de $G$-graduaciones.

Esta es la razón por la que en los artículos anteriores
clasificábamos graduaciones en álgebras asociativas.
Por supuesto que dichas clasificaciones son interesantes por sí mismas,
pero nuestra motivación era usarlas en el estudio de las álgebras de Lie.
Además,
ahora podemos explicar por qué siempre hemos asumido que
los grupos graduadores eran abelianos:
el soporte de una graduación en un álgebra de Lie simple
siempre genera un subgrupo abeliano del grupo graduador
(ver por ejemplo \cite[lema 2.1]{BZ2006} o \cite[proposición 1]{DM2006}).

\section{Resumen}

En esta sección vamos a dar una breve idea
de los ingredientes que aparecen en la clasificación.

Queremos caracterizar salvo $G$-isomorfismo
el álgebra graduada con involución $(\mathcal{R},\varphi)$.
Olvidémonos de $\varphi$ por el momento.
Escribimos $\mathcal{R}$ como $M_k(\mathcal{D})$,
de la misma forma que en la sección \ref{secc:est_art}.
Recordemos que,
como álgebra $G$-graduada,
$M_k(\mathcal{D})$ estaba definida por
el álgebra graduada de división $\mathcal{D}$
y $k$ elementos del grupo abeliano $G$.
En el primer artículo \cite{Rod2016} clasificamos las posibles $\mathcal{D}$.
Llamamos $T$ al soporte de $\mathcal{D}$
y fijamos en cada componente homogénea $\mathcal{D}_t$
un elemento homogéneo $X_t$.
Curiosamente,
en algunos casos el $X_e$ que escogemos no es la unidad
(esto sucede si $\mathcal{L}$ es
$\mathfrak{sp}_{2r}(\mathbb{R})$ o $ \mathfrak{u}^* (r) $,
y $ \mathcal{D}_e \cong \mathbb{C} $).
En realidad,
la clase de isomorfismo de $M_k(\mathcal{D})$
no cambia si permutamos los $k$ elementos de $G$,
o si multiplicamos cualquiera de ellos por un elemento de $T$.
Es decir,
la información está contenida en una ``función de multiplicidad''
$ \kappa : G/T \to \mathbb{Z}_{ \geq 0 } $.
Si enumeramos las multiplicidades no nulas,
$ \kappa(x_i) = k_i $ con $ 1 \leq i \leq s $,
entonces su suma total es $ k_1 + \dots + k_s = k $.

Tengamos ahora en cuenta la involución $\varphi$.
Después de algunos cálculos vemos que podemos escribirla como
$ \varphi(X) = \Phi^{-1} \varphi_0(X^T) \Phi $,
donde $\Phi$ es una matriz de $M_k(\mathcal{D})$ diagonal por bloques
y $\varphi_0$ es una involución de $\mathcal{D}$,
que actúa sobre $ X^T \in M_k(\mathcal{D}) $ elemento a elemento.
Recordemos que clasificamos las posibles $\varphi_0$
en el segundo artículo \cite{BKR2018a}.
En muchos casos vamos a poder elegir que
la involución $\varphi_0$ sea distinguida.
Los bloques de la matriz $\Phi$ son de la forma $ X_{t_i} S_i $,
donde las $S_i$ son matrices de tamaño $ k_i \times k_i $.
Como ejemplo de los posibles valores que puede tomar $S_i$
podemos mencionar las siguientes dos matrices;
en la primera $ k_i = p_i + q_i $,
y en la segunda $k_i$ tiene que ser par:
\begin{equation*}
\begin{pmatrix} I_{p_i} & 0 \\ 0 & - I_{q_i} \end{pmatrix}
\qquad \qquad
\begin{pmatrix} 0 & I_{ k_i / 2 } \\ - I_{ k_i / 2 } & 0 \end{pmatrix}
\end{equation*}
Definimos una ``función de signatura'' $ \sigma : G/T \to \mathbb{Z} $
como $ \sigma(x_i) = p_i - q_i $ para los valores de $i$ con signatura,
y $ \sigma(x) = 0 $ para todos los demás $ x \in G/T $.
Además,
necesitamos otros dos ingredientes para describir $\varphi$,
un parámetro $ \delta \in \{ \pm 1 \} $
dado por $ \varphi_0(\Phi^T) = \delta \Phi $,
y un elemento $ g_0 \in G $ que en cierto sentido recoge el grado de $\Phi$.
En total,
el álgebra graduada con involución $(\mathcal{R},\varphi)$
viene dada por seis parámetros:
$\mathcal{D}$, $\varphi_0$, $g_0$, $\kappa$, $\sigma$, $\delta$.

El siguiente paso es considerar dos de estas álgebras graduadas con involución,
$ (\mathcal{R},\varphi) = M( \mathcal{D} , \allowbreak
\varphi_0 , \allowbreak g_0 , \allowbreak
\kappa , \allowbreak \sigma , \allowbreak \delta ) $ y
$ (\mathcal{R}',\varphi') = M( \mathcal{D}' , \allowbreak
\varphi_0' , \allowbreak g_0' , \allowbreak
\kappa' , \allowbreak \sigma' , \allowbreak \delta' ) $,
y determinar cuándo son isomorfas.
Para ello podemos suponer que $ \mathcal{D} = \mathcal{D}' $,
$ \varphi_0 = \varphi_0' $ y $ \delta = \delta' $,
pero por diferentes motivos.
Por un lado,
si las álgebras graduadas de división $\mathcal{D}$ y $\mathcal{D}'$
no son isomorfas,
entonces tampoco lo son $\mathcal{R}$ y $\mathcal{R}'$.
Por lo tanto la clasificación total será la suma de las clasificaciones
según $\mathcal{D}$ recorre todas las posibles clases de isomorfismo.
Por otro lado,
podemos cambiar la involución $\varphi_0$ y elegir la que más nos convenga
(si es posible elegimos la involución distinguida de $\mathcal{D}$).
Si escogiéramos otra,
obtendríamos la misma clasificación,
pero parametrizada de distinta forma.
Finalmente también podemos asumir que $\delta$ está fijada;
curiosamente si $\mathcal{L}$ es de tipo $A_r$,
esto es debido al segundo motivo,
mientras que si $\mathcal{L}$ es de tipo $B_r$, $C_r$ o $D_r$,
es por el primer motivo.

La cuestión está en los parámetros $(g_0,\kappa,\sigma)$.
La función de signatura $\sigma$ no es el invariante adecuado
para resolver la condición de isomorfismo,
ya que el signo de $ p_i - q_i $
depende del representante que tomemos
en la clase lateral $ x_i \in G/T $.
Para evitar esta ambigüedad,
consideramos un nuevo parámetro,
que llamamos ``función de signatura extendida''
$ \tilde{\sigma} : G \to \mathbb{Z} $,
y que es equivalente a $\sigma$.
El resultado lo presentamos mediante
la acción de un grupo que actúa sobre los parámetros,
de forma que $(\mathcal{R},\varphi)$ y $(\mathcal{R}',\varphi')$
son isomorfas si y solo si
$(g_0,\kappa,\sigma)$ y $(g_0',\kappa',\sigma')$
están en la misma órbita.
Por ejemplo,
en los tipos $B_r$, $C_r$ y $D_r$
esta acción consiste en la suma de
la acción natural de $ g \in G $ sobre $\kappa$ y $\tilde{\sigma}$
(que también reemplaza $g_0$ por $ g^{-2} g_0 $),
y la acción de cambiar todas las signaturas de signo.

Finalmente particularizamos la clasificación
a cada uno de los tipos de
álgebras de Lie reales simples centrales clásicas.
En algunos casos tenemos que calcular
la signatura global de la involución $\varphi$
en función de los parámetros,
lo que por cierto nos da una bonita fórmula.

\selectlanguage{spanish}
\chapter{Cuarto artículo}\label{cap:art_4}

\begin{enumerate}

\item[\cite{ER2018}]
A. Elduque\ and\ A. Rodrigo-Escudero,
Clifford algebras as twisted gro\-up algebras and the Arf invariant,
Adv. Appl. Clifford Algebr. {\bf 28} (2018), no.~2, 28:41.
\\
\url{https://arxiv.org/abs/1801.07002}
\\
\url{https://doi.org/10.1007/s00006-018-0862-y}

\end{enumerate}

\section{Introducción}\label{secc:art_4_intr}

El propósito de los tres primeros artículos de mi tesis
ha sido estudiar las graduaciones en las álgebras reales y clasificarlas.
Ahora abordamos un tema más práctico;
el cuarto artículo \cite{ER2018}
es un ejemplo de aplicación de las graduaciones
para resolver problemas en otras áreas de las matemáticas.

Mientras clasificábamos graduaciones de división en álgebras asociativas,
nos dimos cuenta de que estas están íntimamente relacionadas
con las álgebras de Clifford \cite[observación 17]{Rod2016}.
Así que una vez alcanzada la meta inicial del doctorado,
decidimos volver a analizar con más detalle esta conexión.
El resultado es un artículo más relajado que los tres anteriores,
ya que los teoremas que aparecen ya eran conocidos.
Como vamos a ver,
gracias a las graduaciones podemos demostrarlos
de una manera alternativa y bastante sencilla,
en la que el invariante de Arf juega un papel fundamental.

Este artículo posee
una peculiaridad muy poco frecuente entre los artículos matemáticos:
se puede explicar en una charla.
De hecho este capítulo es una adaptación
de una conferencia que impartí en febrero de 2018
en el seminario que realizamos entre los estudiantes de doctorado.

\begin{objetivo}\label{obj_es:art_4}
En este capítulo mostraremos cómo la teoría de graduaciones
permite dar demostraciones alternativas a teoremas clásicos.
Así,
repasaremos las propiedades básicas de las álgebras de Clifford
y veremos que,
en el caso real,
estas están determinadas por el invariante de Arf.
\end{objetivo}

\section{Graduaciones de división en álgebras de Clifford}

Empezamos recordando la definición de álgebra de Clifford.
Sea $V$ un espacio vectorial de dimensión finita $N$
sobre un cuerpo $\mathbb{F}$ de característica distinta de $2$.
Más adelante nos restringiremos al caso $ \mathbb{F} = \mathbb{R} $,
pero de momento los argumentos son válidos en general.
Sea $ Q : V \to \mathbb{F} $ una forma cuadrática en $V$.
Esto es equivalente a decir que existe una forma bilineal simétrica
$ B : V \times V \to \mathbb{F} $ satisfaciendo,
para todo $ v \in V $,
la siguiente ecuación:
\begin{equation}
2 Q(v) = B(v,v)
\end{equation}
Podemos recuperar $B$ a partir de $Q$,
ya que para todos $ u,v \in V $ tenemos la siguiente fórmula:
\begin{equation}\label{ec:form_bil_form_cuadr}
B(u,v) = Q(u+v) - Q(u) - Q(v)
\end{equation}
Notemos que la ecuación \eqref{ec:polarizacion}
no es más que la ecuación \eqref{ec:form_bil_form_cuadr}
escrita con notación multiplicativa.

Sea $ T(V) = \mathbb{F} \oplus V \oplus ( V \otimes V)
\oplus ( V \otimes V \otimes V ) \oplus \dots $
el álgebra tensorial de $V$,
y sea $I(Q)$ el ideal bilátero de $T(V)$
generado por todos los elementos de la forma
$ v \otimes v - Q(v)1 $ con $ v \in V $.
El álgebra de Clifford de $(V,Q)$ es $ \Cl(V,Q) = T(V) / I(Q) $.
Así,
$\Cl(V,Q)$ es un álgebra asociativa sobre el cuerpo $\mathbb{F}$,
pero de momento no sabemos cuál es su dimensión.
En los dos próximos párrafos vamos a demostrar que
el álgebra $\Cl(V,Q)$ es unitaria,
es decir,
que existe un elemento distinto de cero y neutro para el producto,
que como siempre denotamos $ 1 \in \Cl(V,Q) $.
Esta cuestión es equivalente a demostrar que
$ 1 \in T(V) $ no pertenece a $I(Q)$,
y es algo más complicada de lo que parece.
Como corolario obtenemos que la dimensión del álgebra $\Cl(V,Q)$
es estrictamente mayor que $0$.

Primero observamos que el álgebra de Clifford
satisface la siguiente propiedad universal:
Para cualquier $\mathbb{F}$-álgebra $\mathcal{A}$ asociativa y unitaria
dotada de una aplicación $\mathbb{F}$-lineal $ f : V \to \mathcal{A} $
tal que $ f(v)^2 = Q(v) 1 $ para todo $ v \in V $,
existe un único homomorfismo de $\mathbb{F}$-álgebras unitarias
$ \varphi : \Cl(V,Q) \to \mathcal{A} $ tal que
$ f(v) = \varphi(v+I(Q)) $ para todo $ v \in V $.
La existencia de una tal álgebra unitaria $\mathcal{A}$
implicará que $ 1 \in T(V) $ no pertenece a $I(Q)$,
ya que $\varphi(1+I(Q))$ es la unidad de $\mathcal{A}$,
que es diferente del cero de $\mathcal{A}$.

Ahora vamos a ver,
siguiendo \cite[página 38]{Che1954},
que podemos tomar como $\mathcal{A}$
el álgebra de endomorfismos del álgebra exterior $\Lambda(V)$.
Sea $ B_0 : V \times V \to \mathbb{F} $
la forma bilineal simétrica dada por $ 2 B_0 = B $
(luego $ Q(v) = B_0(v,v) $ para todo $ v \in V $).
Para todo $ v \in V $,
consideramos los endomorfismos $L_v$ y $\delta_v$ de $\Lambda(V)$,
donde $L_v$ es la multiplicación a izquierda por $v$,
y $\delta_v$ es la antiderivación correspondiente
a la forma lineal $B_0(v,\cdot)$.
Es decir,
$\delta_v$ está definida por $ \delta_v(1) = 0 $ y la siguiente ecuación:
\begin{equation*}
\delta_v ( u_1 \wedge \dots \wedge u_k )
= \sum_{i=1}^{k} (-1)^{i-1} B_0(v,u_i)
u_1 \wedge \dots \wedge u_{i-1}
\wedge u_{i+1} \wedge \dots \wedge u_k
\end{equation*}
Para todo $ u \in V $ y todo $ x \in \Lambda(V) $ tenemos que
$ \delta_v ( u \wedge x ) = B_0(v,u) x - u \wedge \delta_v(x) $.
De esta fórmula deducimos que $ \delta_v^2 = 0 $
y que $ L_u \delta_v + \delta_v L_u = B_0(v,u) 1 $.
Por lo tanto $ ( L_v + \delta_v )^2 = Q(v)1 $,
y hemos construido nuestra $\mathcal{A}$.
Esto concluye la demostración de que $\Cl(V,Q)$ es un álgebra unitaria.

Por la definición de álgebra de Clifford,
en $\Cl(V,Q)$ tenemos que,
para todo $ v \in V $,
se cumple la siguiente fórmula:
\begin{equation}\label{ec:Cl_cuadr}
v^2 = Q(v)
\end{equation}
Por lo tanto $ uv + vu = B(u,v) $ para todos $ u,v \in V $.
Sea $ v_1 , \dots , v_N $ una base ortogonal de $V$,
para todos $ i \neq j $ tenemos la siguiente ecuación:
\begin{equation}\label{ec:Cl_anticonm}
v_i v_j = - v_j v_i
\end{equation}
Si además tenemos que $ Q(v_1) = \dots = Q(v_N) = 1 $,
denotamos el álgebra de Clifford $\Cl_N(\mathbb{F})$.
En el caso $ \mathbb{F} = \mathbb{R} $,
si $ Q(v_1) = \dots = Q(v_p) = +1 $ y
$ Q(v_{p+1}) = \dots = Q(v_N) = -1 $,
denotamos el álgebra de Clifford $\Cl_{p,q}(\mathbb{R})$,
donde $ p+q = N $.

Llamemos $ v_I = v_{i_1} \cdots v_{i_r} $,
donde $ 1 \leq i_1 < \dots < i_r \leq N $
e $ I = \{ i_1 , \allowbreak \dots , \allowbreak i_r \} $,
y denotemos también $ v_{\emptyset} = 1 $.
Debido a las ecuaciones \eqref{ec:Cl_cuadr} y \eqref{ec:Cl_anticonm},
$\Cl(V,Q)$ está linealmente generada por
$ \{ v_I \}_{ I \subseteq \{ 1,\dots,N \} } $.
Luego la dimensión de $\Cl(V,Q)$ es como mucho $2^N$.
De hecho $ \dim \Cl(V,Q) = 2^N $,
pero la demostración,
al igual que la de que $\Cl(V,Q)$ es un álgebra unitaria,
no es trivial,
ver por ejemplo \cite[II.1.2]{Che1954} o \cite[teorema V.1.8]{Lam2005}.
La siguiente graduación nos va a permitir dar
una demostración alternativa bastante sencilla.

\begin{ejemplo}\label{ejem:Cl_grad}
Podemos dotar al álgebra de Clifford $\Cl(V,Q)$
de una graduación por el grupo abeliano $\mathbb{Z}_2^N$,
como en \cite[proposición 2.2]{AM2002}.
Primero definimos una $\mathbb{Z}_2^N$-graduación
en el álgebra tensorial $T(V)$
mediante la siguiente asignación:
\begin{equation}
\deg v_i = ( \bar{0},\dots,\bar{0}, \overbrace{\bar{1}}^i
,\bar{0},\dots,\bar{0} ) \in \mathbb{Z}_2^N
\end{equation}
El ideal $I(Q)$ está generado por los elementos de la forma
$ v_i \otimes v_j + v_j \otimes v_i - B(v_i,v_j) $,
que son homogéneos porque nuestra base $ v_1 , \dots , v_N $ es ortogonal,
luego es un ideal graduado.
Por lo tanto $\Cl(V,Q)$ hereda la $\mathbb{Z}_2^N$-graduación de $T(V)$
mediante $ \Cl(V,Q)_g = T(V)_g + I(Q) $.
\end{ejemplo}

\begin{proposicion}
La dimensión de un álgebra de Clifford $\Cl(V,Q)$ es $2^N$,
donde $N$ es la dimensión del espacio vectorial $V$:
\begin{equation}
\dim \Cl(V,Q) = 2^N
\end{equation}
\end{proposicion}

\begin{proof}
Como los $v_I$ pertenecen a distintas componentes homogéneas,
son linealmente independientes siempre y cuando sean distintos de cero.
Supongamos primero que la forma cuadrática $Q$ es no degenerada,
es decir,
que $ Q(v_i) \neq 0 $ para todo $ 1 \leq i \leq N $.
Entonces $ v_I^2 = v_{i_1} \cdots v_{i_r} v_{i_1} \cdots v_{i_r}
= \pm Q(v_{i_1}) \cdots Q(v_{i_r}) $.
Este último término es distinto de cero,
porque ya hemos visto que $ 1 \neq 0 $ en $\Cl(V,Q)$,
por lo tanto $ v_I \neq 0 $ para todo
$ I \subseteq \{ 1,\dots,N \} $.

Por otro lado,
permitamos ahora que algunos de los $Q(v_i)$ sean cero.
Consideramos la $\mathbb{F}$-álgebra $\mathcal{A}$:
\begin{equation*}
\mathcal{A} = \mathbb{F}[X_1] / (X_1^2-Q(v_1)) \otimes \dots \otimes
\mathbb{F}[X_N] / (X_N^2-Q(v_N)) \otimes \Cl_N(\mathbb{F})
\end{equation*}
Notemos que la forma bilineal $B'$ de $\Cl_N(\mathbb{F})$
es distinta de la forma bilineal $B$ de $\Cl(V,Q)$.
Sea $ u_1 , \dots , u_N $ una base ortonormal de $(V,B')$;
ya hemos demostrado que
$ \{ u_I \}_{ I \subseteq \{ 1,\dots,N \} } $
es una base lineal de $\Cl_N(\mathbb{F})$.
Sea $ f : V \to \mathcal{A} $ la aplicación lineal
que envía $v_i$ al siguiente elemento $w_i$:
\begin{equation*}
w_i = 1 \otimes \dots \otimes 1 \otimes \overbrace{X_i}^i
\otimes 1 \otimes \dots \otimes 1 \otimes u_i \in \mathcal{A}
\end{equation*}
Tenemos que $ w_i w_j + w_j w_i = B(v_i,v_j) 1 $,
luego podemos aplicar la propiedad universal para obtener
un homomorfismo de $\mathbb{F}$-álgebras unitarias
de $\Cl(V,Q)$ a $\mathcal{A}$ que envía $v_i$ a $w_i$.
Concluimos que $ v_I = v_{i_1} \cdots v_{i_r} $ no puede ser cero,
porque su imagen $ w_{i_1} \cdots w_{i_r} $ es distinta de cero,
y por consiguiente $ \dim \Cl(V,Q) = 2^N $.
\end{proof}

Notemos que si la forma cuadrática $Q$ es no degenerada,
entonces la $\mathbb{Z}_2^N$-graduación en $\Cl(V,Q)$
del ejemplo \ref{ejem:Cl_grad}
es de división y sus componentes homogéneas tienen dimensión $1$.

\begin{proposicion}
La dimensión del centro de un álgebra de Clifford
$ \Cl ( V , \allowbreak Q ) $
viene dada por la siguiente ecuación,
donde $N$ es la dimensión del espacio vectorial $V$:
\begin{equation}
\dim Z( \Cl(V,Q) ) =
\begin{cases}
	1 & \text{si } N \text{ es par} \\
	2 & \text{si } N \text{ es impar}
\end{cases}
\end{equation}
\end{proposicion}

\begin{proof}
La graduación del ejemplo \ref{ejem:Cl_grad} es útil de nuevo,
ya que nos permite saltarnos un paso en el cálculo.
En efecto,
$Z(\Cl(V,Q))$ es un subespacio graduado porque
el grupo graduador $\mathbb{Z}_2^N$ es abeliano
($ x = \sum_{g \in G} x_g \in Z(\mathcal{D}) $
si y solo si $ x y_h = y_h x $
para todo $ h \in G $ y todo $ y_h \in \mathcal{D}_h $;
si y solo si $ x_g y_h = y_h x_g $
para todos $ g,h \in G $ y todo $ y_h \in \mathcal{D}_h $;
si y solo si $ x_g \in Z(\mathcal{D}) $
para todo $ g \in G $),
por lo tanto solo tenemos que comprobar qué $v_I$ están en el centro.
Si $v_i$ es parte del producto
$ v_I = v_{i_1} \cdots v_{i_r} $ pero $v_j$ no,
entonces uno de ellos conmuta con $v_I$ y el otro anticonmuta.
Luego los únicos candidatos a pertenecer al centro
son $1$ y $ v_1 \cdots v_N $.
Este último elemento es central si y solo si $N$ es impar.
\end{proof}

\section{Invariante de Arf}

A partir de ahora nos centramos en el caso $ \mathbb{F} = \mathbb{R} $.
Dado un número real $x$ distinto de cero,
denotamos mediante $\sign(x)$ su signo,
que puede tomar los valores $+1$ o $-1$.
Además,
definimos $ \sign(0) = 0 $.
Recordemos que
el invariante de Arf de una aplicación
$ \mu : T \to \{ \pm 1 \} $
definida en un conjunto finito $T$
es el valor que toma más a menudo dicha aplicación:
$ \Arf(\mu) = \sign (
\vert \mu^{-1} (+1) \vert -
\vert \mu^{-1} (-1) \vert
) $.
Una de las consecuencias de las clasificaciones
de los capítulos \ref{cap:art_1} y \ref{cap:art_2}
es el siguiente teorema.

\begin{teorema}\label{teor:Cl_Arf_isom}
Sea $\mathcal{D}$ un álgebra asociativa real y de dimensión finita
cuyo centro $Z(\mathcal{D})$ tiene dimensión $1$ o $2$.
Supongamos que podemos dotar a $\mathcal{D}$
de una graduación de división
con componentes homogéneas de dimensión $1$
y cuyo soporte $T$ es un grupo isomorfo a $\mathbb{Z}_2^N$.
Entonces la ecuación \eqref{ec:sign_cuadr}
define una forma cuadrática $ \mu : T \to \{ \pm 1 \} $
cuyo invariante de Arf
determina el álgebra real $\mathcal{D}$ salvo isomorfismo,
de acuerdo a la siguiente lista.
\begin{itemize}
	\item Si $ \dim Z (\mathcal{D}) = 1 $
	($ \Leftrightarrow N = 2m $),
	entonces:
	\begin{itemize}
		\item $ \Arf(\mu) = +1 $
		implica que
		$ \mathcal{D} \cong M_{2^m} (\mathbb{R}) $.
		\item $ \Arf(\mu) = -1 $
		implica que
		$ \mathcal{D} \cong M_{2^{m-1}} (\mathbb{H}) $.
	\end{itemize}
	\item Si $ \dim Z (\mathcal{D}) = 2 $
	($ \Leftrightarrow N = 2m+1 $),
	entonces:
	\begin{itemize}
		\item $ \Arf(\mu) = +1 $
		implica que
		$ \mathcal{D} \cong M_{2^m} (\mathbb{R})
		\times M_{2^m} (\mathbb{R}) $.
		\item $ \Arf(\mu) = 0 $
		implica que
		$ \mathcal{D} \cong M_{2^m} (\mathbb{C}) $.
		\item $ \Arf(\mu) = -1 $
		implica que
		$ \mathcal{D} \cong M_{2^{m-1}} (\mathbb{H})
		\times M_{2^{m-1}} (\mathbb{H}) $.
	\end{itemize}
\end{itemize}
\end{teorema}

La graduación del ejemplo \ref{ejem:Cl_grad}
nos permite aplicar el teorema \ref{teor:Cl_Arf_isom}
al álgebra de Clifford $\Cl_{p,q}(\mathbb{R})$.
Denotemos mediante $\mu_{p,q}$ a la correspondiente forma cuadrática.
El objetivo de esta sección es calcular
el invariante de Arf de $\mu_{p,q}$
en función de los valores de $p$ y $q$.

\begin{teorema}\label{teor:Cl_period}
Para todos $ p,q \in \mathbb{N} \cup \{0\} $
tenemos los siguientes isomorfismos de álgebras reales.
\begin{enumerate}
	\item $ \Cl_{p+1,q+1}(\mathbb{R})
	\cong \Cl_{p,q}(\mathbb{R})
	\otimes M_2(\mathbb{R}) $.
	\item $ \Cl_{p+2,q}(\mathbb{R})
	\cong \Cl_{q,p}(\mathbb{R})
	\otimes M_2(\mathbb{R}) $.
	\item $ \Cl_{p,q+2}(\mathbb{R})
	\cong \Cl_{q,p}(\mathbb{R})
	\otimes \mathbb{H} $.
\end{enumerate}
\end{teorema}

\begin{proof}
Por el teorema \ref{teor:Cl_Arf_isom},
el primer isomorfismo es equivalente a
$ \Arf(\mu_{p+1,q+1}) = \Arf(\mu_{p,q}) $.
Identificamos
$ \Cl_{p,q}(\mathbb{R}) \subseteq \Cl_{p+1,q+1}(\mathbb{R}) $
de manera que
$ v_1^2 = \dots = v_p^2 = +1 $,
$ v_{p+1}^2 = \dots = v_{p+q}^2 = -1 $,
$ v_{N+1}^2 = +1 $ y $ v_{N+2}^2 = -1 $,
donde $ N = p+q $.
Para hallar el invariante de Arf de $\mu_{p,q}$
calculamos $v_I^2$ para todo $ I \subseteq \{ 1,\dots,N \} $
y contamos la diferencia entre el número de $+1$ y de $-1$.
Análogamente para $\Arf(\mu_{p+1,q+1})$,
pero ahora tenemos cuatro tipos de términos
según $I$ recorre los subconjuntos de $ \{ 1,\dots,N \} $:
\begin{equation*}
\text{(1) los } v_I
\text{; (2) los } v_I v_{N+1}
\text{; (3) los } v_I v_{N+2}
\text{; (4) los } v_I v_{N+1} v_{N+2} \text{.}
\end{equation*}
Como $ ( v_I v_{N+1} )^2 = -( v_I v_{N+2} )^2 $,
los términos del segundo tipo se cancelan con
los términos del tercer tipo.
Además,
$ v_{N+1} v_{N+2} $ conmuta con $v_I$
y $ ( v_{N+1} v_{N+2} )^2 = +1 $,
luego $ ( v_I v_{N+1} v_{N+2} )^2 = v_I^2 $,
y la contribución de los términos del cuarto tipo
es la misma que la de los términos del primer tipo,
que viene dada por $\Arf(\mu_{p,q})$.

Para el segundo isomorfismo,
sean $ u_1 , \dots , u_N $ y
$ v_1 , \dots , v_N , v_{N+1} , v_{N+2} $
sistemas generadores de
$\Cl_{q,p}(\mathbb{R})$ y $\Cl_{p+2,q}(\mathbb{R})$
satisfaciendo la ecuación \eqref{ec:Cl_anticonm}
y tales que $ v_{N+1}^2 = v_{N+2}^2 = +1 $
y $ u_i^2 = - v_i^2 $ para todo $ 1 \leq i \leq N $.
Tenemos los mismos cuatro tipos de términos
en $\Cl_{p+2,q}(\mathbb{R})$ que antes.
Esta vez $ ( v_{N+1} v_{N+2} )^2 = -1 $,
luego los términos del cuarto tipo
se cancelan con los términos del primer tipo.
Además,
$ ( v_I v_{N+1} )^2 = ( v_I v_{N+2} )^2 $,
por lo que la contribución de los términos del tercer tipo
es la misma que la de los términos del segundo tipo.
Si $ \vert I \vert $ es impar,
entonces $ ( v_I v_{N+1} )^2 = - v_I^2 v_{N+1}^2 = u_I^2 $;
mientras que si $ \vert I \vert $ es par,
también $ ( v_I v_{N+1} )^2 = v_I^2 v_{N+1}^2 = u_I^2 $.
Concluimos que $ \Arf(\mu_{p+2,q}) = \Arf(\mu_{q,p}) $,
lo que implica que $ \Cl_{p+2,q}(\mathbb{R})
\cong \Cl_{q,p}(\mathbb{R}) \otimes M_2(\mathbb{R}) $.

La demostración del tercer isomorfismo es análoga a la del segundo,
pero en este caso $ v_{N+1}^2 = v_{N+2}^2 = -1 $,
por lo que obtenemos que
$ \Arf(\mu_{p,q+2}) = - \Arf(\mu_{q,p}) $.
\end{proof}

\begin{lema}\label{lem:Cl_Arf}
$ \Arf(\mu_{p,0}) =
\sign (
\cos ( p \pi / 4 ) + \allowbreak
\sin ( p \pi / 4 )
) $
para todo $ p \in \mathbb{N} \cup \{0\} $,
y
$ \Arf(\mu_{0,q}) =
\sign (
\cos ( - q \pi / 4 ) + \allowbreak
\sin ( - q \pi / 4 )
) $
para todo $ q \in \mathbb{N} \cup \{0\} $.
\end{lema}

\begin{proof}
Empecemos con el caso en el que $p=N$ y $q=0$.
Como $ v_1^2 = \dots = v_N^2 = +1 $,
el valor de $v_I^2$ solo depende de $ \vert I \vert = r $.
Específicamente,
$ v_I^2 = v_{i_1} \cdots v_{i_r} v_{i_1} \cdots v_{i_r}
= (-1)^{\binom{r}{2}} $.
Así,
podemos contar la diferencia entre el número de $+1$ y de $-1$:
\begin{equation*}
\vert \mu_{N,0}^{-1} (+1) \vert -
\vert \mu_{N,0}^{-1} (-1) \vert
= \sum_{r=0}^{N} \binom{N}{r} (-1)^{(r-1)r/2}
= S_0 + S_1 - S_2 - S_3
\end{equation*}
En el último paso hemos abreviado
la escritura de las sumas binomiales
con la notación del lema \ref{lem:sum_bin}
($S_0$, $S_1$, $S_2$ y $S_3$).
Precisamente aplicando dicho lema
obtenemos la fórmula deseada para $\Arf(\mu_{p,0})$.

El caso en el que $p=0$ y $q=N$ es análogo,
pero ahora $ v_I^2 = (-1)^{\binom{r}{2}} (-1)^r $.
Luego obtenemos la siguiente ecuación:
\begin{equation*}
\vert \mu_{0,N}^{-1} (+1) \vert -
\vert \mu_{0,N}^{-1} (-1) \vert
= \sum_{r=0}^{N} \binom{N}{r} (-1)^{(r-1)r/2+r}
= S_0 - S_1 - S_2 + S_3
\end{equation*}
\end{proof}

Calculemos,
mediante un argumento algebraico,
las sumas binomiales $S_0$, $S_1$, $S_2$ y $S_3$
que hemos usado en la demostración del lema \ref{lem:Cl_Arf}.

\begin{lema}\label{lem:sum_bin}
Para todo número entero $N$ mayor o igual que $1$
se satisfacen las siguientes fórmulas binomiales:
\begin{equation*}
S_0 := \binom{N}{0} + \binom{N}{4} + \dots
+ \binom{N}{ 4 \lfloor \frac{N-0}{4} \rfloor + 0 }
= \frac{1}{2} \left( 2^{N-1} + 2^{N/2}
\cos \frac{N\pi}{4} \right)
\end{equation*}
\begin{equation*}
S_1 := \binom{N}{1} + \binom{N}{5} + \dots
+ \binom{N}{ 4 \lfloor \frac{N-1}{4} \rfloor + 1 }
= \frac{1}{2} \left( 2^{N-1} + 2^{N/2}
\sin \frac{N\pi}{4} \right)
\end{equation*}
\begin{equation*}
S_2 := \binom{N}{2} + \binom{N}{6} + \dots
+ \binom{N}{ 4 \lfloor \frac{N-2}{4} \rfloor + 2 }
= \frac{1}{2} \left( 2^{N-1} - 2^{N/2}
\cos \frac{N\pi}{4} \right)
\end{equation*}
\begin{equation*}
S_3 := \binom{N}{3} + \binom{N}{7} + \dots
+ \binom{N}{ 4 \lfloor \frac{N-3}{4} \rfloor + 3 }
= \frac{1}{2} \left( 2^{N-1} - 2^{N/2}
\sin \frac{N\pi}{4} \right)
\end{equation*}
\end{lema}

\begin{proof}
Consideramos el siguiente isomorfismo $\varphi$
de $\mathbb{R}$-álgebras
dado por el teorema chino del resto:
\begin{align*}
\varphi : \mathbb{R}[T] / (T^4-1)
& \longrightarrow \mathbb{R}[X] / (X-1)
\times \mathbb{R}[Y] / (Y+1)
\times \mathbb{R}[Z] / (Z^2+1)
\\ f(T)+(T^4-1)
& \longmapsto ( f(X)+(X-1)
, f(Y)+(Y+1) , f(Z)+(Z^2+1) )
\end{align*}
Si escribimos $ i = Z+(Z^2+1) $
y $ t = T+(T^4-1) $,
entonces $\varphi$ está definido por
$ \varphi(1) = (1,1,1) $,
$ \varphi(t) = (1,-1,i) $,
$ \varphi(t^2) = (1,1,-1) $ y
$ \varphi(t^3) = (1,-1,-i) $.
Recíprocamente,
$\varphi^{-1}$ está determinado por
$ \varphi^{-1}(1,0,0) = (1+t+t^2+t^3)/4 $,
$ \varphi^{-1}(0,1,0) = (1-t+t^2-t^3)/4 $,
$ \varphi^{-1}(0,0,1) = (1-t^2)/2 $ y
$ \varphi^{-1}(0,0,i) = (t-t^3)/2 $.

Por un lado,
podemos calcular $(1+t)^N$ en $\mathbb{R}[t]$:
\begin{equation*}
(1+t)^N = \sum_{r=0}^N \binom{N}{r} t^r
= S_0 1 + S_1 t + S_2 t^2 + S_3 t^3
\end{equation*}
Por otro lado,
podemos aplicar primero $\varphi$ a $(1+t)$,
que da $(2,0,1+i)$,
y elevar el resultado a la $N$-ésima potencia:
\begin{equation*}
(2,0,1+i)^N = (2^N,0,(1+i)^N) = \left( 2^N,0,
2^{N/2} \cos \frac{N\pi}{4} +
2^{N/2} \sin \frac{N\pi}{4} i \right)
\end{equation*}
Aplicando $\varphi^{-1}$ a esta última expresión,
y comparándola con la otra que tenemos,
obtenemos las fórmulas deseadas.
\end{proof}

Juntando el teorema \ref{teor:Cl_period} y el lema \ref{lem:Cl_Arf}
obtenemos el teorema final de este capítulo.
Por supuesto,
la ecuación \eqref{ec:Cl_Arf_pq_mod} es bien conocida,
aunque ni mi director ni yo la habíamos visto escrita antes
en la forma de la ecuación \eqref{ec:Cl_Arf_pq_trigon}.

\begin{teorema}
Para todos $ p,q \in \mathbb{N} \cup \{0\} $
el invariante de Arf de $\mu_{p,q}$
viene dado por las siguientes fórmulas:
\begin{align}
\Arf(\mu_{p,q}) & {} =
\sign \left( \cos \frac{ (p-q) \pi }{4}
+ \sin \frac{ (p-q) \pi }{4} \right)
\label{ec:Cl_Arf_pq_trigon}
\\ & {} =
\begin{cases}
	1	& \text{si } p-q+1 \equiv 1,2,3	\pmod{8} \\
	0	& \text{si } p-q+1 \equiv 0,4	\pmod{8} \\
	-1	& \text{si } p-q+1 \equiv 5,6,7	\pmod{8}
\end{cases}
\label{ec:Cl_Arf_pq_mod}
\end{align}
\end{teorema}

\selectlanguage{spanish}
\chapter*{Conclusiones}
\addcontentsline{toc}{chapter}{Conclusiones}

Estos han sido los principales resultados de mi tesis.
Primero,
la clasificación salvo isomorfismo y equivalencia
de las graduaciones de división
en las álgebras asociativas reales simples.
Recordemos que este resultado completaba
la clasificación salvo isomorfismo
de las graduaciones (no necesariamente de división)
en dichas álgebras.
Segundo,
la clasificación salvo isomorfismo y equivalencia
de las involuciones
en las álgebras asociativas reales simples graduadas de división.
Y tercero,
la clasificación salvo isomorfismo
de las graduaciones
en las álgebras de Lie reales simples centrales clásicas,
excepto en las de tipo $D_4$.

Creo que he sido afortunado en el doctorado,
ya que estas clasificaciones han sido abordables.
El número de casos que aparecen es lo suficientemente grande
para que los problemas sean interesantes,
y lo suficientemente pequeño
para que se puedan resolver en una tesis.
Además hemos obtenido algunos resultados bonitos,
como el teorema \ref{teor:engros},
la $ \mathbb{Z}_2 \times \mathbb{Z}_4 $-graduación de división
en el álgebra $ M_2(\mathbb{R}) \times \mathbb{H} $
del ejemplo \ref{ejem:grad_semisimple},
o la ecuación \eqref{ec:Cl_Arf_pq_trigon}.
También hemos podido aplicar esta teoría de graduaciones
a las álgebras de Clifford.

Por supuesto,
no hemos resuelto todas las cuestiones que
la clasificación de graduaciones en álgebras de Lie reales simples
plantea.
Nuestro modelo de reducir
la clasificación de graduaciones en álgebras de Lie a
la clasificación de graduaciones en álgebras asociativas con involución
cubre una buena parte de los casos cuando el álgebra de Lie es de tipo $D_4$,
pero no todos.
El resto han sido recientemente analizados en \cite{EK2018arx},
por lo que la clasificación en el tipo $D_4$ ya está completa.

El trabajo futuro se divide de manera natural en tres vías.
Uno,
clasificar las graduaciones finas salvo equivalencia
tanto en las álgebras asociativas reales simples
como en las álgebras de Lie reales simples centrales clásicas,
y calcular sus respectivos grupos de Weyl.
Dos,
clasificar las graduaciones en las álgebras de Lie complejas simples,
vistas como álgebras sobre el cuerpo de los números reales.
Y tres,
clasificar las graduaciones en las álgebras de Lie reales simples excepcionales.
Notemos que algunos casos de este tercer punto ya han sido estudiados en
\cite{CDM2010} y \cite{DG2016}.

\selectlanguage{english}
\clearpage
\phantomsection

\end{document}